\documentclass[11pt,a4paper]{amsart}
\usepackage[utf8]{inputenc}
\usepackage[english]{babel}
\usepackage{amsmath}
\usepackage{amsfonts}
\usepackage{amssymb}
\usepackage{graphicx, overpic}
\usepackage{amsthm}
\usepackage{textcomp}
\usepackage{mathrsfs}
\usepackage{hyperref}
\usepackage{enumitem}
\usepackage{microtype}

\usepackage[margin=1.25in]{geometry}

\theoremstyle{plain}
\newtheorem{thm}{Theorem}[section]
\newtheorem{lem}[thm]{Lemma}
\newtheorem{prop}[thm]{Proposition}
\newtheorem{cor}[thm]{Corollary}
\newcounter{theoremalph}

\newtheorem{thmAlph}[theoremalph]{Theorem}
\newtheorem{corAlph}[theoremalph]{Corollary}

\theoremstyle{definition}
\newtheorem{defn}[thm]{Definition}
\newtheorem{example}[thm]{Example}
\newtheorem{rem}[thm]{Remark}
\newtheorem{conj}[thm]{Conjecture}

\newtheorem{assumption}[thm]{Assumption}
\newtheorem{notation}[thm]{Notation}
\newtheorem{question}[thm]{Question}

\newcommand{\Z}{\mathbb{Z}}
\newcommand{\R}{\mathbb{R}}
\newcommand{\N}{\mathbb{N}}

\renewcommand{\H}{\mathbb{H}}
\newcommand{\Hy}{\mathbb{H}}

\newcommand{\C}{\mathbb{C}}

\newcommand{\cC}{\mathcal{C}}

\newcommand{\cG}{\mathcal{G}}

\newcommand{\cJ}{\mathcal{J}}

\newcommand{\cS}{\mathcal{S}}

\newcommand{\bZ}{\mathbb{Z}}
\newcommand{\bN}{\mathbb{N}}
\newcommand{\bH}{\mathbb{H}}
\newcommand{\bR}{\mathbb{R}}
\newcommand{\E}{\mathbb{E}}

\newcommand{\gL}{\Lambda}
\newcommand{\G}{\Gamma}

\newcommand{\p}{\partial}
\newcommand{\hX}{\hat{X}}

\newcommand{\hQ}{\hat{Q}}
\newcommand{\la}{\left\langle}
\newcommand{\ra}{\right\rangle}

\makeatletter
\newsavebox{\@brx}
\newcommand{\llangle}[1][]{\savebox{\@brx}{\(\m@th{#1\langle}\)}%
  \mathopen{\copy\@brx\kern-0.5\wd\@brx\usebox{\@brx}}}
\newcommand{\rrangle}[1][]{\savebox{\@brx}{\(\m@th{#1\rangle}\)}%
  \mathclose{\copy\@brx\kern-0.5\wd\@brx\usebox{\@brx}}}
\makeatother

\usepackage{mathtools}

\DeclareMathOperator{\Aut}{Aut}
\DeclareMathOperator{\cd}{cd}

\DeclareMathOperator{\Isom}{Isom}
\DeclareMathOperator{\Homeo}{Homeo}
\DeclareMathOperator{\id}{Id}
\DeclareMathOperator{\QI}{QI}
\DeclareMathOperator{\lk}{lk}
\DeclareMathOperator{\PB}{PB}

\DeclareMathOperator{\stab}{Stab}
\DeclareMathOperator{\Stab}{Stab}
\DeclareMathOperator{\Star}{Star}

\DeclareMathOperator{\Out}{Out}
\DeclareMathOperator{\Hig}{Hig}

\DeclareMathOperator{\PSL}{PSL}
\DeclareMathOperator{\SL}{SL}

\DeclareMathOperator{\Comm}{Comm}
\DeclareMathOperator{\WCH}{WCH}
\DeclareMathOperator{\Sym}{Sym}
\DeclareMathOperator{\Spin}{Spin}
\DeclareMathOperator{\SO}{SO}

\newcommand{\Haus}{{\text{Haus}}}

\usepackage{tikz-cd} %
\tikzset{
	labl/.style={anchor=south, rotate=90, inner sep=.5mm}
}

\definecolor{amethyst}{rgb}{0.6, 0.4, 0.8}

\newcommand{\hide}[1]{}

\title{Graphically discrete groups and rigidity}

\author{Alex Margolis}
\address{Alex Margolis, Department of Mathematics and Computer Science, Wesleyan University, Science Tower 655, 265 Church Street, Middletown, CT 06459, USA}
\email{amargolis@wesleyan.edu}

\author{Sam Shepherd}
\address{Sam Shepherd, Department of Mathematics, Vanderbilt University, 1326 Stevenson Center, Nashville, TN 37240, USA}
\email{samuel.shepherd@vanderbilt.edu}

\author{Emily Stark}
\address{Emily Stark, Department of Mathematics and Computer Science, Wesleyan University, Science Tower 655, 265 Church Street, Middletown, CT 06459, USA}
\email{estark@wesleyan.edu}

\author{Daniel Woodhouse}
\address{Daniel Woodhouse}
\email{mr.daniel.woodhouse@gmail.com}

\date{\today}

\begin{document}

  \begin{abstract}
    We introduce the notion of graphical discreteness to group theory. A finitely generated group is graphically discrete if whenever it acts geometrically on a locally finite graph, the automorphism group of the graph is compact-by-discrete.
    Notable examples include finitely generated nilpotent groups, most lattices in semisimple Lie groups, and irreducible non-geometric 3-manifold groups.  We show graphs of groups with graphically discrete vertex groups frequently have strong rigidity properties. We prove free products of one-ended virtually torsion-free graphically discrete groups are action rigid within the class of virtually torsion-free groups. We also prove quasi-isometric rigidity for many hyperbolic graphs of groups whose vertex groups are closed hyperbolic manifold groups and whose edge groups are non-elementary quasi-convex subgroups. This includes the case of two hyperbolic 3-manifold groups amalgamated along a quasi-convex malnormal non-abelian free subgroup. We provide several additional examples of graphically discrete groups and illustrate this property is not a commensurability invariant.
  \end{abstract}

 \maketitle

\section{Introduction}

    We are interested in the question of quasi-isometric rigidity for finitely generated groups.
    There are many variants of the question, but we consider the strongest version: if two groups are quasi-isometric, are they virtually isomorphic?
    A general approach is to first prove that the two groups act geometrically (i.e. properly, cocompactly, and by isometries) on the same proper metric space. 
    This action would then serve as a suitable springboard for achieving virtual isomorphism.
    We refer to this implied second question as \emph{action rigidity}: if two groups act geometrically on the same proper metric space, are they virtually isomorphic?
        
    Supposing that $\Gamma$ and $\Gamma'$ are our groups acting geometrically on a proper metric space $X$, it then suffices to consider the images of $\Gamma$ and $\Gamma'$ in the locally compact group $\Isom(X)$, where, modulo finite normal subgroups, they are uniform lattices.
    There are two distinct cases in which commensurability might then be achieved. (See Sections~\ref{subsec:lattices} and \ref{subsec:prelim_action} for the various notions of commensurability and virtual isomorphism.)
    
    The first case is that $\Isom(X)$ is compact-by-discrete. In this case, the images of $\Gamma$ and $\Gamma'$ in the discrete quotient of $\Isom(X)$ are commensurable, and $\Gamma$ and $\Gamma'$ are virtually isomorphic.
    This paper approaches this possibility with the notion of \emph{graphical discreteness} (defined below and detailed in Section~\ref{sec:gd}).

    The second approach considers the exciting possibility that the automorphism group in question is so rich and full-bodied that such quotients and intersections are impossible.
    For example, in the hyperbolic setting, uniform lattices in the isometry group of a rank-1 symmetric space are not all abstractly commensurable, and quasi-isometric rigidity in this strong sense does not hold.
    But, we wonder whether symmetric spaces yield the exception for hyperbolic groups; the paper considers spaces with automorphism groups that are highly nontrivial but all of whose lattices are abstractly commensurable. 
    In this second case, abstract commensurability is proven via ``common covering'' theorems. Such arguments have been given when $X$ is a tree, a tree of hyperbolic $n$-spaces~\cite{starkwoodhouse2024action}, and a hyperbolic right-angled Coxeter group building~\cite{shepherd2024commensurability}.
    Quasi-isometric rigidity of uniform lattices in many Fuchsian buildings follows by combining results by Bourdon-Pajot and Xie~\cite{bourdonpajot,xie2006quasiisometric} with work of Haglund and Shepherd~\cite{haglund2006commensurability, shepherd2024commensurability}. Additionally, Huang proved quasi-isometric rigidity for certain right-angled Artin groups by showing that all uniform lattices in the  universal cover of the associated Salvetti complex are commensurable \cite[Theorem 1.5]{huang2018commensurability}.
    This paper develops powerful tools to prove common cover theorems in the graph of groups setting. 

    We not only employ both of these two approaches in this paper, but we aim to combine them in novel and revealing ways. 
    Our results are outlined below, but we first provide a sample quasi-isometric rigidity application to motivate what follows. 

     \begin{thmAlph}\label{thm:intro:common_model_geom}(Theorem \ref{thm:common_model_geom})
        Let $\Gamma$ be the fundamental group of a finite graph of groups $\cG$ satisfying the following properties:
        \begin{enumerate}
	       \item Each vertex group of $\cG$ is cubulated and is the fundamental group of a closed hyperbolic $n$-manifold for some $n\geq 3$.
	       \item For each vertex group of $\cG$, the collection  of images of incident edge maps is an almost malnormal family of  non-elementary quasi-convex subgroups with  cohomological codimension at least two.
        \end{enumerate}
        Then any finitely generated group quasi-isometric to $\Gamma$ is abstractly commensurable to~$\Gamma$.
    \end{thmAlph}

     With regards to the approaches discussed above, we use the facts that, first, fundamental groups of closed hyperbolic manifolds have only trivial totally disconnected lattice envelopes, and second, the cubical structure and conditions on the edge groups fit into our common cover framework. 
    The vertex groups above are automatically cubulated in the case $n=3$ \cite{kahnmarkovic2012immersing,bergeronwise2012boundary}, and the cohomological codimension  condition is satisfied if the  edge groups are all quasi-convex free groups (see also Corollary \ref{cor:QI3manifold}). Thus, Theorem \ref{thm:intro:common_model_geom} yields the following corollary:

\begin{corAlph}
    Let $\Gamma_1$ and $\Gamma_2$ be fundamental groups of closed hyperbolic 3-manifolds, and let $F_1\leq \Gamma_1$ and  $F_2\leq \Gamma_2$ be isomorphic quasi-convex malnormal non-abelian free subgroups. Then $\Gamma = \Gamma_1*_{F_1=F_2}\Gamma_2$ is quasi-isometrically rigid: any finitely generated group quasi-isometric to $\Gamma$ is abstractly commensurable to $\Gamma$. 
\end{corAlph}

    The result of Theorem~\ref{thm:intro:common_model_geom} should be compared to the work of Mosher--Sageev--Whyte \cite{moshersageevwhyte2011quasiactions} and Biswas~\cite[Corollary 1.3]{biswas12flows}, which proved that the {\it family} of groups that split as a graph of group with closed hyperbolic manifold vertex groups and suitable edge groups is quasi-isometrically rigid, while we obtain quasi-isometric rigidity of the {\it group}.  

    In this paper we introduce graphically discrete groups, which are well-suited for studying rigidity phenomena. %
	A finitely generated group $\Gamma$ is \emph{graphically discrete} if whenever it acts geometrically on a connected locally finite graph $X$, the automorphism group $\Aut(X)$ is compact-by-discrete, meaning $\Aut(X)$ contains a compact normal subgroup so that the quotient group is discrete. 
    In this case, the topological groups $\Aut(X)$ and $\Gamma$ are
    virtually isomorphic (in the sense of Definition \ref{defn:AC_VI}), and $\Gamma$ has only trivial totally disconnected uniform lattice embeddings modulo finite kernels.
    Intuitively, a finitely generated group $\Gamma$ is graphically discrete if it does not act on any graph that has ``too many more'' symmetries than those coming from~$\Gamma$. See Sections \ref{subsec:graphdiscrete} and \ref{subsec:compactbydiscrete} for other equivalent formulations of graphical discreteness and Section~\ref{sec:intrographs}, which motivates restricting to group actions on graphs. 
 
    Crucially, there are a large number of examples of groups that are graphically discrete; see Section \ref{sec:intro_graph_dis}. Moreover, if $\Gamma$ is graphically discrete and $\Gamma$ and $\Gamma'$ both act geometrically on the same connected locally finite graph, then $\Gamma$ and $\Gamma'$ are virtually isomorphic (Proposition~\ref{prop:gdimpliesvi}).
    While, a priori, asking for a common model geometry that is a graph is a strong requirement, in fact, recent work of Margolis~\cite[Theorem A]{margolis2024model} implies that graphically discrete groups are action rigid, with the exception of two families; see Theorem~\ref{thm:intro_alex_gdar}.

    We note the deep connections to existing programs of research within geometry and group theory.
    The question of graphical discreteness is a particular case of the {\it Lattice Envelope Problem}, which asks, given a finitely generated group $\Gamma$, classify the locally compact groups containing $\Gamma$ as a uniform lattice. 
    This question has roots in Mostow's Strong Rigidity~\cite{mostow_strongrigidity}, which classifies the Lie groups that contain the fundamental group of a closed hyperbolic manifold as a uniform lattice. This classification was extended to lattice embeddings into locally compact groups and for a wider class of finitely generated groups by Furman~\cite{furman2001mostow} and Bader--Furman--Sauer~\cite{bader2020lattice}. Dymarz~\cite{dymarz2015envelopes} classified lattice envelopes of certain solvable groups, and the classification of lattice envelopes of virtually free groups was given by Mosher--Sageev--Whyte~\cite{moshersageevwhyte_trees}. Important groups in geometric group theory have been shown to have only trivial lattice envelopes (see Definition~\ref{def:VIembeddings}), including non-exceptional mapping class groups~\cite{kida2010MCGs}, $\Out(F_n)$ for $n \geq 3$~\cite{guirardel2021measure}, and various $2$-dimensional Artin groups~\cite{horbez2020boundary}. 

    Quasi-isometric rigidity has been a driving focus of geometric group theory, leading to an abundance of new ideas with far-reaching consequences, including the relationship between the topological notion of ends and algebraic splittings~\cite{stallings,dunwoody1985accessibility}, the development of quasi-conformal geometry~\cite{tukia,gabai,cassonjungreis, Schwartz95, bourdonpajot}, and the structure of asymptotic cones~\cite{gromov81,vandendrieswilkie1984gromovs,KapovichLeeb97}. In this paper we approach quasi-isometric rigidity by focusing on rigidity of lattice envelopes. 
    
    
    Proving common covering theorems draws upon deep questions of the residual and subgroup separability properties of the group.
    Recent results, including Theorem~\ref{thm:intro:common_model_geom}, have built upon the foundations provided by the theory of special cube complexes and the resolution of the virtual Haken conjecture~\cite{Agol13, wise2021structure}.
    We refer the reader to the discussion in~\cite{starkwoodhouse2024action}.
    
    We provide further directions of study in Section~\ref{sec:open}, but first we draw the reader's attention to the question of hyperbolic groups more generally.
    It is an open question whether hyperbolic groups with Menger curve boundary are quasi-isometrically rigid, with the theorem for certain Fuchsian buildings being the outstanding positive result.
    We conjecture that generically - for random groups - hyperbolic groups with Menger curve boundary are graphically discrete, which would reduce the question of their quasi-isometric rigidity to providing common geometric actions.
    We also note that the failure of graphical discreteness would necessitate nontrivial lattice envelopes for such groups, which would in and of itself be a remarkable, possibly shocking, development.

    Many of our results are action rigidity for groups that are not quasi-isometrically rigid. We are not selling the reader short with such results, which delineate distinctions between what it means to be quasi-isometric and what it means to act geometrically on a metric space. 


	\subsection{Rigidity theorems}\label{sec:intro_rigidity}	
	The central rigidity results of this paper concern groups which  are not typically  graphically discrete, but decompose as a graph of groups whose vertex groups  are graphically discrete. Given a class $\cC$ of finitely generated groups, we say a group $\Gamma\in \cC$ is \emph{action rigid within $\cC$} if for all $\Gamma'\in \cC$ such that both $\Gamma$ and $\Gamma'$ act geometrically on the same proper metric space, then $\Gamma$ and $\Gamma'$ are virtually isomorphic. If $\cC$ is the family of all finitely generated groups, then $\Gamma$ is said to be {\it action rigid}. 

\begin{thmAlph} \label{thm:intro:freeprod} (Theorem~\ref{thm:actionrigidA}.)
    Let $\Gamma$ be a virtually torsion-free finitely presented infinite-ended group such that each one-ended vertex group in a Stallings--Dunwoody decomposition of $\Gamma$ is  graphically discrete. 
    Then $\Gamma$ is action rigid within the family of virtually torsion-free groups. 
\end{thmAlph}

    The virtually torsion-free hypotheses in the theorem above are necessary: in Example~\ref{example_notRF} we exhibit an amalgamated free product of graphically discrete groups over a finite subgroup that is not action rigid. Theorem~\ref{thm:actionrigidA} actually holds within the class of groups that \emph{command their finite subgroups}, a more general class of groups including both virtually torsion-free and residually finite groups.

    By adding an assumption to the one-ended vertex groups of $\Gamma$ (that these \emph{persistently command their finite subgroups}; see Definition~\ref{def:persistcommand}), we prove that $\Gamma$ is action rigid in the absolute sense; see Theorem~\ref{thm:actionrigidity}. As a consequence, we give the following classification of action rigidity for $3$-manifold groups.
    
    \begin{corAlph} (Theorem~\ref{thm:3manactionrigid}.)
        The fundamental group of a (possibly reducible) closed $3$-manifold  is action rigid if and only if the manifold does not admit $\Hy^3$ or {\bf Sol} geometry. 
    \end{corAlph}

	Work of Whyte and Papasoglu--Whyte demonstrates that infinite-ended groups possess a remarkable degree of quasi-isometric flexibility \cite{whyte1999amenability,papasogluwhyte2002quasiisometries}. For instance, within the class of  infinite-ended groups, Whyte produced the first examples to demonstrate the sign of the Euler characteristic and the ratios of $\ell_2$-Betti numbers, which are commensurability invariants, are not quasi-isometry invariants. In contrast, Theorem~\ref{thm:intro:freeprod} shows that while quasi-isometric rigidity does not typically hold for infinite-ended groups, the weaker notion of action rigidity frequently does. 

	 Theorem~\ref{thm:intro:freeprod}  generalizes previous work of  Stark--Woodhouse~\cite{starkwoodhouse2024action}, who proved the result when the one-ended vertex groups are residually finite and fundamental groups of closed hyperbolic manifolds. A starting point to this project was to understand the properties of closed hyperbolic manifolds that lead to these rigidity phenomena. Graphical discreteness turned out to be an apt condition, interesting in its own right, and satisfied in far greater generality.

	Like Theorem \ref{thm:intro:freeprod}, the next action rigidity theorem concerns graphs of groups with graphically discrete vertex groups; a key difference is that we now allow infinite edge groups. See Section \ref{sec:actionrigidityGoG} for the relevant definitions.
	 
	 \begin{thmAlph}\label{thm:intro:hyp_action_rigid_criterion}(Theorem \ref{thm:hyp_action_rigid_criterion})
Suppose $\Gamma$ is the fundamental group of a minimal finite graph of groups $\cG$ satisfying the following properties:
	\begin{enumerate}
		\item Every vertex group of $\cG$  is graphically discrete, hyperbolic  and cubulated.
		\item For each vertex group of $\cG$, the collection  of images of incident edge maps is an almost malnormal family of infinite quasi-convex subgroups.
		\item The associated Bass--Serre  tree of spaces $(X,T)$ is preserved by quasi-isometries.
	\end{enumerate}
	Then $\Gamma$ is action rigid.
\end{thmAlph}

    We provide the following strategy detailed in Section~\ref{sec:actiontoQIrigid} for upgrading action rigidity to quasi-isometric rigidity. We show that if $\Gamma$ is a finitely generated \emph{tame} group that is action rigid and has the \emph{uniform quasi-isometry property}, then $\Gamma$ is quasi-isometrically rigid. So, by specializing to the case of hyperbolic $n$-manifold vertex groups, and adding an assumption about cohomological dimension, we use work of Mosher--Sageev--Whyte \cite{moshersageevwhyte2011quasiactions} and Biswas \cite{biswas12flows} to upgrade the conclusion of Theorem \ref{thm:intro:hyp_action_rigid_criterion} to quasi-isometric rigidity as in Theorem~\ref{thm:intro:common_model_geom}. Note that if we weakened the hypotheses of Theorem \ref{thm:intro:common_model_geom} to allow two-ended edge groups, then $\Gamma$ would no longer possess the uniform quasi-isometry property and the above strategy would break down.

\subsection{Graphical discreteness}\label{sec:intro_graph_dis}

	A vast family of well-studied groups can be shown to be graphically discrete using the existing literature: virtually nilpotent groups~\cite{trofimov1984polynomial}, irreducible lattices in connected center-free real semisimple Lie groups without compact factors, other than non-uniform lattices in $\PSL_2(\bR)$  \cite{bader2020lattice}, uniform lattices in the isometry group of the $3$-dimensional geometry {\bf Sol}~\cite{dymarz2015envelopes}, mapping class groups of non-exceptional finite-type surfaces~\cite{kida2010MCGs}, $\Out(F_n)$ for $n \geq 3$~\cite{guirardel2021measure}, various $2$-dimensional Artin groups of hyperbolic type~\cite{horbez2020boundary}, hyperbolic surface-by-free groups~\cite{farbmosher2002surfacebyfree}, the Higman groups~\cite{horbez2025measure}, and topologically rigid hyperbolic groups~\cite{kapovichkleiner2000hyperbolic}. See Theorem~\ref{thm:propc_examples} for details. 
	The results above, together with work of Kleiner--Leeb~\cite{kleinerleeb2001symmetric} imply that the fundamental group of a closed geometric $3$-manifold is graphically discrete; see Theorem~\ref{thm:propc-3manifolds}.
	
	We show fundamental groups of closed irreducible non-geometric $3$-manifolds have only trivial  lattice envelopes, and are therefore graphically discrete. 

    \begin{thmAlph} (Theorem~\ref{thm:nongeo3mangd}.)
    		Let $M$ be a closed irreducible, oriented $3$-manifold with a nontrivial geometric decomposition. Let $G$ be a locally compact group containing $\pi_1(M)$ as a  lattice. Then $G$ contains a compact open normal subgroup $K$ such that $G/K$ is isomorphic to the fundamental group of a compact $3$-orbifold  covered by $M$. In particular, $\pi_1(M)$ has no nontrivial lattice embeddings, hence is graphically discrete.
    \end{thmAlph}

	Cannon's Conjecture states that a hyperbolic group with boundary a 2-sphere is virtually the fundamental group of a closed hyperbolic 3-manifold, hence must be graphically discrete as noted above.  The next result says that any potential counterexamples to Cannon's Conjecture must also  have this property.
	
	\begin{thm} (Theorems~\ref{thm:graphical_discrete_sphere_boundaries}, \ref{thm:graphical_discrete_sier_boundaries}.)
		A hyperbolic group with visual boundary homeomorphic to an $n$-sphere with $n \leq 3$ or to a Sierpinski carpet is graphically discrete. 
	\end{thm}

	The proof of the theorem above is an application of the resolution of the Hilbert--Smith Conjecture concerning which locally compact groups act continuously and faithfully on manifolds by homeomorphisms. The Hilbert--Smith Conjecture is open for $n>3$ and additional positive solutions would imply the result above holds for the same values of~$n$.
	
	Given a property of groups such as graphical discreteness, it is natural to ask how robust the property is: is it a quasi-isometry or virtual isomorphism invariant? We show graphical discreteness is neither via the following examples. 
	A manifold is a \emph{minimal element} in its commensurability class if it does not nontrivially cover any orbifold. We note that there exist closed hyperbolic 3-manifolds that are minimal elements in their commensurability class; see Remark \ref{rem:minelts}.
    
    \begin{thm} (Theorem~\ref{thm:no_fi_subgroup}.)
        Let $M$ and $M'$ be closed hyperbolic $3$-manifolds that are minimal elements in their commensurability classes. Then $\Gamma = \pi_1(M)* \pi_1(M')$ has no nontrivial lattice embeddings, but has a finite-index subgroup that does have nontrivial lattice embeddings. In particular, $\Gamma$ is graphically discrete but contains a finite-index subgroup that is not. 
    \end{thm}

    \begin{cor}  
      The properties of being graphically discrete and possessing no nontrivial lattice embeddings do not pass to finite-index subgroups. In particular, these properties are not invariant under virtual isomorphisms nor under quasi-isometries.
    \end{cor}
	
	We also provide additional  examples of groups that are not graphically discrete. A good source of such examples are graphs of groups that exhibit some sort of ``flipping symmetry''. We exploit this idea in the following setting.
  A \emph{simple surface amalgam} is a topological space constructed by taking a set of $n \geq 3$ compact surfaces $\Sigma_1,\ldots, \Sigma_n$ with negative Euler characteristic and a single boundary component, and identifying all the boundary components by a homeomorphism.
  For fundamental groups of simple surface amalgams, graphical discreteness is equivalent to having no nontrivial lattice embeddings (Theorem \ref{thm:ssa_char}). 
  The following result characterizes when the fundamental group of a simple surface amalgam has nontrivial lattice embeddings, and provides examples of one-ended groups for which graphical discreteness is not a quasi-isometry invariant.
 
 \begin{thm} \label{thm:intro:ssa_char} (Theorem \ref{thm:ssa_char})
  The fundamental group of a simple surface amalgam  has no nontrivial lattice embeddings if and only if  the surfaces in the amalgam have pairwise-distinct Euler characteristics.
 \end{thm}
 
 In Proposition \ref{prop:HNNnotGD} we give a general criterion to show an HNN extension is not graphically discrete. In particular, we use this to deduce a right-angled Artin group $A_\Gamma$ is graphically discrete if and only if $\Gamma$ is a complete graph.%

 If the automorphism group of a connected locally finite  graph  contains uniform lattices that are not virtually isomorphic, then these uniform lattices are neither graphically discrete nor action rigid. Results in \cite{BurgerMozes00,forester2024incommensurable} then imply that any product of free groups $F_m\times F_n$ and certain Baumslag--Solitar groups (including BS($k, kn$) for coprime $k$ and $n$) are neither graphically discrete nor action rigid.
 In addition, work of Francoeur \cite{francoeur2025bireversible} implies that infinite bireversible groups are not graphically discrete (see Theorem \ref{thm:bireversible}).

    We also note that graphically discrete groups were considered by de la Salle--Tessera, as groups in \emph{family $\cC$}~\cite[Section 6]{delasalletessera2019char}. They proved that Cayley graphs of  torsion-free graphically discrete groups are \emph{uniquely strongly local-to-global rigid}; see the reference for definitions.

\subsection{Methods of proof} 
We highlight tools and strategies developed in this paper that we hope and expect will serve as a blueprint for future rigidity results. 


The rough strategy that we employ to prove action rigidity is as follows. Suppose that $\G$ splits as the fundamental group of a graph of groups, and suppose $\Gamma$ and another finitely generated group $\G'$ act geometrically on the same proper  metric space $X$. We divide the argument into three steps:
\begin{enumerate}
    \item\label{item:GT} Show that the full isometry group $G=\Isom(X)$ acts continuously on a tree $T$ that corresponds (at least in some way) to the splitting of $\G$.
    \item\label{item:XtoY} Upgrade the metric space $X$ to a simply connected combinatorial cell complex~$Y$ that decomposes as a tree of spaces over $T$.
    \item\label{item:commoncover} Build a common finite cover for the quotient spaces $Y/\G$ and $Y/\G'$.
\end{enumerate}
Given the assumptions in Theorems \ref{thm:intro:freeprod} and \ref{thm:intro:hyp_action_rigid_criterion}, Step (\ref{item:GT}) follows relatively easily from results in the literature, so most of the work lies in Steps (\ref{item:XtoY}) and (\ref{item:commoncover}).

For Step (\ref{item:XtoY}), we provide a general construction of a common graphical model geometry with desirable properties. For example, we use the following proposition (or one of its more technical versions from Section \ref{sec:blowups}).
This result can be viewed as a variant of the Cayley--Abels graph for a tdlc group $G$ (see \cite{kronmoller08roughcayley}) %
that is compatible with the action of $G$ on a tree $T$.

\begin{prop}\label{prop:intro:blowup}
    Let $G$ be a locally compact group acting continuously, minimally and cocompactly on a tree $T$. 
    Suppose all vertex and edge stabilizers of $T$ are compactly generated and $G$ contains a compact open subgroup.
    Then there exists a locally finite connected graph $X$ admitting a geometric action of $G$, and a surjective $G$-equivariant map $p:X\to T$ such that the fibers of vertices in $T$ are connected subgraphs of $X$.
\end{prop}

Our main common covering result used for Step (\ref{item:commoncover}) is Theorem \ref{thm:commoncover_BV}. 
This theorem takes two compact graphs of spaces $X_1$ and $X_2$ with a common universal cover $\tilde{X}$, satisfying certain additional assumptions, and guarantees the existence of a common finite cover.
The assumptions involved are too technical to state here in full, but the rough idea for the most important assumptions is as follows.
First we assume that the  vertex spaces in $\tilde{X}$ have discrete automorphism groups. This allows us to construct common finite covers for pairs of vertex spaces in $X_1$ and $X_2$. In the context of Theorems \ref{thm:intro:freeprod} and \ref{thm:intro:hyp_action_rigid_criterion}, this assumption can be deduced  from the graphical discreteness of vertex groups. 

Then, to build a common finite cover of $X_1$ and $X_2$, we want to glue together these finite covers of vertex spaces along elevations of edge spaces.
However, we need an additional assumption in order to make these elevations of edge spaces match up.
The required assumption involves the notion of a group \emph{commanding} a finite collection of subgroups (see Definition \ref{defn:commands}). More precisely, we assume that each automorphism group of a vertex space in $\tilde{X}$ commands a collection of stabilizers of orbit representatives of incident edge spaces.
For Theorem~\ref{thm:intro:freeprod}, the commanding assumption derives from the hypothesis that  vertex groups are virtually torsion-free; while for Theorem \ref{thm:intro:hyp_action_rigid_criterion} it derives from the assumption that each vertex group of $\cG$ is hyperbolic and cubulated with quasi-convex incident edge groups.

To determine which groups are graphically discrete, we use a range of different arguments and employ many results from the literature. In particular, we make use of the following characterization of graphical discreteness.

\begin{thm}(Theorem \ref{thm:compactbydiscrete})\label{thm:intro:compactbydiscrete}
    A finitely generated group $\Gamma$ is graphically discrete if and only if for every geometric action of $\Gamma$ on a connected locally finite graph $X$ and for all vertices $x \in X$, the homomorphism  $(\Aut(X))_{x} \rightarrow \QI(X)$ from the stabilizer of $x$ to the quasi-isometry group of $X$ has finite image.
\end{thm}
An analogous statement holds for non-elementary hyperbolic groups, replacing $\QI(X)$ with $\Homeo(\p X)$; see Theorem \ref{thm:boundary_char}. %
In a similar vein, one can characterize graphical discreteness for one-ended hyperbolic groups with nontrivial JSJ decomposition using the action of $(\Aut(X))_{x}$ on the JSJ tree (Corollary \ref{cor:vertex_stab_jsj_alt}).
Another source of graphically discrete groups are those with discrete quasi-isometry groups; more precisely, any finitely generated tame group $\G$ for which the natural map $\G\to\mathrm{QI}(\G)$ has finite-index image is graphically discrete (see Theorem \ref{thm:propc_examples}(\ref{propc:QI}) and Theorem \ref{thm:discrete_qi}).
	
\subsection{Groups acting on graphs} \label{sec:intrographs}

The study of locally compact groups is often split into the Lie group setting and the totally disconnected locally compact group (tdlc) setting. Indeed, a locally compact group $G$ sits in a short exact sequence $1 \rightarrow G^0 \rightarrow G \rightarrow G/G^0 \rightarrow 1$, where the identity component~$G^0$ is an inverse limit of connected Lie groups and $G/G^0$ is a tdlc group; see~\cite{tao_hilbert}. This paper primarily concerns the tdlc setting; the automorphism group of a locally finite graph is a tdlc group and a partial converse holds as well~\cite{kronmoller08roughcayley}. Restricting to groups acting on graphs is particularly natural in light of recent work of  Margolis~\cite{margolis2024model}, who characterized which finitely generated groups can embed, modulo finite kernel, as uniform lattices into locally compact groups that are not compact-by-(totally disconnected). In particular:

	\begin{thm} \label{thm:intro_alex_gdar} \cite{margolis2024model}
		Let $\Gamma$  be a finitely generated group with no infinite commensurated  subgroup that is either a finite-rank free abelian group or a uniform lattice in a connected semisimple Lie group with finite center. Then, if $\Gamma$ and another finitely generated group $\Gamma'$ act geometrically on the same proper  metric space, then $\Gamma$ and $\Gamma'$ act geometrically on the same connected locally finite graph. In particular, if $\Gamma$ as above is graphically discrete, then $\Gamma$ is action rigid.
	\end{thm}

    Understanding the finitely generated groups with only trivial lattice envelopes is a question of considerable interest in the field. We note that a finitely generated group is graphically discrete exactly when the group has only trivial \emph{tdlc} lattice envelopes (Proposition~\ref{prop:propC}). The theorem above implies that the only potential sources of examples where these are different notions are finitely generated groups with infinite commensurated  subgroups that are either finite-rank free abelian or uniform lattices in connected semisimple Lie groups.

\subsection{Open questions} \label{sec:open}

 We make the following conjectures: 
 
 \begin{conj} \label{conj:product}
  If $\G$ and $\G'$ are graphically discrete, 
then $\G \times \G'$ is graphically discrete.
 \end{conj}
 In Theorem \ref{thm:dirprod_hyp}, we prove Conjecture~\ref{conj:product} holds when $\G$ and $\G'$ are non-elementary hyperbolic groups.
 Conjecture~\ref{conj:product} is open even in the special cases that $\G'$ is $\mathbb{Z}$, abelian, or nilpotent.
 
 Uniform lattices in isometry groups of thick hyperbolic buildings are not graphically discrete; see Corollary~\ref{cor:buildings}. These hyperbolic groups in dimension two have Menger curve visual boundary and the failure of graphical discreteness is realized in the automorphism group of the building. On the other hand, we conjecture that many hyperbolic groups with Menger curve boundary are graphically discrete.
  
  \begin{conj}
  Random groups in the few-relator model and Gromov's density model are graphically discrete asymptotically almost surely as the length of the chosen relators tends to infinity.
 \end{conj}
  
 \begin{conj}
  A one-ended, hyperbolic free-by-cyclic group $F_n \rtimes_{\phi} \Z$ with trivial JSJ decomposition is graphically discrete. 
 \end{conj}
 
	We show hyperbolic groups with boundary homeomorphic to an $n$-sphere for $n \leq 3$ are graphically discrete. While the bound on dimension is used in the proof, we conjecture:
  
 \begin{conj}
    Hyperbolic groups with boundary the $n$-sphere for $n \geq 4$ are graphically discrete. 
 \end{conj}  
    
    As explained above, quasi-isometric rigidity for tame groups is implied by the group being both graphically discrete and having the uniform quasi-isometry property, with the exception of the two families from Theorem~\ref{thm:intro_alex_gdar}. Thus, we are very interested in the following. 
    
    \begin{question}
        Does the uniform quasi-isometry property hold for a one-ended hyperbolic group that does not split over a virtually cyclic subgroup and is not a lattice in the isometry group of a rank-1 symmetric space?
    \end{question}

\subsection*{Acknowledgments} The authors are thankful for helpful discussions with Tullia Dymarz and with Genevieve Walsh, who explained Remark~\ref{rem:minelts}. We also thank Adrien Le Boudec for pointing out an error in an earlier version of the article. The authors thank the referee for valuable comments. The third author was supported by NSF Grant No. DMS-2204339.

  \tableofcontents

\section{Preliminaries}

 \subsection{Graphs and cell complexes}
 
 \begin{defn}
   A \emph{graph} $X$ is a set of \emph{vertices} $VX$ and \emph{edges} $EX$.
An edge $e$ is oriented with an \emph{initial vertex} $\iota(e)$ and a \emph{terminal vertex} $\tau(e)$.
Associated to $e$ is the edge $\bar{e}$ with reversed orientation: $\iota(e) = \tau(\bar{e})$ and $\tau(e) = \iota(\bar{e})$. We have $e\neq\bar{e}$ and $e=\bar{\bar{e}}$ for all $e\in EX$. 
We also consider graphs as geodesic spaces by making each edge isometric to the unit interval.
For $v$ a vertex in a graph, the \emph{link} of $v$ is defined to be the following set:
\[
 \lk(v) = \{ e \mid \tau(e) = v\}
\]

A \emph{tree} is a connected graph with no cycles. In a tree, or more generally in a simplicial graph, each edge is determined by its endpoints, so we will often write $e=(u,v)$ to mean $u=\iota(e)$ and $v=\tau(e)$.
 \end{defn}
 
 \begin{defn}
   A 2-dimensional \emph{cell complex} is a space obtained by attaching 2-cells to a graph, where the boundary of each cell is glued to a loop in the graph. We only deal with 2-dimensional cell complexes in this paper and refer to them as  cell complexes.
   A cell complex (or graph) is \emph{locally finite} if each vertex is incident to finitely many edges and 2-cells.
 \end{defn}
 \subsection{Coarse geometry}\label{sec:coarsegeom}
 \begin{defn}
 	Let $X$ and $Y$ be metric spaces. Let $K \geq 1$ and $A \geq 0$. A function $f:X\to Y$ is:
 	\begin{enumerate}
 		\item \emph{$(K,A)$-coarse Lipschitz} if for all $x,x'\in X$, $\ d_Y\bigl(f(x),f(x')\bigr)\leq Kd_X(x,x')+A$.
 		\item \emph{$(K,A)$-quasi-isometry} is a map $f:X\to Y$ such that:
 		\begin{enumerate}
 			\item for all $x,x'\in X$, \[\frac{1}{K}d_X(x,x')-A\leq d_Y\bigl(f(x),f(x')\bigr)\leq Kd_X(x,x')+A,\]
 			\item for all $y\in Y$, there is some $x\in X$ such that $d_Y(f(x),y)\leq A$.
 		\end{enumerate} 
 	\end{enumerate}
 A map is \emph{coarse Lipschitz} (resp. a \emph{quasi-isometry}) if it  is $(K,A)$-coarse Lipschitz (resp. a $(K,A)$-quasi-isometry) for some $K\geq 1$ and $A\geq 0$.
 A map $f:X\to Y$ is \emph{coarsely surjective} if for some $A$, $N_A(f(X))=Y$.
 
 Two quasi-isometries $f,g:X\to Y$ are \emph{$A$-close} if $\sup_{x\in X} d_Y(f(x),g(x))\leq A$, and are \emph{close} if they are $A$-close for some $A\geq 0$.
 \end{defn}
 
If $\Gamma$ is a finitely generated group, $\Gamma$ can be equipped with the word metric with respect to a finite generating set. This metric is well-defined up to quasi-isometry.

\begin{defn}
	The \emph{quasi-isometry group} $\QI(X)$ of a metric space $X$ consists of self-quasi-isometries of $X$ modulo the equivalence relation $f\sim g$ if $f$ and $g$ are close. An equivalence class in $\QI(X)$ containing a quasi-isometry $f$ is denoted $[f]$.  The group operation on $\QI(X)$ is given by composition $[f][g]=[f\circ g]$, which can readily be shown to be well-defined.
\end{defn}

\begin{notation}
	If $G$ is a group, then we use $1_G$ (or $1$ in the absence of ambiguity) to denote the identity element of $G$. 
\end{notation}

\begin{defn}
If $G$ is a group and $X$ is a metric space, then a \emph{$(K,A)$-quasi-action} of $G$ on $X$ is a collection of maps $\{f_g\}_{g\in G}$ such that
\begin{enumerate}
	\item for every $g\in G$, $f_g$ is a $(K,A)$-quasi-isometry from $X$ to $X$;
	\item for every $g,k\in G$, $f_{gk}$ is $A$-close to $f_g\circ f_k$;
	\item $f_1$ is $A$-close to the identity on $X$.
\end{enumerate}
A \emph{quasi-action} of $G$ on $X$ is a $(K,A)$-quasi-action of $G$ on $X$ for some $K\geq 1$ and $A\geq 0$.
A quasi-action $\{f_g\}_{g\in G}$ of $G$ on $X$ is said to be \emph{cobounded} if there exists a $B$ such that for all $x,x'\in X$, there is some $g\in G$ such that $d(f_g(x),x')\leq B$. Two quasi-actions $\{f_g\}_{g\in G}$ and $\{h_g\}_{g\in G}$  of $G$ on $X$ and $Y$ respectively, are \emph{quasi-conjugate} if there is a quasi-isometry $f:X\to Y$ and a constant $C$ such that \[d\bigl(f(f_g(x)),h_g(f(x))\bigr)\leq C\] for all $g\in G$ and $x\in X$.
\end{defn}

\begin{rem}
    A quasi-action $\{f_g\}_{g\in G}$ of $G$ on $X$ induces a homomorphism $G\to \QI(X)$ given by $g\mapsto [f_g]$. In particular, the natural isometric action of a finitely generated group $\Gamma$ on itself by left multiplication induces a map $\Gamma\to \QI(\Gamma)$.
\end{rem}

 \subsection{Locally compact groups and lattices} \label{subsec:lattices}
 
 By convention, topological groups are assumed to be Hausdorff. If $G$ is topological group and $H\leq G$ is a closed subgroup, $G/H$ will always be assumed to be equipped with the quotient topology.  The group $G/H$ is discrete if and only if $H$ is open. It follows that $G$ is a compact-by-discrete if and only if $G$ contains a compact open normal subgroup.
 
 If $X$ is a metric space, then $\Isom(X)$ denotes the isometry group of~$X$, equipped with the compact-open topology. If $X$ is proper, i.e.\ closed balls are compact, then $\Isom(X)$ is a second countable locally compact topological group; moreover,  the compact-open topology on $\Isom(X)$ coincides with the topology of pointwise convergence. If $X$ is a connected graph, then we identify its automorphism group $\Aut(X)$ with a subgroup of the isometry group of the graph, metrized so each edge has length one.

    \begin{defn}
      A \emph{lattice} of a locally compact group $G$ is a discrete subgroup $\Gamma\leq G$ so that $\Gamma$ acts on $G$ with a Borel fundamental domain of finite Haar measure. The lattice is \emph{uniform} if in addition $\Gamma\backslash G$ is compact. A \emph{(uniform) lattice embedding} into a locally compact group $G$ is a monomorphism $\Gamma\rightarrow G$ whose image is a (uniform) lattice in $G$. %
        A \emph{(uniform) lattice embedding modulo finite kernel} into a locally compact group $G$ is a homomorphism $\Gamma \rightarrow G$ with finite kernel and with image a (uniform) lattice in~$G$. 
        If there exists a (uniform) lattice embedding  modulo finite kernel $\Gamma \rightarrow G$, we say $\Gamma$ is a \emph{(uniform) lattice modulo finite kernel} in $G$.%
    \end{defn}

We make frequent use of the following lemma:
\begin{lem}[{\cite[\S 2.C]{capracemonod2012lattice}}]\label{lem:open_lattice}
	Let $\Gamma$ be a (uniform) lattice in a locally compact group $G$.  If $H\leq G$ is open, then $\Gamma\cap H$ is a (uniform) lattice in $H$.
\end{lem}

    The following terminology is due to Bader--Furman--Sauer~\cite{bader2020lattice}. We generalize the definition given in \cite{bader2020lattice} from lattice embeddings to lattice  embeddings modulo finite kernel.

    \begin{defn}[{\cite[Definition 3.1]{bader2020lattice}}] \label{def:VIembeddings}
        Let $\Gamma$ and $\Gamma'$ be countable groups. Two lattice embeddings modulo finite kernels $\rho:\Gamma\rightarrow G$ and $\rho':\Gamma'\rightarrow G'$ are \emph{virtually isomorphic} if there exist:
        \begin{enumerate}
            \item open finite-index subgroups $H\leq G$ and $H'\leq G'$;
            \item compact normal subgroups $K\vartriangleleft H$ and $K'\vartriangleleft H'$;
            \item a commutative square
            \[
            \begin{tikzcd}
                \rho^{-1}(H)/\rho^{-1}(K) \arrow[hookrightarrow]{r}{}	\arrow{d}{\cong} 	& H/K\arrow{d}{\cong}\\
                \rho'^{-1}(H')/\rho'^{-1}(K') \arrow[hookrightarrow]{r}{} 		& H'/K'
            \end{tikzcd}\]
            where the horizontal arrows are lattice embeddings induced by $\rho$ and $\rho'$ and the second vertical map is a topological isomorphism.
        \end{enumerate}
        A lattice embedding modulo finite kernel is \emph{trivial} if it is virtually isomorphic to the identity map $\id_\Gamma:\Gamma\rightarrow \Gamma$. 
    \end{defn}

\begin{rem}\label{rem:triv<->c-by-d}
	It follows from the preceding definition that a lattice embedding modulo finite kernel $\rho:\Gamma\to G$ is trivial if and only if either of the following equivalent conditions are satisfied:
	\begin{enumerate}
		\item  $G$ is compact-by-discrete, i.e.\ $G$ contains a compact normal subgroup $K$ such that the quotient $G/K$ is discrete.
		\item $G$ contains a compact open normal subgroup.
	\end{enumerate}
    In Remark~\ref{rem:lattice_embedding_virt_isom}, another characterization of trivial lattice embeddings is given in terms of virtual isomorphisms of topological groups. 
\end{rem}

    We refer the reader to \cite[Lemmas 3.3--3.5]{bader2020lattice} to verify that the preceding definition makes sense and that the relation of virtual isomorphism is an equivalence relation.

The following definitions are analogues of finite generation and finite presentation in the setting of locally compact groups.
\begin{defn}[{\cite[\S 7]{cornulierdlH2016metric}}]
	A locally compact group $G$ is \emph{compactly generated} if it admits a compact generating set. A locally compact group $G$ is \emph{compactly presented} if it admits a presentation $\langle S\mid R\rangle$ such that $S$ is compact and the relations $R$ are of bounded length as words in $S$.
\end{defn}

\begin{lem}[{\cite[Proposition 5.C.3 and Corollary 8.A.5]{cornulierdlH2016metric}}]\label{lem:finitevscompact}
	Suppose $\Gamma$ is a uniform lattice modulo finite kernel in a locally compact group $G$. Then:
	\begin{enumerate}
		\item $G$ is compactly generated if and only if $\Gamma$ is finitely generated.
		\item $G$ is compactly presented if and only if $\Gamma$ is finitely presented.
	\end{enumerate}
\end{lem}

A compact normal subgroup is \emph{maximal} if it is not a proper subgroup of any other compact normal subgroup.  Since the product of compact normal subgroups is also a compact normal subgroup, a maximal compact subgroup, if it exists, is unique and is the union of all compact normal subgroups.

The following remark allows us to restrict our attention to locally compact groups that are second countable in many parts of this paper.
\begin{rem}\label{rem:2ndcountable}
	 Suppose $\Gamma$ is finitely generated  and $\rho:\Gamma\to G$ is a (uniform) lattice embedding modulo finite kernel. Since $G$ is compactly generated by Lemma \ref{lem:finitevscompact}, $G$ contains a compact normal subgroup $K\vartriangleleft G$ such that $G/K$ is second countable \cite[Remark 2.B.7.2]{cornulierdlH2016metric}. Then $\Gamma\xrightarrow{\rho} G\to G/K$ is also (uniform) lattice embedding modulo finite kernel. Moreover, $G$  is compact-by-discrete if and only if $G/K$ is; see \cite[Theorem 5.25]{hewittross1963abstract}. 
\end{rem}

	Although we are primarily focused with uniform lattices, results of Bader--Furman--Sauer show a large class of groups do not possess any non-uniform lattice embeddings:
\begin{prop}[{\cite[Corollary 1.6]{bader2020lattice}}]\label{prop:non_uniform_lattices}
	Let $\Gamma$ be an acylindrically hyperbolic group. If $\Gamma$ is not virtually isomorphic to a  non-uniform lattice in a  connected rank one simple real Lie group, then all lattice embeddings of $\Gamma$ are  uniform.
\end{prop}
\begin{proof}
	Let $\Gamma$ be an acylindrically hyperbolic group that is isomorphic to a non-uniform lattice in locally compact group.  As noted in \cite[\S 2]{bader2020lattice}, acylindrically hyperbolic groups satisfy the hypotheses of \cite[Corollary 1.6]{bader2020lattice}. Therefore, $\Gamma$ is a virtually isomorphic to either
	\begin{itemize}
		\item a non-uniform lattice in a connected, center-free, semisimple real Lie group
		without compact factors;
		\item a non-uniform  $S$-arithmetic lattice in a semisimple group with both  real  and non-Archimedean   factors present \cite[\S 4.1]{bader2020lattice}.
	\end{itemize}
	In particular, $S$-arithmetic lattices in the sense of \cite{bader2020lattice}  always have $S$-rank at least two. Margulis' normal subgroup theorem  asserts that for higher-rank lattices in semisimple algebraic groups, all normal subgroups 	are finite or of finite-index \cite[Chapter VIII]{margulis}. However, since acylindrically hyperbolic groups are SQ-universal \cite{osin2016acylindrically},  $\Gamma$ cannot satisfy the conclusions of the normal subgroup theorem, since there are uncountably many countable groups.  Thus $\Gamma$ is a non-uniform lattice in a rank one simple Lie group as required.
\end{proof}

We make frequent use of the following useful lemma concerning continuous maps from totally disconnected groups to Lie groups.
\begin{lem}\label{lem:totdisc_to_lie}
	If $\rho:G\rightarrow L$ is a continuous homomorphism from a totally disconnected locally compact group to a Lie group, then  $\ker(\rho)$ is open.
\end{lem}

This follows from two well-known properties of Lie and totally disconnected groups:

\begin{prop}[No small subgroups]\label{prop:no_small_sbgps}
	If $G$ is a Lie group, then there exists a neighborhood $U$ of the identity such that every subgroup of $G$ contained in $U$ is trivial.
\end{prop}

\begin{thm}[\cite{vandantzig36}]\label{thm:vanD}
	If $G$ is a totally disconnected locally compact group, then every neighborhood of the identity contains a compact open subgroup.
\end{thm}

\begin{proof}[Proof of Lemma \ref{lem:totdisc_to_lie}]
	Proposition \ref{prop:no_small_sbgps} ensures there is an identity neighborhood $U\subseteq L$ such that~$U$  contains no nontrivial subgroups. By Theorem \ref{thm:vanD}, $\rho^{-1}(U)$ contains a compact open subgroup~$V$. Our choice of $U$ ensures $\rho(V)\subseteq U$ is trivial, hence $V\leq \ker(\rho)$. Since $V$ is open and $\ker(\rho)$ is a union of cosets of $V$, $\ker(\rho)$ is open.
\end{proof}

We will also make use of the following related notions. 
\begin{defn}
	Let $G$ be a group and let $H,K\leq G$ be two subgroups. 
	\begin{enumerate}
		\item $H$ and $K$ are \emph{commensurable} if $H\cap K$ has finite index in both $H$ and $K$.
		\item $H$ and $K$ are \emph{weakly commensurable} if there exists some $g\in G$ such that $H$ and $gKg^{-1}$ are commensurable.
		\item $H$ is \emph{commensurated} if for all $g\in G$, $H$ and $gHg^{-1}$ are commensurable.
	\end{enumerate}
\end{defn}
Two important and closely related sources  of commensurable and commensurated subgroups are the following:
\begin{example}
	Two compact open subgroups of a topological group are commensurable.  In particular, every compact open subgroup is commensurated.
\end{example}

\begin{example}\label{exmp:vstab_comm}
	If a group $G$ acts on a connected locally finite graph $X$ by graph automorphisms, then for all $v,w\in VX$, the vertex stabilizers $G_v$ and $G_w$ are commensurable. In particular, every $G_v$ is commensurated in $G$.
\end{example}

  \subsection{Actions of locally compact groups on metric spaces}	
\begin{defn}\label{defn:action_loc_compact}
	Let $X$ be a proper metric space and $G$ a locally compact group.   An isometric action $\rho:G\to \Isom(X)$ is:
	\begin{enumerate}
		\item \emph{continuous} if $\rho$ is continuous, or equivalently, the map $G\times X\to X$ given by $(g,x)\mapsto \rho(g)(x)$ is continuous \cite[Proposition 5.B.6]{cornulierdlH2016metric};
		\item  \emph{proper} if for every compact $K\subseteq X$, the set $\{g\in G\mid gK\cap K\neq \emptyset\}$ has compact closure in $G$;
		\item \emph{cocompact} if there exists a compact $K\subseteq X$ such that $X=\rho(G)K$;
		\item \emph{geometric} if it is proper, cocompact, and continuous;
		\item \emph{discrete} if $\rho(G)$ is a discrete subgroup of $\Isom(X)$.
	\end{enumerate}
\end{defn}

\begin{rem}
    Throughout this paper, if $\Gamma$ is a finitely generated group, we will assume that $\Gamma$ is equipped with the discrete topology unless stated otherwise.  
\end{rem}

Note that if $X$ is a proper metric space, then $\Isom(X)$ acts properly and continuously on $X$.
    We make extensive use of the following well-known lemma; see Proposition~\ref{prop:ARequivalences}.

    \begin{lem}\label{lem:geom_lattice}
        If $\rho:\G\to\Isom(X)$ is a geometric action of a discrete group on a proper metric space, then $\rho$ is a uniform lattice embedding modulo finite kernel.
    \end{lem}

The following equivalences are essential in the theory of totally disconnected groups from a geometric viewpoint; see \cite{kronmoller08roughcayley}.

\begin{lem}[{\cite{abels1973/74speckerkompaktifizierungen}, see also \cite{kronmoller08roughcayley} and \cite[Proposition 2.14]{margolis2024model}}]\label{lem:vanDantzig}
	Let $G$ be a compactly generated locally compact group. The following are equivalent:
	\begin{enumerate}
		\item \label{item:vanDan1} $G$ is compact-by-(totally disconnected);
		\item \label{item:vanDan2} $G$ contains a compact open subgroup;
		\item \label{item:vanDan3} there exists a connected locally finite vertex-transitive graph $X$ and a   geometric action $G\curvearrowright X$.
	\end{enumerate}
\end{lem}

We now briefly discuss some elementary facts about compact subsets of locally compact groups that act on metric spaces. %
\begin{prop}\label{prop:compact<->bounded}
	Let $G$ be a locally compact group acting continuously and  properly on a proper metric space $X$. A subset $L\subseteq G$ has compact closure if and only if $L$ has bounded orbits in $X$.
\end{prop}
\begin{proof}
	Since the action of $G$ on $X$ is continuous, if $L\subseteq G$ has  compact closure, then the orbit $\overline{L}\cdot x$ is also compact, hence bounded.
	Conversely, suppose an orbit $L\cdot x$ is bounded, hence $K=\overline{L\cdot x}$ is compact. Then $L\subseteq \{g\in G\mid gK\cap K\neq \emptyset\}$ has compact closure as the action is proper.
\end{proof}

In  order to discuss additional some properties concerning compact normal subgroups, we require the following definitions:
\begin{defn}
	A isometry $\phi$ of a metric space $X$ is \emph{bounded} if $\sup_{x\in X} d(x,\phi(x))<\infty$.
\end{defn}

The following can be thought of as a coarse analogue of a space having no nontrivial bounded isometries:

  \begin{defn}\label{defn:tame}
	A metric space $X$ is \emph{tame} if for every $K\geq 1$ and $A\geq 0$, there is a constant $C = C(K,A)$ such that if $f,g:X\to X$ are $(K,A)$-quasi-isometries such that $\sup_{x\in X}d(f(x),g(x))<\infty$, then  $\sup_{x\in X}d(f(x),g(x))\leq C$.   \end{defn}

A finitely generated group is \emph{tame} if it is tame when equipped with the word metric.  

\begin{example}\label{exmp:tame_morse}
	A proper geodesic metric space whose Morse boundary contains at least three points is tame~\cite[Proposition 4.21]{margolis2022discretisable}. In particular, an acylindrically hyperbolic group is tame. 
\end{example}

\begin{example}[Not tame]
    The metric space $\R$ is not tame because the isometries $x \mapsto x+n$ are each bounded distance from the identity, but their distance in the sup metric from the identity is unbounded over $\{n \in \N\}$. More generally, any finitely generated group with infinite center is not tame by a similar argument using left translations by elements in the center. 
\end{example}

We use these definitions to describe maximal compact normal subgroups of certain locally compact groups.

\begin{prop}\label{prop:compactnormal.vs.bounded}
	Let $X$ be a proper metric space, $G$ a locally compact group and $\rho:G\to \Isom(X)$ a geometric action. 
	 If $X$ is tame, then	 \[B=\{g\in G\mid \rho(g) \text{ is a bounded isometry}\}\] is the maximal compact normal subgroup of $G$. In particular, if $X$ has no nontrivial bounded isometries, then $\ker(\rho)$ is  the maximal compact normal subgroup of $G$. 
\end{prop}
\begin{proof}
	 It follows from the definition that $B$ is normal. The tameness condition ensures there is a constant $C$ such that for any two bounded isometries $\phi,\psi$ of $X$, $\sup_{x\in X}d(\phi(x),\psi(x))\leq C$. Thus $B$ has bounded orbits in $X$, hence has  compact closure by Proposition \ref{prop:compact<->bounded}. The tameness condition ensures the limit of a convergent sequence of bounded isometries is also bounded, hence $B$ is closed. Thus $B$ is a compact normal subgroup.

	 We now show that any compact normal subgroup $K\lhd G$ is contained in $B$, which shows $B$ is maximal. The Milnor--Schwarz lemma for locally compact groups \cite[Theorem 4.C.5]{cornulierdlH2016metric} ensures that the action of $G$ on the Cayley graph of $G/K$ is quasi-conjugate to the action of $G$ on $X$. Thus every element of $K$ must be a bounded isometry of $X$, hence $K\leq B$.
\end{proof}

We will make frequent  use of the following lemma to deduce  continuity of a group action.  

\begin{lem}\label{lem:continuous_quasiconj}
	Let $X$ and $Y$ be proper metric spaces, and assume $Y$ is tame with cocompact isometry group.   Suppose  $H\leq \Isom(X)$ is  a closed subgroup and $\rho:H\rightarrow \Isom(Y)$ is an isometric action. Suppose also $f:X\rightarrow Y$ is a coarsely surjective, coarse Lipschitz map such that $\sup_{x\in X,h\in H}d(\rho(h)f(x),f(hx))<\infty.$
	Then the composition \[H\xrightarrow{\rho} \Isom(Y)\xrightarrow{q} \Isom(Y)/K\] is continuous, where $K\vartriangleleft \Isom(Y)$ is the maximal compact normal subgroup of $\Isom(Y)$ and $q$ is the quotient map. 
\end{lem}
\begin{proof}
	By  Proposition \ref{prop:compactnormal.vs.bounded},  the subgroup $K$ of $\Isom(Y)$ consisting of bounded isometries of $Y$ is the maximal compact normal subgroup of $\Isom(Y)$.	
	Suppose  $(\phi_i)$ is a sequence that converges to the identity in $H$. For each $x\in X$, there exists an $n_x$ such that $d(\phi_i(x),x)\leq 1$ for all $i\geq n_x$. We choose $L$ and $A$ such that  $f$ is $(L,A)$-coarse Lipschitz with $N_A(f(X))=Y$ and $\sup_{x\in X,h\in H}d(\rho(h)f(x),f(hx))\leq A$. For each $y\in Y$, there exists $x_y\in X$ with $d(f(x_y),y)\leq A$. Hence for all $i\geq n_{x_y}$, we have 
	\begin{align*}
		d(\rho(\phi_i)(y),y)&\leq d(\rho(\phi_i)(f(x_y)),f(x_y))+2A\leq d(f(\phi_i (x_y)),f(x_y))+3A\\
		&\leq Ld(\phi_i( x_y),x_y)+4A\leq L+4A\eqqcolon A'.
	\end{align*}
    Thus, $\{\rho(\phi_i)\}$ has bounded orbits in $Y$ and, by    
    Proposition~\ref{prop:compact<->bounded}, the closure of $\{\rho(\phi_i)\}$ in $\Isom(Y)$ is compact. Therefore, every subsequence of $(\rho(\phi_i))$ has a convergent subsequence.      
    
    We wish to show $(\rho (\phi_i)K)$ converges to $K$ in $\Isom(Y)/K$. If this is not the case, then there exists  a convergent subsequence of $(\rho (\phi_i))$ that converges to some $\lambda\in \Isom(Y)\setminus K$. However, we know for each $y\in Y$, there exists an $N_y$ sufficiently large such that  $d(\rho(\phi_i)(y),y)\leq A'$ for all $i\geq N_y$, which implies $d(\lambda(y),y)\leq A'$ for all $y\in Y$.  This contradicts  our assumption $\lambda\notin K$.
\end{proof}
In the case where $Y$ has no nontrivial bounded isometries, Proposition \ref{prop:compactnormal.vs.bounded} ensures that  the maximal compact normal subgroup of $Y$ is trivial. 
We therefore deduce the following important special case of Lemma \ref{lem:continuous_quasiconj}.
\begin{cor}\label{cor:cts_quasiaction}
	Let $X$, $Y$, $H$, $f$ and $\rho$ be as in Lemma \ref{lem:continuous_quasiconj}. In addition, assume that $Y$ has no nontrivial bounded isometries. 
	Then  $\rho$ is   continuous. 
\end{cor}

\subsection{Action rigidity} \label{subsec:prelim_action}
In this section we give  equivalent characterizations of action rigidity in terms of group actions on metric spaces and in terms of lattice embeddings modulo finite kernels into locally compact groups.

A metric space $X$ is \emph{quasi-geodesic} if it is quasi-isometric to a geodesic metric space. This property is called \emph{large-scale geodesic} in \cite{cornulierdlH2016metric}; see Definition 3.B.1 and Lemma 3.B.6 of \cite{cornulierdlH2016metric} for more details and equivalent formulations. 
For instance, a nontrivial finitely generated group equipped with a word metric is not a geodesic metric space but is  a quasi-geodesic space, since it is quasi-isometric to the corresponding Cayley graph.

\begin{prop}[{\cite[Corollary 7]{moshersageevwhyte_trees}}]\label{prop:ARequivalences}
	Let $\Gamma_1$ and $\Gamma_2$ be finitely generated groups. The following are equivalent:
	\begin{enumerate}
		\item \label{item:ARequiv1} 
		$\Gamma_1$ and $\Gamma_2$ are uniform lattices modulo finite kernels in the same locally compact group;
		\item \label{item:ARequiv2} $\Gamma_1$ and $\Gamma_2$ act geometrically on the same proper quasi-geodesic metric space;
		\item \label{item:ARequiv3} $\Gamma_1$ and $\Gamma_2$ act geometrically on the same proper metric space.
	\end{enumerate}
\end{prop}
\begin{proof}
	The equivalence of (\ref{item:ARequiv1}) and (\ref{item:ARequiv2}) was shown in \cite[Corollary 7]{moshersageevwhyte_trees}, (\ref{item:ARequiv2})$\implies$(\ref{item:ARequiv3}) is obvious, and (\ref{item:ARequiv3})$\implies$(\ref{item:ARequiv1}) follows from Lemma \ref{lem:geom_lattice}.
\end{proof}

\begin{defn}
    Groups $G_1$ and $G_2$ are {\it abstractly commensurable} if they have isomorphic finite-index subgroups. 
\end{defn}

\begin{defn} \label{defn:AC_VI}
  Topological groups $G_1$ and $G_2$ are \emph{virtually isomorphic} if there exist finite-index open subgroups $H_1 \leq G_1$ and $H_2 \leq G_2$ and compact normal subgroups $K_1 \triangleleft H_1$ and $K_2 \triangleleft H_2$ so that $H_1/K_1$ and $H_2/K_2$ are topologically isomorphic. 
\end{defn}

\begin{rem}\label{rem:lattice_embedding_virt_isom}
    It can be readily verified that a lattice embedding modulo finite kernel $\rho:\Gamma\to G$ is trivial if and only if $G$ and $\Gamma$ are virtually isomorphic.
\end{rem}

Two discrete groups that are abstractly commensurable are also virtually isomorphic. Although the converse does not hold in general, there is the following partial converse:

\begin{lem}\label{lem:tfree}
	Suppose $\Gamma_1$ and $\Gamma_2$ are virtually torsion-free discrete groups. Then $\Gamma_1$ and $\Gamma_2$  are abstractly commensurable if and only if they are virtually isomorphic.
\end{lem}

The equivalences in the next definition follow from Proposition~\ref{prop:ARequivalences}. 

\begin{defn}\label{defn:AB_COMM&VISOM}
	A finitely generated group $\Gamma$ is said to be \emph{action rigid} if either of the following equivalent conditions hold: 
	\begin{enumerate}
		\item whenever $\Gamma'$ is a finitely generated group  such that $\Gamma$ and $\Gamma'$ both act geometrically on the same proper metric space, then $\Gamma$ and $\Gamma'$ are virtually isomorphic; 
		\item whenever $\Gamma'$ is a finitely generated group and there exist uniform lattice embeddings modulo finite kernels of $\Gamma$ and $\Gamma'$ into the same locally compact group, then $\Gamma$ and $\Gamma'$ are virtually isomorphic.
	\end{enumerate}
\end{defn}

    The following lemma follows immediately from definitions. 
    
    \begin{lem} \label{lem:triv_env_act_rig}
        If all uniform lattice embeddings modulo finite kernels of a finitely generated group $\Gamma$ are trivial, then $\Gamma$ is action rigid. 
    \end{lem}

 \subsection{Graphs of groups and spaces}
 
 In this subsection we define graphs of groups, graphs of spaces, and related concepts.
 We refer to~\cite{SerreTrees, ScottWall79} for further background.

 \begin{defn}\label{defn:GoS}
A \emph{graph of spaces} $(X,\Lambda)$ is a graph $\Lambda$ with the following data:
\begin{enumerate}
	\item a connected length space $X_v$ for each $v \in V\Lambda$ (called a \emph{vertex space})
	\item a connected length space $X_e$ for each $e\in E\Lambda$ (called an \emph{edge space}) such that $X_{\bar{e}} = X_e$
	\item a $\pi_1$-injective path-isometric map $\phi_e : X_e \rightarrow X_{\tau(e)}$ for each $e\in E\Lambda$.
\end{enumerate}

The \emph{total space} $X$ is the following quotient space:
\[
 X = \left( \bigsqcup_{v \in V \Lambda} X_v \bigsqcup_{e \in E\Lambda} X_e \times [0,1]\right)  \; \Big/ \; \sim, 
\]
where $\sim$ is the relation that identifies $X_e \times \{0\}$ with $X_{\bar{e}} \times \{0\}$ via the identification of $X_e$ with $X_{\bar{e}}$, and the point $(x, 1)\in X_e \times [0,1]$ with $\phi_e(x)$.
Note that $X$ inherits a length space metric from the vertex and edge spaces.
If $\Lambda$ is a tree, then we call $(X,\Lambda)$ a \emph{tree of spaces}.
\end{defn}

\begin{defn} \label{def:gog_gosp} 
      A \emph{graph of groups} $\mathcal{G}$ consists of the following data:
      \begin{enumerate}
      	\item a graph $\Lambda$  (called the \emph{underlying graph}),
          \item a group $G_v$ for each $v \in V\Lambda$ (called a \emph{vertex group}),
          \item  a group $G_e$ for each $e \in E\Lambda$ (called an \emph{edge group}) such that $G_e = G_{\bar{e}}$,
          \item and injective homomorphisms $\Theta_e: G_e \rightarrow G_{\tau(e)}$ for each $e \in E\Lambda$ (called \emph{edge maps}).    
      \end{enumerate}
           If $\cG$ is a graph of groups with underlying graph $\Lambda$, a \emph{graph of spaces associated to  $\mathcal{G}$} consists of a graph of spaces $(X,\Lambda)$ with points $x_v \in X_v$ and $x_e \in X_e$ so that $\pi_1(X_v,x_v) \cong G_v$ and $\pi_1(X_e, x_e) \cong G_e$ for all $v \in V\Lambda$ and $e \in E\Lambda$, and the maps $\phi_e:X_e\rightarrow X_{\tau(e)}$ are such that $(\phi_e)_* = \Theta_e$.  
      The \emph{fundamental group} of the graph of groups $\mathcal{G}$ is $\pi_1(X)$.
      The fundamental group is independent of the choice of $X$, so we will often denote it by $\pi_1(\cG)$.
      \end{defn}

Given a graph of spaces $(X,\Lambda)$, any covering  $f:\hat{X}\to X$ naturally has the structure of a graph of spaces $(\hat{X},\hat{\Lambda})$, where the edge and vertex spaces correspond to components of the preimages of the edge and vertex spaces in $(X,\Lambda)$.
Moreover, $f$ restricts to coverings between edge and vertex spaces.
We will usually denote such a covering by $f:(\hat{X},\hat{\Lambda})\to (X,\Lambda)$ to emphasize that the graph of spaces structure is preserved.

\begin{defn}\label{defn:actonToS}
  An \emph{isometric action of a group $G$ on a tree of spaces $(X,T)$} consists of the following data:
  \begin{enumerate}
      \item an action of $G$ on $T$
      \item for each $g\in G$, $v\in VT$ and $e\in ET$, isometries $g_v:X_v\to X_{gv}$ and $g_e=g_{\bar{e}}:X_e\to X_{ge}$, such that if $v=\tau(e)$,  the following diagram commutes \[
            \begin{tikzcd}
                X_e \arrow{r}{g_e}	\arrow{d}{\phi_e} 	& X_{ge}\ar{d}{\phi_{ge}}\\
                X_v \arrow{r}{g_v} 		& X_{gv}.
            \end{tikzcd}\]
            Moreover, if the vertex and edge spaces are cell complexes, then we additionally require the maps $g_v$ and $g_e$ to be cellular.
  \end{enumerate}
  Note there is an induced isometric action of $G$ on the total space $X$.
  We say that the action of $G$ on $(X,T)$ is \emph{geometric} if the action on $X$ is geometric.
\end{defn}

We will often consider a group $\G$ acting on a tree of spaces $(X,T)$, with simply connected edge and vertex spaces, free actions of the edge and vertex stabilizers on the corresponding edge and vertex spaces, and no edge inversions in $T$. Then $X/\G$ has the structure of a graph of spaces  with underlying graph $T/\G$.
We call this the \emph{quotient graph of spaces}, and denote it by $(X,T)/\G$.
Moreover, the quotient map defines a covering $(X,T)\to (X,T)/\G$, and $\pi_1(X/\G)=\G$.

Let $G$ be a locally compact group generated by a compact set $S\subseteq G$ with Cayley graph $X$. The \emph{number of ends} $e(G)$ of $G$ is the supremum of the number of unbounded components of $X\setminus B$, where $B$ is a bounded subset of $X$. This is independent of the choice of $S$, since any two Cayley graphs of $G$ with respect to different compact generating sets are quasi-isometric.  We remark that if the Cayley graph of $G$ is not locally finite, then the number of ends as defined above may be different from the standard definition as defined in \cite{DrutuKapovich18}.

    A \emph{Stallings--Dunwoody decomposition} of a group $\Gamma$ is a graph of groups decomposition satisfying the conclusions of the  next theorem. 
     
    \begin{thm}[\cite{dunwoody1985accessibility, stallings}] 
      If $\Gamma$ is a finitely presented group, then $\Gamma$ is the fundamental group of a finite graph of groups with finite edge groups and vertex groups that have at most one end. 
    \end{thm}

\subsection{Commanding subgroups} \label{sec:commanding}

This subsection concerns the notion of a group commanding a collection of subgroups. See \cite[Section 6]{Shepherd23} for more details.

\begin{defn} \label{defn:commands}
	A group $\G$ \emph{commands} a collection of subgroups $(\gL_1,...,\gL_k)$ if there exist finite-index subgroups $\dot{\gL}_i<\gL_i$ such that, for any choice of finite-index subgroups $\hat{\gL}_i<\dot{\gL}_i$ with $\hat{\gL}_i\triangleleft \gL_i$, there exists a finite-index normal subgroup $\hat{\G}\triangleleft \G$ such that $\gL_i\cap\hat{\G}=\hat{\gL}_i$ for $1\leq i\leq k$.
\end{defn}

If a group $\G$ is the fundamental group of a graph of groups such that each vertex group commands its collection of incident edge groups, then it is possible to build a finite cover of the graph of groups in a flexible manner by picking finite-index subgroups of the vertex groups and using the commanding property to ensure that the incident edge groups match up. We will employ a variation of this argument in Section \ref{sec:commoncover} to a setting where two groups act on a tree of spaces, and we wish to build a common finite cover of the quotient graphs of spaces.

\begin{rem}
	Note that the collection of subgroups $(\gL_i)$ from Definition \ref{defn:commands} might contain duplicates. 
	Indeed, if $\G$ commands $(\gL_i)$, then we can add any number of duplicates of the trivial subgroup to our family of subgroups and $\G$ will still command them.
\end{rem}

\begin{defn}\label{defn:commandsfinite}
We say that a group $\G$ \emph{commands its finite subgroups} if it commands any collection $(\gL_1,...,\gL_k)$ of finite subgroups. Equivalently, for every finite subgroup $\gL \leq \Gamma$, there exists a finite-index subgroup $\Gamma' \leq \Gamma$ so that $\Gamma' \cap \gL = \{1\}$. 
\end{defn}

For example, residually finite groups and virtually torsion-free groups command their finite subgroups.

\begin{example}
	A group is \emph{omnipotent} in the sense of \cite{wise2000subgroup} if and only if it commands the subgroups $\{\langle g_i\rangle\}$ for any finite set of independent elements $\{g_i\}$ \cite{Shepherd23}. A finite set $\{g_i\}$  is \emph{independent} if the $g_i$ have infinite order and no non-zero power of $g_i$ is conjugate to a non-zero power of $g_j$ for $i\neq j$.
\end{example}

\begin{example}
	The group $\mathbb{Z}^n$ commands subgroups $\{\gL_i\}$ if and only if they are independent in the sense that $\sum_i \lambda_i=0$ with $\lambda_i\in \gL_i$ implies $\lambda_i=0$ for all $i$ \cite[Proposition 6.5]{Shepherd23}.
\end{example}
A group is said to be \emph{cubulated} if it acts geometrically on a CAT(0) cube complex. 
	A collection $(\gL_i)_{i\in I}$ of subgroups of $G$ is \emph{almost malnormal} if for all $i,j\in I$ and $g\in G$, the intersection $g\gL_ig^{-1}\cap \gL_j$ is finite unless  $i=j$ and $g\in \gL_i$.
\begin{example}\label{rem:special_command}
	A cubulated hyperbolic group commands any finite almost malnormal collection of quasi-convex subgroups. This follows from Wise's Malnormal Special Quotient Theorem \cite{wise2021structure} combined with theorems of Agol \cite{Agol13} and Osin \cite{Osin07}; see \cite[Theorem 1.4]{Shepherd23}.
	If hyperbolic groups are all residually finite, then we can extend this to any hyperbolic group commanding any finite, almost malnormal collection of quasi-convex subgroups \cite[Theorem 6.9]{Shepherd23}.
\end{example}

\begin{rem}
	In general one cannot choose $\dot{\gL}_i=\gL_i$ in Definition \ref{defn:commandsfinite}. For instance, let $\G$ be the free group with generators $a, b$, and consider $\gL_1=\langle a\rangle$, $\gL_2=\langle b\rangle$, $\gL_3=\langle ab\rangle$ and $\hat{\gL}_1=\gL_1$, $\hat{\gL}_2=\gL_2$, $\hat{\gL}_3=\langle abab\rangle$. Then any $\hat{\G}\triangleleft \G$ with $\gL_i\cap\hat{\G}=\hat{\gL}_i$ for $i=1,2$ must contain $a$ and $b$, so $\hat{\G}=\G$, and $\gL_3\cap\hat{\G}=\gL_3\neq \hat{\gL}_3$.
\end{rem}

\section{Introducing graphically discrete groups} \label{sec:gd}

\subsection{Equivalent definitions of graphical discreteness}\label{subsec:graphdiscrete}

    If the automorphism group of a connected locally finite graph $X$ is discrete, then two groups $\Gamma$ and $\Gamma'$ acting geometrically on that graph are virtually isomorphic. Indeed, modulo the finite kernels of the action, $\Gamma$ and $\Gamma'$ are finite-index subgroups of $\Aut(X)$. However, it is too much to expect that for any action of a group $\Gamma$ on a connected locally finite graph $X$ that $\Aut(X)$ is discrete: indeed, attaching a pair of leaves at every vertex of $X$ produces a graph with non-discrete automorphism group. The following definition accounts for these compact obstructions to discreteness. See Figure~\ref{figure:imprim}.
 
    \begin{defn}
        If  $X$ is a graph and $G\leq \Aut(X)$, then a \emph{$G$-imprimitivity system} is an equivalence relation $\sim$ on $VX$ that is invariant under the action of $G$, i.e.\ $g\cdot v\sim g\cdot w$ whenever $v\sim w$. Equivalence classes of a $G$-imprimitivity system are called \emph{blocks}. The \emph{quotient graph} $X/\sim$ is the simplicial graph whose vertices are blocks, and two blocks $[v]$ and $[w]$  are joined by an edge if and only if $(v',w')\in EX$ for some  $v'\in [v]$ and $w'\in [w]$. 
    \end{defn}
\begin{rem}
	Even if the imprimitivity system $\sim$ is trivial, i.e.\ all blocks consist of a single vertex, the graph $X/\sim$ might not be equal to $X$. This is because $X$ might not be a simplicial graph, and quotienting out by the trivial imprimitivity system will remove one-edge loops and multiple edges.
\end{rem}

\begin{rem}
	If $G\leq \Aut(X)$ and $\sim$ is a $G$-imprimitivity system, then there is an \emph{induced action} of $G$ on $X/\sim$ given by $g\cdot[v]=[g\cdot v]$. 
\end{rem}

    \begin{figure}[ht!]
    \begin{center}
        \begin{overpic}[width=.8\textwidth, tics=5]{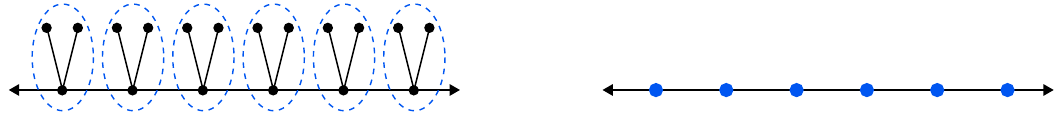} %
        \put(-2,7){$X$}
        \put(55,7){$X/\sim$}
        \end{overpic}
        \caption{The automorphism group of the graph $X$ is not discrete. The vertices are grouped in blue $\Aut(X)$-invariant finite equivalence classes so that the action of $\Aut(X)$ on the quotient graph $X/\sim$ is discrete.}
        \label{figure:imprim}
    \end{center}
    \end{figure}   
    
    A topological group $G$ is \emph{compact-by-discrete} if $G$ contains a compact normal subgroup such that the quotient, equipped with the quotient topology,  is discrete. Equivalently, $G$ is compact-by-discrete if it contains a compact open normal subgroup, since a normal subgroup is open if and only if the corresponding quotient is discrete. We can characterize compact-by-discrete automorphism groups using imprimitivity systems as follows. See Theorem \ref{thm:compactbydiscrete} for further characterizations.
    
    \begin{prop}\label{prop:compactbydiscrete}
        Let $X$ be a connected locally finite graph and let $G\leq\Aut(X)$ be a closed subgroup that acts cocompactly on $X$. Then $G$ is compact-by-discrete if and only if there exists a $G$-imprimitivity system $\sim$ on $X$ with finite blocks such that the induced action $G\curvearrowright X/\sim$ is discrete.
    \end{prop}

    \begin{proof}
        First, suppose there is a compact subgroup $K\vartriangleleft G$ such that $G/K$ is discrete.
        In particular, $K$ is open. We define an equivalence relation  $\sim$ on $VX$ by $v\sim w$ if and only if there exists $k\in K$ such that $kv=w$. As $K$ is normal, this is a $G$-imprimitivity system. Since $K$ is compact, every orbit $K\cdot v$ is finite by Proposition \ref{prop:compact<->bounded}, and so $\sim$ has finite blocks.  
        The kernel of the induced action $G\curvearrowright X/\sim$ is open because it contains the open subgroup $K$, and this is equivalent to saying that the induced action $G\curvearrowright X/\sim$ is discrete.
        
        Conversely, suppose there exists a $G$-imprimitivity system $\sim$ on $X$ with finite blocks such that the induced action $G\curvearrowright X/\sim$ is discrete. Since $\sim$ has finite blocks, the kernel~$K$ of the action $G\curvearrowright X/\sim$ is compact by Proposition \ref{prop:compact<->bounded}. The quotient $G/K$ is discrete because the action $G\curvearrowright X/\sim$ is discrete. Thus, $G$ is compact-by-discrete.
    \end{proof}

    \begin{example}
    If $G\leq \Aut(X)$ is cocompact and compact-by-discrete, then Proposition~\ref{prop:compactbydiscrete} tells us that there exists a $G$-imprimitivity system $\sim$ on $X$ with finite blocks such that the induced action $G\curvearrowright X/\sim$ is discrete.
    However, it is not necessarily true that $\Aut(X/\sim)$ is a discrete group, as shown in Figure~\ref{figure:autdiscrete}; note that $\Aut(X)$ does not surject onto $\Aut(X/\sim)$ in this example. 
	\end{example}

    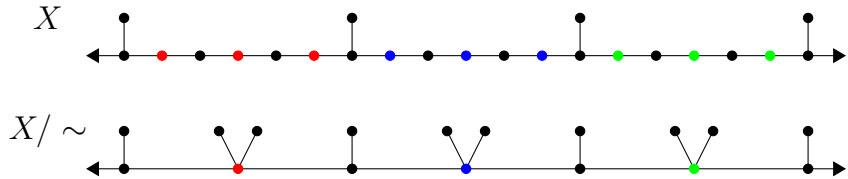
\begin{figure}[ht!]
		\centering
		\scalebox{.5}{
			\begin{tikzpicture}[auto,node distance=2cm,
				thick,every node/.style={circle,draw,fill,inner sep=0pt,minimum size=7pt},
				every loop/.style={min distance=2cm},
				hull/.style={draw=none},
				]
				\tikzstyle{label}=[draw=none,fill=none,font=\Huge]
				\tikzstyle{a}=[isosceles triangle,sloped,allow upside down,shift={(0,-.05)},minimum size=10pt]

\draw[{Triangle[scale=2]}-{Triangle[scale=2]}] (0,0)--(20,0);
\node at (1,0){};
\node[red] at (2,0){};
\node at (3,0){};
\node[red] at (4,0){};
\node at (5,0){};
\node[red] at (6,0){};
\node at (7,0){};
\node[blue] at (8,0){};
\node at (9,0){};
\node[blue] at (10,0){};
\node at (11,0){};
\node[blue] at (12,0){};
\node at (13,0){};
\node[green] at (14,0){};
\node at (15,0){};
\node[green] at (16,0){};
\node at (17,0){};
\node[green] at (18,0){};
\draw(1,0)--(1,1);
\draw(7,0)--(7,1);
\draw(13,0)--(13,1);
\draw(19,0)--(19,1);
\node at (19,0){};
\node at (1,1){};
\node at (7,1){};
\node at (13,1){};
\node at (19,1){};

\draw[{Triangle[scale=2]}-{Triangle[scale=2]}] (0,-3)--(20,-3);
\draw(9.5,-2)--(10,-3)--(10.5,-2);
\node at (9.5,-2){};
\node at (10.5,-2){};
\draw(3.5,-2)--(4,-3)--(4.5,-2);
\node at (3.5,-2){};
\node at (4.5,-2){};
\draw(15.5,-2)--(16,-3)--(16.5,-2);
\node at (15.5,-2){};
\node at (16.5,-2){};

\node at (1,-3){};
\node[red] at (4,-3){};
\node at (7,-3){};
\node[blue] at (10,-3){};
\node at (13,-3){};
\node[green] at (16,-3){};
\node at (19,-3){};
\draw(1,-3)--(1,-2);
\draw(7,-3)--(7,-2);
\draw(13,-3)--(13,-2);
\draw(19,-3)--(19,-2);
\node at (1,-2){};
\node at (7,-2){};
\node at (13,-2){};
\node at (19,-2){};

\node[label] at (-1,1) {$X$};
\node[label] at (-1,-2) {$X/\sim$};
				
			\end{tikzpicture}
		}
		\label{figure:autdiscrete}
		\caption{An example where $\Aut(X)$ acts discretely on $X/\sim$, but the group $\Aut(X/\sim)$ is not discrete. On the left, each set of colored vertices indicates one equivalence class, and each black vertex is in its own equivalence class.}
	\end{figure}

    We now turn to the definition of graphical discreteness.
    
    \begin{prop}\label{prop:propC}
        Let $\Gamma$ be a finitely generated group. The following are equivalent:
            \begin{enumerate}
                \item \label{item:propC1} If $\Gamma$ is a virtual  uniform lattice in a tdlc group $G$, then  $G$ is compact-by-discrete.
                \item \label{item:propC3} If  $\rho:\Gamma \to G$ is a uniform lattice embedding modulo finite kernel with  $G$ a tdlc group, then $\rho$ is trivial (in the sense of Definition \ref{def:VIembeddings}).
                \item \label{item:propC4} If $\Gamma$ acts geometrically on a connected locally finite graph $X$, then  $\Aut(X)$ is compact-by-discrete.
            \end{enumerate}
        If $\Gamma$ satisfies any of these equivalent properties, we say  $\Gamma$ is \emph{graphically discrete}. 
    \end{prop}

    \begin{proof}
        (\ref{item:propC3})$\iff$(\ref{item:propC1}) is a consequence of Definition \ref{def:VIembeddings} and Remark \ref{rem:triv<->c-by-d}.

        (\ref{item:propC1})$\implies$(\ref{item:propC4}): Suppose $\rho:\Gamma\rightarrow \Aut(X)$ is a geometric action of $\Gamma$ on a connected locally finite graph $X$. By Lemma \ref{lem:geom_lattice},  $\rho$ is a uniform lattice embedding modulo finite kernel.   Since $\Aut(X)$ is totally disconnected, (\ref{item:propC1}) implies  $ \Aut(X)$ is compact-by-discrete. 
        
        (\ref{item:propC4})$\implies$(\ref{item:propC1}): Suppose $\Gamma$ is a virtual  uniform lattice in a tdlc group $G$.
        Then $G$ is compactly generated by Lemma \ref{lem:finitevscompact}, and, by Lemma \ref{lem:vanDantzig},
         $G$ admits a  continuous, proper, cocompact action $\rho:G\to \Aut(X)$ on a connected locally finite graph $X$. In particular, $\Gamma$ acts geometrically on $X$. It then follows from (\ref{item:propC4}) that $\Aut(X)$ has a compact open normal subgroup $K$.
         Then $N\coloneqq \rho^{-1}(K)$ is clearly open and normal in $G$, and is compact because the action of $G$ on $X$ is proper.
         Hence $G$ is compact-by-discrete.
    \end{proof}

	\begin{rem}\label{rem:gd_maxcmpctopen}
		Suppose $\Gamma$ is graphically discrete and contains a maximal finite normal subgroup. If $\Gamma\to G$ is a uniform lattice embedding modulo finite kernel with $G$ totally disconnected, then there exists a compact normal subgroup $K\vartriangleleft G$ such that $G/K$ is discrete. Letting $\rho:\Gamma\to G/K$ be the uniform lattice embedding modulo finite kernel composed with the quotient map, we see that $\rho(\Gamma)$ is isomorphic to a finite-index subgroup of $G/K$, hence $G/K$ also has a maximal finite normal subgroup. Thus $G$ has a maximal compact normal subgroup, which contains $K$ hence is open.
	\end{rem}
    
   It thus follows from Proposition \ref{prop:propC} that a finitely generated group $\Gamma$ is graphically discrete if all its uniform lattice embeddings modulo finite kernels into tdlc groups are trivial. Since much of the literature is concerned with genuine (i.e.\ injective) lattice embeddings rather than lattice embeddings modulo finite kernels,  we note the following consequence of Proposition~\ref{prop:propC}. 

    \begin{cor}\label{cor:propctorsionfree}
        Let $\Gamma$ be a finitely generated group. Suppose there exists a finite-index torsion-free subgroup $\Gamma'\leq \Gamma$  such that whenever $\rho:\Gamma' \to G$ is a uniform lattice embedding with  $G$ a tdlc group, then $\rho$ is trivial.  Then $\Gamma$ is graphically discrete.  
    \end{cor}
\begin{proof}
	If $\rho:\Gamma\to G$ is lattice embedding modulo finite kernel, then $\rho|_{\Gamma'}$ is injective, hence a lattice embedding of $\Gamma'$.
\end{proof}

    One useful consequence of graphical discreteness, foreshadowing many of the rigidity results to appear later, is the following:
    
    \begin{prop} \label{prop:gdimpliesvi}
        Let $\Gamma_1$ and $\Gamma_2$ be two finitely generated groups, at least one of which is graphically discrete. If $\Gamma_1$ and $\Gamma_2$  act geometrically on the same connected locally finite graph, then $\Gamma_1$ and $\Gamma_2$ are virtually isomorphic.
    \end{prop}
    \begin{proof}
        Suppose $\Gamma_1$ is graphically discrete and that $\rho_i:\Gamma_i\rightarrow \Aut(X)$ is a geometric action on a connected locally finite graph $X$ for $i=1,2$. As $\Gamma_1$ is graphically discrete, there exists a compact subgroup $K\vartriangleleft \Aut(X)$ with $\Aut(X)/K$ discrete. For $i=1,2$, $\Gamma_i/\rho^{-1}_i(K)$ is isomorphic to a finite-index subgroup of $\Aut(X)/K$, thus $\Gamma_1/\rho^{-1}_1(K)$ and $\Gamma_2/\rho^{-1}_2(K)$ are abstractly commensurable. Since $\rho^{-1}_1(K)$ and $\rho^{-1}_2(K)$ are finite, $\Gamma_1$ and $\Gamma_2$ are virtually isomorphic.
    \end{proof}

\subsection{Characterization of compact-by-discrete automorphism groups}\label{subsec:compactbydiscrete}

\begin{thm}\label{thm:compactbydiscrete}
    Let $X$ be a connected locally finite graph and let $G\leq \Aut(X)$ be a closed subgroup that acts cocompactly on $X$.
    With the topology on $G$ induced from that on $\Aut(X)$, the following are equivalent:
    \begin{enumerate}
        \item\label{item:compactbydiscrete} $G$ is compact-by-discrete.
        \item\label{item:Abounded} There exists a constant $A>0$ such that $|G_x\cdot y|\leq A$ for all $x,y\in VX$, where $G_x$ denotes the stabilizer of $x$.
        \item\label{item:openQIkernel} The homomorphism $\phi:G\to\QI(X)$ has open kernel.
        \item\label{item:QIfinitestabs} For some (equivalently every) $x\in VX$, the stabilizer $G_x$ has finite image under the homomorphism $\phi:G\to\QI(X)$.
        \item\label{item:imprimitivity} There exists a $G$-imprimitivity system $\sim$ on $X$ with finite blocks such that the induced action  $G\curvearrowright X/\sim$ is discrete.
    \end{enumerate}
\end{thm}

\begin{rem}
	As $X$ is locally finite, all vertex stabilizers $G_x$ are commensurable by Example \ref{exmp:vstab_comm}. This ensures that for $\phi:G\to\QI(X)$ as above, $\phi(G_x)$ is finite for some $x\in VX$ if and only if $\phi(G_x)$ is finite for every $x\in VX$.
\end{rem}

    The theorem above provides several additional ways to define graphical discreteness. Note the equivalence of (\ref{item:compactbydiscrete}) and (\ref{item:imprimitivity}) was established in Proposition~\ref{prop:compactbydiscrete}. The result is also a useful tool for showing that certain groups are not graphically discrete. For instance, non-abelian free groups act geometrically on the $4$-valent tree, so one can deduce the following corollary from Theorem~\ref{thm:compactbydiscrete}(\ref{item:Abounded}) or (\ref{item:QIfinitestabs}).

    \begin{cor}
        Non-abelian free groups are not graphically discrete. 
    \end{cor}
    
See Section \ref{subsec:nongraphdisc} for more examples of non-graphically-discrete groups whose proofs use Theorem~\ref{thm:compactbydiscrete}, and see Section~\ref{sec:altgraphdiscrete} for further analysis in the setting of hyperbolic groups.

We first prove some of the easier implications in Theorem~\ref{thm:compactbydiscrete}.

\begin{lem}
The implications (\ref{item:compactbydiscrete})$\iff$(\ref{item:Abounded}) and (\ref{item:compactbydiscrete})$\implies$(\ref{item:openQIkernel})$\implies$(\ref{item:QIfinitestabs}) of Theorem~\ref{thm:compactbydiscrete} hold.
\end{lem}
\begin{proof}
   (\ref{item:compactbydiscrete})$\implies$(\ref{item:Abounded}): Let $K\triangleleft G$ be a compact open normal subgroup. The vertex orbits of $K$ are finite because $K$ is compact (Proposition \ref{prop:compact<->bounded}), and this set of orbits is $G$-invariant because $K$ is normal. Since $G$ acts cocompactly on $X$, the sizes of the vertex orbits of~$K$ are bounded by some constant $N$. Given $x\in VX$, the stabilizer $G_x$ is compact and open (Proposition \ref{prop:compact<->bounded}). Since the intersection $K_x=G_x\cap K$ is open, $K_x$ has finite index in $G_x$. Then given a second vertex $y$, we have 
   $$|G_x\cdot y|\leq [G_x:K_x]|K_x\cdot y|\leq [G_x:K_x]N.$$
   Finally, the index $[G_x:K_x]$ can be bounded independently of $x$ because $K$ is normal and $G$ acts cocompactly on $X$.
   
   (\ref{item:Abounded})$\implies$(\ref{item:compactbydiscrete}): Schlichting's Theorem \cite{Schlichting80} (see also \cite[Theorem 3]{BergmanLenstraHendrik89}) states that if $G$ is a group with a subgroup $H$ such that the conjugates of $H$ are \emph{uniformly commensurable} (meaning there is a finite upper bound for the indices $[H:H\cap gHg^{-1}]$ with $g\in G$), then $H$ is commensurable with a normal subgroup $K\triangleleft G$.
   Moreover, we can assume that $K$ contains a finite intersection of conjugates of $H$ (see the proof of \cite[Theorem 3]{BergmanLenstraHendrik89}).
   With $G$ being the group from Theorem \ref{thm:compactbydiscrete}, now suppose there exists a constant $A>0$ such that $|G_x\cdot y|\leq A$ for all $x,y\in VX$. Since $[G_x: G_x \cap G_y]=|G_x\cdot y|$,  all conjugates of a vertex stabilizer $G_x$ are uniformly commensurable. So, we can apply Schlichting's Theorem with $H$ being any vertex stabilizer $G_x$. The resulting normal subgroup $K\triangleleft G$ contains a finite intersection of vertex stabilizers as a finite-index subgroup. Hence, $K$ is open by the definition of the compact-open topology, and $K$ is compact by Proposition~\ref{prop:compact<->bounded}.

   (\ref{item:compactbydiscrete})$\implies$(\ref{item:openQIkernel}): Let $K\triangleleft G$ be a compact open normal subgroup. The vertex orbits of $K$ are finite because $K$ is compact (Proposition \ref{prop:compact<->bounded}), and the set of orbits is $G$-invariant because $K$ is normal. Since $G$ acts cocompactly on $X$ we deduce that the diameters of the vertex orbits of $K$ are uniformly bounded, hence $K$ lies in the kernel of the homomorphism $\phi:G\to\QI(X)$. As $K$ is open, we conclude that $\ker (\phi)$ is open.
   
   (\ref{item:openQIkernel})$\implies$(\ref{item:QIfinitestabs}): Suppose $\ker(\phi)$ is open and let $x\in VX$. As $\ker(\phi)$ is an open neighborhood of the  identity automorphism, the pointwise $G$-stabilizer of some ball about $x$ is contained in $\ker(\phi)$ (see \cite[Example 5.B.11]{cornulierdlH2016metric}). But $X$ is locally finite, so this pointwise stabilizer has finite index in $G_x$. Hence $\phi(G_x)$ is finite.
   \end{proof}

To complete the proof of Theorem \ref{thm:compactbydiscrete}, it remains to prove the implication (\ref{item:QIfinitestabs})$\implies$(\ref{item:compactbydiscrete}).
For this we need the following theorem of Trofimov.  Recall a \emph{bounded automorphism} $\phi\in \Aut(X)$ of a  graph $X$ is a graph automorphism of $X$ such that $\sup_{x\in VX}d(x,\phi(x))<\infty$. Note that a graph automorphism is bounded precisely if it lies in the kernel of the natural map $\Aut(X)\to \QI(X)$.

\begin{thm}\cite[Proposition 2.3]{trofimov1984polynomial}\label{prop:trofimov}
    Let $X$ be a connected locally finite vertex-transitive graph.
    There exists an $\Aut(X)$-imprimitivity system $\sim$ on $X$ with finite blocks, such that no nontrivial bounded automorphism of the induced action $\Aut(X)\curvearrowright X/\sim$ stabilizes a vertex of $X/\sim$.
\end{thm}

We now complete the proof of Theorem \ref{thm:compactbydiscrete} with the following lemma.

\begin{lem}
The implication (\ref{item:QIfinitestabs})$\implies$(\ref{item:compactbydiscrete}) from Theorem \ref{thm:compactbydiscrete} holds.
Namely, given a connected locally finite graph $X$ and $G<\Aut(X)$ a closed subgroup acting cocompactly on $X$, if the stabilizer $G_x$ has finite image under the homomorphism $\phi:G\to\QI(X)$ for all $x\in VX$, then $G$ is compact-by-discrete.
\end{lem}
\begin{proof}
    Fix $x\in VX$. Given $d>0$, we define a graph $Y$ with vertex set the $G$-orbit $G\cdot x$ and an edge joining $x_1,x_2\in G\cdot x$ whenever $d(x_1,x_2)\leq d$. The graph $Y$ is locally finite for all $d$, connected for sufficiently large $d$, and the group $G$ acts on $Y$ with a single vertex orbit. 
    We pick $d$ sufficiently large so that $Y$ is connected.
    Let $\sim$ be the $\Aut(Y)$-imprimitivity system on $Y$ provided by Proposition \ref{prop:trofimov}.
    In particular, $\sim$ is a $G$-imprimitivity system with finite blocks such that no nontrivial bounded automorphism of the induced action $G\curvearrowright Y/\sim$ stabilizes a vertex of $Y/\sim$.
    
    Let $K$ be the kernel of the homomorphism $G\to \Aut(Y/\sim)$. Proposition~\ref{prop:compact<->bounded} implies that $K$ is compact, so it remains to show that $K$ is open.
    Let $K_x= G_x\cap K$; since $K$ is a union of $K_x$-cosets, it suffices to prove $K_x$ is open.
    The subgroup $G_x$ stabilizes the vertex $[x]\in V(Y/\sim)$, so all elements of $G_x - K_x$ act on $Y/\sim$ as unbounded automorphisms, hence they also act on $X$ as unbounded automorphisms.
    Thus, the elements of $G_x - K_x$ lie outside the kernel of $\phi:G\to\QI(X)$.
    We assumed that $\phi(G_x)$ is finite, so we deduce that $K_x$ has finite index in $G_x$.
    Let $G_x=K_x\sqcup g_1 K_x\sqcup...\sqcup g_n K_x$ be the partition of $G_x$ into $K_x$-cosets.
    The $g_i$ are not in $K$, so $g_i[x_i]\neq [x_i]$ for some choice of $\sim$-block $[x_i]$.
    We deduce that
    $$K_x=G_x\cap\bigcap_{i=1}^n G_{[x_i]},$$
    where $G_{[x_i]}$ is the stabilizer of the block $[x_i]$.
    A finite intersection of open subgroups is open, therefore $K_x$ is open.
    Thus, $K$ is open, as required.
\end{proof}

\subsection{Examples of graphically discrete groups}

    \begin{thm}\label{thm:propc_examples}   
        Finitely generated groups in the following families are graphically discrete:
    \begin{enumerate}
        \item Virtually nilpotent groups \cite{trofimov1984polynomial}. \label{propc:trofimov}
        \item Irreducible lattices  in connected  center-free, real semisimple Lie groups without compact factors,  other than non-uniform lattices in $\PSL_2(\bR)$ \cite{furman2001mostow, bader2020lattice}. \label{propc:latticelie}
        \item Irreducible $S$-arithmetic lattices in the sense of \cite{bader2020lattice}. \label{propc:sarithmetic}
        \item Fundamental groups of closed irreducible $3$-manifolds with nontrivial geometric decomposition. \label{propc:nongeom}
        \item Finitely generated groups $\Gamma$ containing a  finite-index  torsion-free subgroup with no nontrivial lattice embeddings. Examples include: \label{propc:trivial lattice}
            \begin{enumerate}
                \item The mapping class group of a non-exceptional finite-type	orientable surface \cite[Theorem 1.4]{kida2010MCGs}
                \item The pure braid group $\PB_k(S_g)$ of $k$ strands on the closed surface of genus $g$, where $g,k\geq 2$ \cite[Theorem B]{chifankida2015and}.
                \item $\Out(F_n)$ for $n\geq 3$ \cite[Theorem 1.5]{guirardel2021measure}
                \item Various  2-dimensional Artin groups of hyperbolic type \cite[Theorem 10.9]{horbez2020boundary}
                \item The Higman groups \[\Hig_n=\langle a_1,\dots, a_n\mid a_ia_{i+1}a_i^{-1}=a_{i+1}^2 \text{ for $i=1,\dots, n\mod n$ }\rangle \] for $n\geq 5$ \cite[Corollary 1.3]{horbez2025measure}. 
                \item Fundamental groups of Riemannian $n$-manifolds for $n\geq 5$, with pinched negative curvature in the range $[-(1+\frac{1}{n-1})^2,-1]$,  that are not homeomorphic to closed hyperbolic manifolds \cite[Theorem E]{bader2020lattice}.
            \end{enumerate}
        \item Any finitely generated tame (see Definition \ref{defn:tame}) group $\Gamma$ for which the  natural map $\Gamma\rightarrow \QI(\Gamma)$ has finite-index image. Examples include: \label{propc:QI}
        \begin{enumerate}
            \item the mapping class group of non-exceptional finite-type	orientable surface \cite{behrstock2012mcgs};
            \item hyperbolic surface--by--free groups considered by Farb--Mosher \cite[Theorem 1.3]{farbmosher2002surfacebyfree};
            \item topologically rigid hyperbolic groups of Kapovich--Kleiner \cite{kapovichkleiner2000hyperbolic}.
        \end{enumerate}
    \end{enumerate}
    \end{thm}

    \begin{proof}
        (\ref{propc:trofimov}): Since $\Gamma$ is virtually nilpotent, it has polynomial growth \cite{wolf1968growth}. Suppose $\Gamma$  acts geometrically on a locally finite vertex transitive graph $X$. Then $\Gamma$ is quasi-isometric to $X$ so $X$ has polynomial growth.
        A result of Trofimov, which generalizes Gromov's polynomial growth theorem,    ensures there is an $\Aut(X)$-imprimitivity system $\sim$ on $X$ with finite blocks such that the induced action  $\Aut(X)\curvearrowright X/\sim$ is discrete \cite[Theorem 1]{trofimov1984polynomial}. Thus $\Gamma$ satisfies (\ref{item:propC1}) of Proposition \ref{prop:propC}, hence is graphically discrete.
        
        (\ref{propc:latticelie}): Let $H$ be a center-free, connected real semisimple Lie group without compact factors. Let  $\Gamma\leq H$ be an irreducible lattice, which is assumed to be uniform if $H$ is locally isomorphic to  $\PSL_2(\bR)$.  By Selberg's lemma,  $\Gamma$ contains a finite-index torsion-free subgroup $\Gamma'$, which is also a lattice in $H$. Then \cite[Theorems B and D]{bader2020lattice} ensure  every lattice embedding $\rho:\Gamma' \hookrightarrow G$ is trivial or virtually isomorphic to $\Gamma' \hookrightarrow H$. The latter is impossible if $G$ is totally disconnected, thus $\Gamma$ is  graphically discrete by Corollary \ref{cor:propctorsionfree}.
           
        (\ref{propc:sarithmetic}): This also follows from \cite[Theorem B]{bader2020lattice},  similarly to (\ref{propc:latticelie}) above (see \cite[Remark 4.2]{bader2020lattice}).
       
		(\ref{propc:nongeom}): This is shown in Theorem \ref{thm:nongeo3mangd}.       
        
        (\ref{propc:trivial lattice}): Any such group satisfies the hypotheses of Corollary \ref{cor:propctorsionfree}, hence is graphically discrete. We note that $\Hig_n$ is torsion-free \cite[Lemma 2.7]{horbez2025measure}, and every lattice embedding of $\Hig_n$  is trivial; see \cite[Corollary 5.20]{horbez2025measure}.
        
        (\ref{propc:QI}): This will follow from Theorem \ref{thm:discrete_qi}.
    \end{proof}

\subsection{Examples of non-graphically-discrete groups}\label{subsec:nongraphdisc}

\begin{prop}\label{prop:HNNnotGD}
Let $\G$ be a finitely generated group and let $\Lambda\leq \G$ be a proper subgroup. Then the HNN extension 
$$\G*_\Lambda=\langle\G,t\mid t\lambda t^{-1}=\lambda, \lambda\in\Lambda\rangle$$
is not graphically discrete.
\end{prop}
\begin{proof}
    Let $S$ be a finite generating set for $\G$ and let $X$ be the Cayley graph of $\G*_\Lambda$ with respect to the generating set $S\cup\{t\}$.
    
    We first define an involution of $\G*_\Lambda$.
    Write each element of $\G*_\Lambda$ in the form $w=t^{n_0}\gamma_1 t^{n_1}\gamma_2t^{n_2}\cdots\gamma_k t^{n_k}$ with $\gamma_i\in\Gamma - \Lambda$ (with $\gamma_1\in\Lambda$ also permitted if $k=1$ and $n_1=0$) and with $n_i\neq0$ (except possibly $n_0$ and $n_k$). This normal form is unique in the sense that the numbers $n_0,n_1,\dots,n_k$ are uniquely determined by $w$.
    Define an involution $\sigma:\G*_\Lambda\to \G*_\Lambda$ by $t^{n_0}\gamma_1 t^{n_1}\cdots\gamma_k t^{n_k}\mapsto t^{-n_0}\gamma_1 t^{n_1}\cdots\gamma_k t^{n_k}$.
    
    We claim that $\sigma$ defines an automorphism of $X$.
    Indeed, given $w$ as above and $s\in S$, the normal form for $ws$ is $t^{n_0}\gamma_1 t^{n_1}\cdots\gamma_k t^{n_k}s$ if $s\notin\Lambda$ and $t^{n_0}\gamma_1 t^{n_1}\cdots\gamma_ks t^{n_k}$ if $s\in\Lambda$. Hence $\sigma(ws)=\sigma(w)s$, and $g$ maps the edge $(w,ws)$ to the edge $(\sigma(w),\sigma(w)s)$.
    The normal form for $wt$ is $t^{n_0}\gamma_1 t^{n_1}\cdots\gamma_k t^{n_k+1}$, unless $k=1$, $n_1=0$ and $\gamma_1\in\Lambda$, in which case the normal form is $t^{n_0+1}\gamma_1$.
    In the first case, we have $\sigma(wt)=\sigma(w)t$, while in the second case we have $\sigma(wt)=\sigma(w)t^{-1}$ -- in either case the edge $(w,wt)$ maps to another edge in $X$. A similar argument holds for $wt^{-1}$.

    Let $w\in\G - \Lambda$ and $w_i=t^iw$ for $i\in\Z$.
     Each $w_i$ and $w_i^{-1}$ act on $\G*_\Lambda$ by left multiplication, so define automorphisms of $X$. 
     Hence we get a sequence $(w_i\sigma w_i^{-1})$ in $\Aut(X)$.
     We claim that these automorphisms are at infinite distance from each other.
     Indeed, let $i\neq j$ and let $k\geq1$ be an integer. We claim that the distance between $(w_i\sigma w_i^{-1})(t^i(wt)^k)$ and $(w_j\sigma w_j^{-1})(t^i(wt)^k)$ is at least $2(2k-1)$.
     We have
     \begin{align*}
         (w_i\sigma w_i^{-1})(t^i(wt)^k)&=w_i\sigma(t(wt)^{k-1})\\
         &=t^iwt^{-1}(wt)^{k-1},\\
     \end{align*}
     and
     \begin{align*}
         (w_j\sigma w_j^{-1})(t^i(wt)^k)&=w_j\sigma(w^{-1}t^{-j+i}(wt)^k)\\
         &=t^jww^{-1}t^{-j+i}(wt)^k\\
         &=t^i(wt)^k.
     \end{align*}
     The distance between these vertices is the word length of the following element:
     \begin{align*}
         (t^i(wt)^k)^{-1}t^iwt^{-1}(wt)^{k-1}&=(t^{-1}w^{-1})^k t^{-i}t^iwt^{-1}(wt)^{k-1}\\
         &=(t^{-1}w^{-1})^{k-1}t^{-2}(wt)^{k-1}.
     \end{align*}
     As $w$ has word length at least 1, the above element (which is in normal form) has word length at least $2(2k-1)$.

     The automorphisms $(w_i\sigma w_i^{-1})$ all fix the base vertex $1\in \G*_\Lambda=VX$, so we deduce from Theorem \ref{thm:compactbydiscrete}(\ref{item:compactbydiscrete}) and (\ref{item:QIfinitestabs}) that $\Aut(X)$ is not compact-by-discrete, and that $\G*_\Lambda$ is not graphically discrete.
\end{proof}

If $\Lambda \leq \Gamma$ in the presentation above is the centralizer of an element $g \in \Gamma$, such an HNN-extension is known as a {\it centralizer extension}; thus such a group is graphically discrete if the element $g$ is not in the center of $\Gamma$. 

We use Proposition \ref{prop:HNNnotGD} to classify the graphically discrete right-angled Artin groups.

\begin{cor} \label{cor:RAAG}
    The right-angled Artin group $A_\G$ is graphically discrete if and only if $\G$ is a complete graph.
\end{cor}
 \begin{proof}
     If $\G$ is complete, then $A_\G$ is free abelian, and graphical discreteness follows from Theorem \ref{thm:propc_examples}(\ref{propc:trofimov}).
     Otherwise, there exists $v \in \Gamma$ so that $\{v\} \cup \lk(v) \neq V\Gamma$. Thus, $A_\G$ can be expressed as a nontrivial HNN extension
     \[A_\G=A_{\G-\{v\}}*_{A_{\lk(v)}},\] 
     where the generator $v$ is the stable letter as in Proposition \ref{prop:HNNnotGD}. Thus, the group $A_\G$ is not graphically discrete.
 \end{proof}
 
 Proposition \ref{prop:HNNnotGD} also yields many examples of non-graphically-discrete infinite-ended groups (beyond right-angled Artin groups).
 
 \begin{cor} \label{cor:nongdfreeprod}
     For all nontrivial finitely generated groups $\G$, the free product $\G*\Z$ is not graphically discrete.
 \end{cor}

\begin{cor} \label{cor:nonvirtgd.freeprod}
	If $\Gamma_1,\Gamma_2$ are nontrivial finitely generated  groups containing nontrivial proper finite-index subgroups, the free product $\Gamma=\Gamma_1*\Gamma_2$ has a finite-index subgroup that is not graphically discrete.
\end{cor}
\begin{proof}
		 Since the groups $\Gamma_1$ and $\Gamma_2$ contain proper finite-index subgroups, there exists a finite-index subgroup $\hat{\G} \leq \G$ whose Grushko decomposition contains a nontrivial free factor that is a free group.		
		 For example, let $Y$ be the graph of spaces associated to $\Gamma_1*\Gamma_2$. By taking covers of the vertex spaces of $Y$, there exists a finite cover $\hat Y\to Y$ such that the underlying graph of $\hat Y$ is not a tree.   Then $\hat\G\coloneqq \pi_1(\hat Y)$ is a finite-index subgroup of $\G$ containing a nontrivial free factor that is a free group. By Corollary \ref{cor:nongdfreeprod}, $\hat \G$ is not graphically discrete.
\end{proof}
 
 Yet more examples of non-graphically-discrete infinite-ended groups are given by the following proposition.
 
 \begin{prop}\label{prop:notgd_fp}
     Let $\G_1,\G_2$ be nontrivial finitely generated groups. Then the free product $\G=\G_1*\G_2*\G_2$ is not graphically discrete.
   \end{prop}
 \begin{proof}
     Let $S_1,S_2$ be finite generating sets for $\G_1,\G_2$ respectively. Let $\G'_2$ be a copy of $\G_2$ and let $S'_2$ be the copy of $S_2$ in $\G'_2$. We can then write $\G=\G_1*\G_2*\G'_2$, and consider the Cayley graph $X$ of $\G$ with respect to the finite generating set $S=S_1\sqcup S_2\sqcup S'_2$.
     Now construct a graph $Y$ from $X$ by replacing each vertex $x\in VX$ with a tripod $T_x$. That is, let $x_1,x_2,x'_2$ denote the endpoints of $T_x$, and for each edge incident to $x$ in $X$ labeled by a generator in $S_1$ (resp. $S_2,S'_2$) we connect this edge to $x_1$ (resp. $x_2,x'_2$) in $Y$.
     There is a natural geometric action of $\G$ on $Y$.
     Assuming $|S_1|,|S_2|>1$, the centers of the tripods are the only degree 3 vertices in $Y$, so all automorphisms of $Y$ preserve the collection of tripods. 
     It is clear from the symmetry of the construction that for each $x\in VX$, there is an automorphism of $Y$ that fixes $x_1$ and swaps $x_2,x'_2$ -- call this the \emph{flip at $T_x$}.
     
     By collapsing the edges labeled by $S_1,S_2,S'_2$ we get an $\Aut(Y)$-equivariant map to a tree $p:Y\to T$ (this is just a Bass--Serre tree for the splitting $\G=\G_1*\G_2*\G'_2$).
     The flip at $T_x$ induces an automorphism of $T$ that fixes $p(x_1)$ and swaps $p(x_2),p(x'_2)$.
     Given a vertex $y\in Y$, it is clear that the $\Aut(Y)$-stabilizer of $y$ contains infinitely many flips, so it has arbitrarily large vertex orbits in both $Y$ and $T$. It then follows from Theorem \ref{thm:compactbydiscrete}(\ref{item:compactbydiscrete}) and (\ref{item:Abounded}) that $\Aut(Y)$ is not compact-by-discrete, hence $\G$ is not graphically discrete.
 \end{proof}
We remark that Corollary \ref{cor:nongdfreeprod} and Proposition \ref{prop:notgd_fp} demonstrate that infinite-ended groups are frequently not graphically discrete, and so action rigidity theorems for such groups, for instance  Theorem \ref{thm:actionrigidA}, cannot be deduced cheaply via Proposition \ref{prop:gdimpliesvi}.

Due to work of Francoeur, another source of non-graphically-discrete groups are groups generated by bireversible Mealy automata. This class of groups includes free groups \cite{glasnermozes05automata}, groups of the form $A\wr\Z$ for $A$ a nontrivial finite abelian group \cite{francoeur23bireversiblelamplighter}, and some groups with Kazhdan's property (T) \cite{glasnermozes05automata}.

\begin{thm}(\cite[Theorem 4.1]{francoeur2025bireversible}).\label{thm:bireversible}
    Infinite bireversible groups are not graphically discrete.
\end{thm}

    \subsection{Uniform lattices and quasi-actions}\label{sec:qi} 
    In this subsection, we first summarize a construction due to Furman  relating uniform lattice embeddings and quasi-actions \cite[\S 3.2]{furman2001mostow}.
    We conclude the subsection by showing that certain finitely generated groups that have finite index in their quasi-isometry group are graphically discrete. 
    Recall that by Lemma~\ref{lem:finitevscompact} if a locally compact group $G$ contains a finitely generated uniform lattice modulo finite kernel, then $G$ is compactly generated. 

\begin{defn}
   Let $\Gamma$ be a finitely generated group, $G$ a locally compact group, and $\rho:\Gamma\to G$ a uniform lattice embedding modulo finite kernel. Equip $\Gamma$ and $G$ with word metrics $d_{\Gamma}$ and $d_G$ with respect to a finite and compact generating set, respectively (see Lemma \ref{lem:finitevscompact}). The map $\rho:\Gamma \rightarrow G$ is then a quasi-isometry with coarse inverse $\overline{\rho}$~\cite[Proposition 5.C.3]{cornulierdlH2016metric}.
   The \emph{induced quasi-action} of $G$ on $\Gamma$ 
   is the collection of maps $\{f_g\}_{g\in G}$, where $f_g:\Gamma\to\Gamma$ is the quasi-isometry given by $f_g=\overline{\rho} \circ L_g\circ \rho$, where $L_g$ is left multiplication by $g$. 
\end{defn}

    \begin{rem}
        Although the topology on $G$ induced by the metric $d_G$ does not coincide with the topology on $G$ (unless $G$ is discrete), this  metric  is \emph{adapted} in the sense of \cite{cornulierdlH2016metric}. That is, the metric $d_G$ is left-invariant, proper, and locally bounded, i.e.\ every point in $G$ has a neighborhood of finite diameter. In particular, the following properties are satisfied:
    \begin{enumerate}
    	\item[(P1)] $L \subseteq G$ has compact closure if and only if it is bounded in $(G,d_G)$;
    	\item[(P2)] there is some sufficiently large $R$ such that for every $g\in G$, the open ball $B_R(g)$ contains an open neighborhood of $g$. 
    \end{enumerate}
    \end{rem}

   The next lemma relates the topology of $G$ to the geometry of the induced quasi-action. 
      
   \begin{lem}\label{lem:coarse_conv}
   	Let  $\rho:\Gamma\to G$ be as above, and let $\{f_g\}_{g\in G}$ be the induced quasi-action of $G$ on $\Gamma$. There is a constant $B\geq 0$ such that if a sequence $(g_i)_{i \in \N}$ converges to $g$ in $G$, then for all $x\in \Gamma$, there is some $N_x \in \N$ such that  \[d_{\Gamma}(f_{g_i}(x),f_g(x))\leq B\] for all $i\geq N_x$.
   \end{lem}
\begin{proof}
	Pick $K \geq 1$ and $A \geq 0$ such that $\{f_g\}_{g\in G}$ is a $(K,A)$-quasi-action, $\rho$ and $\overline{\rho}$ are $(K,A)$-quasi-isometries, and the compositions $\overline{\rho}\rho$ and $\rho\overline{\rho}$ are $A$-close to the identity. Without loss of generality, we may assume $\rho(1_\Gamma)=1_G$ and $\overline\rho(1_G)=1_\Gamma$.
		Since $\{f_g\}_{g\in G}$ is a quasi-action, it is sufficient to prove the lemma in the case where $g$ is the identity. So, let $(g_i)_{i \in \N}$ be a sequence in $G$ converging to $1_G$. 

    Using Property~(P2) above, pick a constant $R$ sufficiently large such that $B_R(1_G)$ contains an open neighborhood of $1_G$. Let $x\in \Gamma$. As $g_i$ converges to the identity, so does $\rho(x)^{-1}g_i\rho(x)$. Therefore $d_G(\rho(x)^{-1}g_i\rho(x),1_G)\leq R$ for all $i$ sufficiently large.  It follows that \begin{align*}
			d_{\Gamma}(f_{g_i}(x),x)&=d_{\Gamma}(\overline{\rho}(g_i\rho(x)), x)\leq K d_G(g_i \rho(x), \rho(x))+2A \\&= Kd_G(\rho(x)^{-1}g_i\rho(x), 1_G)+2A\leq KR+2A
		\end{align*}
	for all $i$ sufficiently large.
\end{proof}

We can also describe compact normal subgroups of  $G$ geometrically:
\begin{lem}\label{lem:cpct_normal_qaction}
		Let  $\rho:\Gamma\to G$ be as above, and let $\{f_g\}_{g\in G}$ be the induced quasi-action of $G$ on $\Gamma$. If $K\lhd G$ is a compact normal subgroup, there is a constant $B$ such that for all $k\in K$, $\sup_{x\in \Gamma}d(f_k(x),x)\leq B$.
\end{lem}
\begin{proof}
	The Milnor--Schwarz lemma for locally compact groups \cite[Theorem 4.C.5]{cornulierdlH2016metric} says that the action of $G$ on the Cayley graph of $G/K$ is quasi-conjugate the action of $G$ on its Cayley graph.  Therefore, the action of $G$ on $G/K$ is quasi-conjugate to the induced quasi-action $\{f_g\}_{g\in G}$ of  $G$ on $\Gamma$. Since $K$ acts trivially on $G/K$, the result follows.
\end{proof}

       \begin{thm}\label{thm:discrete_qi}
    	Let $\Gamma$ be a finitely generated tame group such that the image of the homomorphism $\phi:\Gamma\to \QI(\Gamma)$   induced by left-multiplication is a finite-index subgroup of $\QI(\Gamma)$. Then every uniform lattice embedding modulo finite kernel of $\Gamma$ is trivial. In particular, $\Gamma$ is graphically discrete.
    \end{thm}
\begin{proof}
     Suppose $\rho:\Gamma\to G$ is a uniform  lattice embedding modulo finite kernel. By Remark \ref{rem:2ndcountable} we can restrict to the case where $G$ is second countable.
    Let $\psi:G \rightarrow \QI(\Gamma)$ be the homomorphism induced by the quasi-action $\{f_g\}_{g\in G}$ of $G$ on $\Gamma$, where $\psi(g) = [f_g]$. We will show $\psi$ has compact open kernel. 
    Equip $\QI(\Gamma)$ with the word metric with respect to a finite generating set. The natural homomorphism $\phi:\Gamma \rightarrow \QI(\Gamma)$, which has finite-index image by assumption and finite kernel since $\Gamma$ is tame, is then a quasi-isometry. We note the composition $\psi\circ \rho$ coincides with $\phi$. 
  
    To set notation, pick $K$, $A$, $R$ and $B$ as in the proof and statement of Lemma \ref{lem:coarse_conv}.  	
	Let $F=\{[f_1],\dots,[f_n]\}\subseteq \QI(\Gamma)$ be a  set of right $\phi(\Gamma)$ representatives.
	    Increasing $K$ and $A$ if necessary, we can assume each $f_i$ is a $(K,A)$-quasi-isometry. Thus every element of $\QI(\Gamma)$ can be represented by a $(K,A)$-quasi-isometry of the form $\phi(h)\circ f_i$ for some $h\in \Gamma$ and $1\leq i\leq n$. Let $C = C(K,A)$ be the tameness constant as in Definition \ref{defn:tame}.

    We first show $\ker(\psi)$ is open, which implies $\psi$ is continuous. Since $G$ is second countable, it is sufficient to show that if $(g_i)$ is a sequence in $G$ converging to $1_G$, then $g_i\in \ker(\psi)$ for $i$ sufficiently large.  By Lemma~\ref{lem:coarse_conv}, $d_{\Gamma}(f_{g_i}(1_{\Gamma}),1_{\Gamma})\leq B$ for all $i$ sufficiently large. There are sequences $(h_i)$ in $\Gamma$  and $(k_i)$ in $\{1,\dots, n\}$ such that $[f_{g_i}]=[\phi(h_i)\circ f_{k_i}]$ for all $i \in \N$.  Let $R_1=\max_{i=1}^n d_{\Gamma}(f_i(1_\Gamma),1_\Gamma)$. 
    Then,        
    \begin{align*}
		d_{\Gamma}(h_i,1_\Gamma) &\leq 
        d_{\Gamma}(h_i, h_if_{k_i}(1_\Gamma))+d_\Gamma(h_if_{k_i}(1_\Gamma),1_\Gamma)  = 
        d_\Gamma(1_\Gamma, f_{k_i}(1_\Gamma))+d(h_if_{k_i}(1_\Gamma),1_\Gamma)\\
        &\leq R_1+d_\Gamma(h_if_{k_i}(1_\Gamma),1_\Gamma)   \leq d_{\Gamma}(f_{g_i}(1_\Gamma),1_\Gamma)+R_1+C\leq B+R_1+C    
	\end{align*}     
for all $i$ sufficiently large. Therefore, $\{\psi(g_i)\mid i\in \bN\}$  is finite. 
If $\psi(g_i)$ is not eventually the identity, we can replace $(g_i)_{i \in \N}$ with a subsequence such that $\psi(g_i)$ is constant and equal to some  $[f]\neq 1_{\QI(\Gamma)}$. Without loss of generality, we may assume $f$ is a $(K,A)$-quasi-isometry. Since $ [f]$ is not the identity, for all $r\geq 0$, there is some $x\in \Gamma$ such that $d_{\Gamma}(f(x),x)>r$. In particular, there is some $x\in \Gamma$ such that $d_{\Gamma}(f(x),x)>B+C$.  Since $[f_{g_i}]=[f]$, $d_\Gamma(f_{g_i}(x),f(x))\leq C$. Therefore, $d_\Gamma(f_{g_i}(x),x)>B$ for all $i$, contradicting  the choice of $B$. 

It remains to show $\ker(\psi)$ is compact.  Let $g\in \ker(\psi)$. Then $d_\Gamma(f_g(1_\Gamma),1_\Gamma)\leq C$, so \[d_G(g,1_G)\leq d_G(\rho(\overline{\rho}(g)),1_G)+A\leq Kd_\Gamma(f_g(1_\Gamma),1_\Gamma)+2A\leq KC+2A.\] Since the word metric on $G$ is adapted, $\ker(\psi)$ has compact closure by Property~(P1). Since $\ker(\psi)$ is the kernel of a continuous map, it is closed. Thus, $\ker(\psi)$ is compact.  \end{proof}

\begin{question}
    Can the tameness assumption be removed from the previous theorem? That is, if $\Gamma$ is a finitely generated group such that the image of the homomorphism $\phi:\Gamma\to \QI(\Gamma)$  induced by left-multiplication is a finite-index subgroup of $\QI(\Gamma)$, is every lattice embedding modulo finite kernel of $\Gamma$ trivial?
\end{question}

\section{Graphical discreteness and non-discreteness for hyperbolic groups}\label{sec:altgraphdiscrete}

\subsection{Boundary characterization of graphical discreteness}

       The next theorem is particularly useful for exhibiting groups that are not graphically discrete. For example, it follows immediately from the theorem that non-abelian free groups are not graphically discrete. We refer to \cite{bridsonhaefliger} for background on boundaries of hyperbolic groups. If $\Gamma$ is a hyperbolic group, we always assume $\Homeo(\partial \Gamma)$ is equipped with the topology of uniform convergence.
    
    \begin{thm} \label{thm:boundary_char}
     A finitely generated non-elementary hyperbolic group $\Gamma$ is graphically discrete if and only if for every geometric action of $\Gamma$ on a connected locally finite graph $X$ and for all vertices $x \in X$, the image of the induced action of the stabilizer $(\Aut(X))_{x} \rightarrow \Homeo(\p X)$ is finite. 
    \end{thm}
	\begin{proof}
		Suppose $\Gamma$ acts geometrically on a connected locally finite graph $X$, and let $G = \Aut(X)$. The induced map $G\to \Homeo(\p X)$ naturally factors through $\QI(X)$, where the induced map $\QI(X)\hookrightarrow\Homeo(\p X)$ is injective \cite[Corollary 11.115]{DrutuKapovich18}. Thus the image of $G_x$ in $\QI(X)$ is finite if and only if the image of $G_x$ in $\Homeo(\p X)$ is finite. The result follows from Proposition \ref{prop:propC} and Theorem \ref{thm:compactbydiscrete}.
	\end{proof}

    We refer to \cite{ronan_buildings, davis_buildings}  for background and the definition of hyperbolic buildings. 

    \begin{cor} \label{cor:buildings}
         Uniform lattices in thick locally finite hyperbolic buildings are not graphically discrete. In particular, non-abelian free groups are not graphically discrete.
    \end{cor}
    \begin{proof} Let $X$ be a thick locally finite hyperbolic building. By definition, each (compact) chamber of $X$ is contained in infinitely many (non-compact) apartments, and for any two such apartments there is an automorphism of the building fixing the chamber pointwise and interchanging the apartments. Each apartment is quasi-convex, hence its limit set embeds in the boundary of $X$. Moreover, distinct apartments have distinct limit sets. Thus, a uniform lattice in the isometry group of the building is not graphically discrete by Theorem~\ref{thm:boundary_char}.
    \end{proof}

A \emph{JSJ decomposition} of a finitely presented group is a graph of groups decomposition encoding how a group splits over a prescribed family of subgroups. In this article, we only consider JSJ decompositions of one-ended hyperbolic groups over two-ended subgroups, and we only consider the  JSJ decomposition described by Bowditch \cite{bowditch}. The associated Bass--Serre tree, called the \emph{JSJ tree}, is  canonically determined by the  local cut point structure of the boundary. The JSJ tree is said to be {\it nontrivial} if the tree has more than one vertex (and hence infinitely many vertices).  We refer the reader to Bowditch's article for details of this \cite{bowditch}, and to the monograph of Guirardel--Levitt  for a reference on  JSJ decompositions more generally \cite{guirardel2017jsj}.

Let $\Gamma$ be a one-ended  finitely generated hyperbolic group with nontrivial  Bowditch JSJ tree $T$.  
We equip  $\Aut(T)$  with the topology of pointwise convergence and note that in non-degenerate cases, the tree $T$ is not locally finite. Using the description of $T$ in terms of the topology of the boundary, we show the following:
\begin{prop}\label{prop:bdry_jsj_injective}
	If $\Gamma$ is a one-ended hyperbolic group with nontrivial JSJ tree $T$, there exists a continuous monomorphism $\Psi:\Homeo(\partial \Gamma)\to \Aut(T)$ extending the action of $\Gamma$ on $T$.
\end{prop}
\begin{proof}
	We first summarize the construction of Bowditch's tree.
	The tree $T$ is defined  from a collection $\Omega$ of non-empty subsets  of $\partial \Gamma$. In the terminology of \cite{bowditch}, $\Omega$ is the collection of $\sim$-classes and $\approx$-pairs, although familiarity with this terminology is not needed here.  The set $\Omega$ is equipped with a certain \emph{betweenness relation}, where $\omega$ is between $\omega_1$ and $\omega_2$ if there exist $a,b\in \omega$, $c\in \omega_1$ and $d\in \omega_2$ such that $c$ and $d$ lie in different components of $\partial \Gamma\setminus\{a,b\}$. The set $\Omega$ and this betweenness relation  depend only on the  local cut point structure of $\partial \Gamma$, hence $\Homeo(\partial \Gamma)$ acts on the set $\Omega$  preserving the betweenness relation.

	 The tree $T$ is constructed in a combinatorial way from $\Omega$ such that elements of $\Omega$ are a subset of vertices of $T$, and the betweenness relation on $T$ agrees with that of $\Omega$. There is no proper subtree of $T$ containing $\Omega$, hence   any  bijection of $\Omega$  preserving the betweenness relation induces a unique automorphism of $T$. Thus there is a canonical homomorphism $\Psi:\Homeo(\partial \Gamma)\to \Aut(T)$ determined by the action of $\Homeo(\partial \Gamma)$ on $\Omega$. The action of  $\Gamma$ on its JSJ tree $T$ is induced by the action of $\Gamma$ on $\partial \Gamma$, hence $\Psi$ extends the action of $\Gamma$.  All this is contained in \cite{bowditch}.
		 
		 We now show $\Psi$ is continuous. Let $(\phi_i)_{i\in I}$ be a net %
		 in $\Homeo(\partial \Gamma)$ converging to the identity.  We show  for each $\omega\in \Omega$, there is some $i_0\in I$ such that $\phi_i(\omega)=\omega$ for all $i\geq i_0$. By the above discussion, this will imply $\Psi(\phi_i)$ converges pointwise on $T$. 
		 Since the tree $T$ is discrete and has no leaves, we can pick $\omega_1,\omega_2\in \Omega$ such that $\omega$ is the unique element of $\Omega$ between $\omega_1$ and $\omega_2$. Thus for all $i\in I$,  $\phi_i(\omega)$ is the unique element of $\Omega$ between $\phi_i(\omega_1)$ and $\phi_i(\omega_2)$.
		 
		  Since $\omega$ is  between $\omega_1$ and $\omega_2$, there exist $a,b\in \omega$, $c\in \omega_1$ and $d\in \omega_2$ such that $c$ and $d$ lie in different components of $\partial \Gamma\setminus\{a,b\}$.	As $(\phi_i)_{i\in I}$ converges to the identity, there exists $i_0\in I$ such that $\phi_i(c)$ and $\phi_i(d)$ lie in different components of $\partial \Gamma\setminus\{a,b\}$ for all $i\geq i_0$. Hence $\omega$ is between  $\phi_i(\omega_1)$ and $\phi_i(\omega_2)$ for all $i\geq i_0$. Therefore, $\phi_i(\omega)=\omega$ for all $i\geq i_0$ as required, showing $\Psi$ is continuous.

	We now show $\Psi$ is injective. We fix $\phi\in \ker(\Psi)$. Pick  an infinite order element $g\in\Gamma$ fixing an edge $e$ of $T$. Then  $\omega=\{g^+,g^-\}$ is the cut pair in $\partial \Gamma$ corresponding to $e$; see \cite{bowditch}. Then for all $h\in \Gamma$,  $\phi$ must stabilize the edge $he$, hence $\phi(h\omega)=h\omega$. 
		We claim that for every open $U\subseteq \partial \Gamma$, there is some $h\in \Gamma$ with $h\omega\subseteq U$. 
	Let $U\subseteq \partial \Gamma$. There is some infinite order $h\in \Gamma$ such that $\{h^+,h^{-}\}$ is disjoint from $\{g^+,g^-\}$ \cite[8.2.G]{gromov1987hyperbolic} and $h^+\in U$. Therefore $h^i(g^+),h^i(g^-)\to h^+$ as $i\to \infty$. Thus $h^i\omega\in U$ for $i$ sufficiently large, proving the claim.
	If $z\in \partial \Gamma$, let $(U_i)$ be a countable neighborhood basis of $z$. By the preceding claim, there is a  sequence $(h_i)$ in  $\Gamma$ such that $h_i\omega\in U_i$. Then $(h_i g^+)$ and $(\phi (h_i g^+))$ both converge to $z$, thus $\phi(z)=z$. Since $\phi(z)=z$ for all $z\in \partial \Gamma$, $\phi=\id$. Thus $\Psi$ is injective.
\end{proof}

\begin{cor} \label{cor:vertex_stab_jsj_alt}
	Let $\Gamma$ be a finitely generated one-ended hyperbolic group $\Gamma$, with nontrivial Bowditch JSJ tree $T$. Then $\Gamma$  is graphically discrete if and only if for every geometric action of $\Gamma$ on a connected locally finite graph $X$ and for all vertices $x \in X$, the image of the induced action of the stabilizer $(\Aut(X))_{x} \rightarrow \Aut(T)$ is finite.
\end{cor}
\begin{proof}
	We consider the composition $\Aut(X)\to \Homeo(\partial X)\xrightarrow{\psi} \Aut(T)$. The result follows from Theorem \ref{thm:boundary_char} and Proposition \ref{prop:bdry_jsj_injective}.
\end{proof}

\subsection{Hyperbolic groups with sphere and Sierpinski carpet boundaries}    
    
    \begin{thm}\label{thm:graphical_discrete_sphere_boundaries}
	If $\Gamma$ is a hyperbolic group with boundary an $n$-sphere for $n\leq 3$, then $\Gamma$ is graphically discrete.
\end{thm}

To prove this, we recall the Hilbert--Smith Conjecture:
\begin{conj}
	If a locally compact group $G$ acts faithfully and continuously by homeomorphisms on a connected $n$-manifold, then $G$ is a Lie group.
\end{conj}
It is well-known that this conjecture reduces to showing $p$-adic integers cannot act faithfully on a connected $n$-manifold for any prime $p$. Although the general conjecture is  open, it is known in the case $n=1,2$; see  Pages 233 \& 249 of \cite{montgomeryzippin1955topological}. The $n=3$ case  was  shown by Pardon  \cite{pardon2013hilbertsmith}. 
\begin{thm}[\cite{montgomeryzippin1955topological,pardon2013hilbertsmith}]\label{thm:hilbert_smith}
The  cases of the Hilbert--Smith Conjecture where $n\leq 3$ are true.
\end{thm}
We use Theorem \ref{thm:hilbert_smith} to prove Theorem \ref{thm:graphical_discrete_sphere_boundaries}. We note that the same proof will work for any  $n$ in which  a positive solution to  the $n$-dimensional   Hilbert--Smith Conjecture is  known.
\begin{proof}[Proof of Theorem \ref{thm:graphical_discrete_sphere_boundaries}]
	Let $\Gamma$ be a hyperbolic group with boundary an $n$-sphere for some $n\leq 3$. Suppose $\Gamma$ acts geometrically on a connected locally finite graph $X$. Let $\rho:\Aut(X)\to \Homeo(S^n)$ be the induced action on the boundary of $S^n$. Then  $\rho$ is continuous with compact kernel $K$; see  \cite[Theorem 3.5]{furman2001mostow}.
		As $\Aut(X)/K$  acts faithfully and continuously on $S^n$ by homeomorphisms, it is a Lie group by Theorem \ref{thm:hilbert_smith}. Since $\Aut(X)$ is totally disconnected,  Lemma \ref{lem:totdisc_to_lie} implies $\ker(\rho)$ is open. Thus $\Aut(X)$ contains a compact open normal subgroup, hence $\Gamma$ is graphically discrete.
\end{proof}

    \begin{thm} \label{thm:graphical_discrete_sier_boundaries}
		If $\Gamma$ is a hyperbolic group with visual boundary homeomorphic to a Sierpinski carpet, then $\Gamma$ is graphically discrete.     
    \end{thm}
    \begin{proof}
		Let $\Gamma$ be a hyperbolic group with visual boundary homeomorphic to the Sierpinski carpet $\cS$. Suppose $\Gamma$ acts geometrically on a connected locally finite graph $X$. As above, there is a continuous homomorphism $\rho:\Aut(X) \rightarrow \Homeo(\cS)$ with compact kernel. Moreover, there is a continuous  homomorphism $\phi:\Homeo(\cS) \rightarrow \Homeo(S^2)$ with trivial kernel. Indeed, the map $\phi$ extends a homeomorphism of the Sierpinski carpet to a homeomorphism of the 2-sphere by extending the map on the peripheral circles to maps on disks; see \cite[Theorem 2.1]{jmelsalhivago}.     As in the above proof of Theorem \ref{thm:graphical_discrete_sphere_boundaries}, this implies $\Gamma$ is graphically discrete.
    \end{proof}

 \subsection{Graphical discreteness of direct products of hyperbolic groups}

\begin{thm}\label{thm:dirprod_hyp}
	The direct product of finitely many graphically discrete non-elementary hyperbolic groups is graphically discrete. 
\end{thm}

To prove this, we require the following lemma:
\begin{lem}[\cite{kapovich1998derahm}]\label{lem:open_productpreserving}
	Let $\Gamma=\Pi_{i=1}^n \Gamma_i$ be the product of non-elementary hyperbolic groups, and let $\rho:\Gamma\to G$ be a uniform lattice embedding modulo finite kernel. Let  $\{f_g\}_{g\in G}$ be the induced quasi-action of $G$ on $\Gamma$.
	Then there is a finite-index open subgroup $G^*\leq G$ such that all elements of the restricted quasi-action $\{f_g\}_{g\in G^*}$  split up to uniform error as a product of $n$ quasi-isometries.
\end{lem}
\begin{proof}
	This is an application of work of Kapovich--Kleiner--Leeb. Pick $K \geq 1$ and $A \geq 0$ such that each $f_g$ is a $(K,A)$-quasi-isometry.  By \cite{kapovich1998derahm}, there are constants $K_1\geq 1$ and $A_1\geq 0$ such that for every $(K,A)$-quasi-isometry $f:\Gamma\to \Gamma$ and $1\leq i\leq n$, there is a permutation $\sigma_f\in \Sym(n)$ and a $(K_1,A_1)$-quasi-isometry $f_i:\Gamma_i\to \Gamma_{\sigma_f(i)}$ such that the following diagram commutes up to an error of at most $A_1$:
	\[\begin{tikzcd}
		\Gamma & \Gamma \\
		{\Gamma_i} & {\Gamma_{\sigma_f(i)}}
		\arrow["f"', from=1-1, to=1-2]
		\arrow["{\pi_i}", from=1-1, to=2-1]
		\arrow["{\pi_{\sigma_f(i)}}", from=1-2, to=2-2]
		\arrow["{f_i}"', from=2-1, to=2-2]
	\end{tikzcd},\]
	where the $\pi_i:\Gamma\to \Gamma_i$ are projections. The projections $\pi_i$ are all $(K_0,A_0)$-coarse Lipschitz for large enough $K_0$ and $A_0$. In particular, $f$ sends the fiber $\pi_i^{-1}(x)$ to the fiber $\pi_{\sigma_f(i)}^{-1}(f_i(x))$ up to uniform Hausdorff distance.  Since any fibers of the form $\pi_i^{-1}(x)$ and $\pi_j^{-1}(y)$  are at finite Hausdorff distance if and only if $i=j$, it follows that $\sigma_f$ is well-defined and that $\sigma_{f\circ g}=\sigma_f\circ \sigma_g$. Let $G^*$ be the kernel of the map $G\to \Sym(n)$ given by $g\mapsto \sigma_{f_g}$. By construction, for each $g\in G^*$, each $f_g$ splits as a product of $n$-quasi-isometries up to error $A_1$. 

    It remains to prove $G^*$ is open.
	By Remark~\ref{rem:2ndcountable}, after replacing $G$ with $G/K$ for some compact normal $K$ if necessary, we may assume that $G$ is second countable.     
    Pick $B$ such that the conclusion of Lemma~\ref{lem:coarse_conv} is true with respect to the induced quasi-action $\{f_g\}_{g\in G}$. Let $g_m$ be a sequence in $G$ that converges to the identity. It suffices to prove that $\sigma_{f_{g_m}}=\id$ for all $m$ sufficiently large. Indeed, then $G^*$ is an identity neighborhood because it is first countable and  any sequence converging to $\id$ is eventually contained in $G^*$. Then, since $G^*$ is a subgroup, it is open.
	
	For each $1\leq i\leq n$, we pick $x,y\in \Gamma$ with $\pi_i(x)=\pi_i(y)$ and $d(\pi_j(x),\pi_j(y))>D$ for all $i\neq j$, where $D=A_1 + 2(K_0B+A_0)$. Pick $M$ such that $d(f_{g_m}(x),x), d(f_{g_m}(y),y)\leq B$ for all $m\geq M$. Suppose for contradiction that for some $m\geq M$, we have $\sigma_f(i)\neq i$, where $f=f_{g_m}$. Then setting $j=\sigma_f(i)$, we have 
	\begin{align*}
		d(\pi_j(x),\pi_j(y))&\leq d(\pi_j(x),\pi_j(f(x)))+ d(\pi_j(f(x)),\pi_j(f(y))) +d(\pi_j(f(y)),\pi_j(y))\\
		& \leq (K_0B+A_0)+ \Bigl(d\bigl(f_i(\pi_i(x)),f_i(\pi_i(y))\bigr)+ A_1 \Bigr) + (K_0B+A_0)\\
		&=A_1 + 2(K_0B+A_0)=D.
	\end{align*}
    The first and third terms on the second line follow from the choice of $B$ and that the maps $\pi_j$ are coarse Lipschitz. The middle term follows from the coarse commutativity of the diagram above. Since $\pi_i(x) = \pi_i(y)$, the equality on the third line holds. This inequality contradicts our choice of $x$ and $y$. We conclude that $G^*$ is an open finite-index subgroup of $G$.
\end{proof}

\begin{proof}[Proof of Theorem \ref{thm:dirprod_hyp}]
Suppose $\Gamma=\Pi_{i=1}^n \Gamma_i$ is product of graphically discrete non-elementary hyperbolic groups and that $\rho:\Gamma\to \Aut(X)$ is a geometric action on a locally finite graph $X$. Let $G = \Aut(X)$. We will show $G$ is compact-by-discrete and apply Proposition~\ref{prop:propC}(\ref{item:propC4}).

Replacing $\Gamma$ with its quotient by the finite normal subgroup $\ker(\rho)$, we can assume $\rho$ is injective and hence $\Gamma$ can be identified with a subgroup of  $G$.  As in Section~\ref{sec:qi}, there is a quasi-isometry $h:X\to \Gamma$ quasi-conjugating the action of $G$ on $X$ to the induced quasi-action $\{f_g\}_{g\in G}$  of $G$ on $\Gamma$. We proceed with an argument similar to the proof of \cite[Theorem 6.1]{margolis2022discretisable}. 

Let $G^*\leq G$ be as in Lemma~\ref{lem:open_productpreserving} - i.e. for every $g\in G^*$, $f_g$ splits as a product up to uniform error. 
 We have that  $\rho(\Gamma)\leq G^*$ and that  there are quasi-actions of $G^*$ on each of $\Gamma_1,\dots, \Gamma_n$. Restricting to $\rho(\Gamma_i)\cong \Gamma_i$ we recover the natural map $\Gamma_i\to \QI(\Gamma_i)$ induced by left multiplication, which does not fix a point of $\p \Gamma_i$. Thus the quasi-action of $G^*$ on $\Gamma_i$ does not fix a point of $\p \Gamma_i$ either. 

There is a quasi-isometry $f_i:\Gamma_i\to Y_i$ that quasi-conjugates the quasi-action of $G^*$ on $\Gamma_i$  to an isometric action of $G^*$ on $Y_i$, where $Y_i$ is either a rank one symmetric space or a locally finite graph by \cite[Corollary 5.7]{margolis2022discretisable} since the $\Gamma_i$ are hyperbolic. We now consider the composition \[h_i:X\xrightarrow{h}\Gamma\xrightarrow{\pi_i}\Gamma_i\xrightarrow{f_i}Y_i,\] where $\pi_i$ is the projection map. The map $h_i$ is a coarse surjective coarse-Lipschitz map that is coarsely $G$-equivariant. We can thus apply Lemma \ref{lem:continuous_quasiconj} to deduce that   the composition $\psi_i:G^*\to \Isom(Y_i)/K_i$ is continuous,  where $K_i\lhd \Isom(Y_i)$ is the maximal compact normal subgroup. Let $H_i\leq \Isom(Y_i)/K_i$ be the closure of the image of $\psi_i$. We claim that $H_i$ is discrete. Indeed, in the case where $Y_i$ is a symmetric space, this follows from Lemma \ref{lem:totdisc_to_lie}, since $H_i$ is a Lie group and  $G^*$ is totally disconnected. Now suppose  $Y_i$ is a locally finite graph.  Since $f_i$ quasi-conjugates the geometric action of $\Gamma_i$ on itself  to a  geometric action of $\Gamma_i$ on $Y_i$, we deduce $\Gamma_i$ is a uniform lattice modulo finite kernel in $\Isom(Y_i)$. As $\Gamma_i$ is graphically discrete and $\Isom(Y_i)$ is totally disconnected with  $K_i$ the maximal compact normal subgroup,  it follows $H_i\leq \Isom(Y_i)/K_i$ is discrete.
 
Let $f=(f_1,\dots,f_n):\Gamma=\Pi_{i=1}^n \Gamma_i\to \Pi_{i=1}^n Y_i$ be the product map. Then $f\circ h:X\to \Pi_{i=1}^n Y_i$ quasi-conjugates the geometric action of $G^*$ on $X$  to the isometric action of $G^*$ on $\Pi_{i=1}^n Y_i$.  It follows that the product map  $\psi=(\psi_1,\dots,\psi_n):G^*\to \Pi_{i=1}^n \Isom(Y_i)/K_i$, which we already know  is continuous  with discrete image, has compact open kernel $K$. Since $G^*\leq G$ is a finite-index open subgroup, $K$ is open in $G$, and so the intersection $\cap_{g\in G}gKg^{-1}$ is a finite-index subgroup of $K$ that is a compact open normal subgroup of $G$. It follows that $G$ is compact-by-discrete as required.
\end{proof}

 \section{Graphical discreteness is not a commensurability invariant}

   \begin{defn}
        A manifold is a \emph{minimal element} in its commensurability class if it does not nontrivially cover any orbifold. 
    \end{defn}
    
        We assume throughout the paper that $3$-manifolds are connected. 

    \begin{rem} \label{rem:minelts} (Existence of manifold minimal elements.) We thank Genevieve Walsh for explaining the following examples. First, let $N$ be a hyperbolic knot complement, and let $\Gamma = \pi_1(N)$. Margulis~\cite{margulis} proved that if $\Gamma$ is non-arithmetic, then $\Gamma$ has finite index in its commensurator, $\Comm_{\PSL(2,\C)}(\Gamma) \leq \PSL(2, \C)$. So, $\Hy^3/\Comm_{\PSL(2,\C)}(\Gamma)$ is the minimal orbifold in the commensurability class of $N$ in this case. Showing that $\Gamma = \Comm_{\PSL(2,\C)}(\Gamma)$ is equivalent to showing $\Gamma$ has no symmetries (i.e.\ $\Gamma$ is equal to its normalizer) and no hidden symmetries (i.e.\ its normalizer is equal to its commensurator). The figure-eight knot complement is the only arithmetic knot complement~\cite{reid91}. For hyperbolic knots with up to 15 crossings, the only examples with hidden symmetries are the figure-eight knot and two dodecahedral knots~\cite{goodmanheardhodgson, chesebrodebloishoffmanmillichapmondalworden2023dehn}. See the survey~\cite{walsh} for more details. For example, the knots $K_1=9\_32$ and $K_2=9\_33$ with nine crossings in Rolfson's Knot Table are hyperbolic knots with no symmetries, as verified on KnotInfo. Thus, $N = S^3-K_i$ is a minimal knot complement in its commensurability class. 
    
    Sufficiently large Dehn filling on these knot complements yields closed hyperbolic manifolds that are minimal elements in their commensurability class, as shown in \cite[Proposition 2.4]{chesebrodebloishoffmanmillichapmondalworden2023dehn}.
    \end{rem}

	The next lemma illustrates the advantage of working with manifolds rather than orbifolds.     
    
    \begin{lem} \label{lemma_aut_dis}
        Let $M$ be a  closed hyperbolic $3$-manifold that is a minimal element in its commensurability class. Let $x \in \Hy^3$, and let $\Aut(\Hy^3, \pi_1(M)\cdot x)$ denote the subgroup of $\Isom(\Hy^3)$ that stabilizes the set $\pi_1(M)\cdot x$. Then $\Aut(\Hy^3, \pi_1(M) \cdot x) = \pi_1(M)$. In particular, if $g \in \Aut(\Hy^3, \pi_1(M) \cdot x)$ fixes $x$, then $g$ is trivial.
    \end{lem}
    \begin{proof}
        The group  $\Aut(\Hy^3, \pi_1(M)\cdot x)$ acts geometrically on $\Hy^3$ and contains $\pi_1(M)$. Hence, these groups are commensurable. 
        Since $M$ is a minimal element in its commensurability class, we deduce that $\Aut(\Hy^3, \pi_1(M) \cdot x) = \pi_1(M)$.
        The final sentence follows since $M$ is a manifold: the only element in $\pi_1(M)$ fixing $x$ is the identity. 
    \end{proof}
	
	\begin{thm} \label{thm:no_fi_subgroup}
		 Let $M$ and $M'$ be closed hyperbolic $3$-manifolds that are minimal elements in their commensurability classes. Then $\Gamma = \pi_1(M)* \pi_1(M')$ has no nontrivial lattice embeddings, but has a finite-index subgroup that does have nontrivial lattice embeddings. In particular, $\Gamma$ is graphically discrete but contains a finite-index subgroup that is not. 
	\end{thm}
	\begin{proof}

		Since $\pi_1(M)$ and $\pi_1(M')$ are residually finite, they contain  proper finite-index subgroups. By Corollary \ref{cor:nonvirtgd.freeprod}, $\Gamma$ contains a finite-index subgroup that is not graphically discrete, so it suffices to show $\Gamma$ has no nontrivial lattice embeddings.		As $\Gamma$ is hyperbolic and not virtually free, it is not virtually isomorphic to a uniform lattice in a connected rank-one semisimple Lie group. Therefore,  Proposition \ref{prop:non_uniform_lattices} ensures that all lattice embeddings of $\Gamma$ are uniform.
			Suppose $\Gamma$ acts geometrically on a proper quasi-geodesic metric space $X$. We will prove $\Isom(X)$ is compact-by-discrete. 
			
			By \cite[Proposition 4.5]{starkwoodhouse2024action}, the $\Isom(X)$-action on $X$ is quasi-conjugate to an $\Isom(X)$-action on an \emph{ideal model geometry} $Z$, which is a tree of spaces with vertex spaces either a point or a copy of $\Hy^3$ and edge spaces that are points. Let  $\rho:\Isom(X)\to \Isom(Z)$ be the action on $Z$, and let $f:X\to Z$ be this quasi-conjugacy. Since $Z$ is tame (see Example~\ref{exmp:tame_morse}),  Lemma \ref{lem:continuous_quasiconj} implies the composition $\Phi:\Isom(X)\xrightarrow{\rho} \Isom(Z)\to \Isom(Z)/K$ is continuous, where $K$ is the maximal compact normal subgroup. (In fact, $K$ is trivial, although we do not need this.) Since $f:X\to Z$ is a quasi-conjugacy and $K$ is compact, $\ker(\Phi)$ is a closed subgroup of $\Isom(X)$ with bounded orbits, hence is compact by Proposition \ref{prop:compact<->bounded}.
			
			We now show $\Isom(Z)$ is discrete, which  implies that $\Isom(X)$  is compact-by-discrete.
			   As $\Gamma$ acts geometrically on~$Z$,  there are two orbits of one-ended vertex spaces, and by Mostow rigidity, the quotient of one-ended vertex space by the corresponding vertex stabilizer is isometric to either $M$ or $M'$. 
			After rescaling the length of the edge, $\Gamma \backslash Z$ is isometric to $Y=\bigl( M \sqcup M' \sqcup [-1,1] \bigr)/\sim$, where  $\sim$ identifies $-1$ to  some $x \in M$ and $1$ to some $x' \in M'$. 
			Let $\widetilde{x}$ be a lift of $x \in Y$ in a copy $\widetilde{M}_0$ of $\widetilde{M}$. If $g \in \Isom(Z)$ fixes $\widetilde{x}$, then $g$ fixes $\widetilde{M}_0$ by Lemma~\ref{lemma_aut_dis}. Hence, $g$ fixes all intervals attached to lifts of $x$ in $\widetilde{M}_0$. Continuing inductively, we see that $g$ is the identity. 		
		Thus, $\Isom(Z)$ is discrete.		
	\end{proof}

 \section{Simple surface amalgams} \label{sec:ssa}

 In this section we prove Theorem \ref{thm:ssa_char}, characterizing graphically discrete simple surface amalgams.  
 
    \begin{notation}
        Let $\G = \pi_1(Z)$ be the fundamental group of a simple surface amalgam, where~$Z$ is the union of surfaces $\Sigma_1, \ldots, \Sigma_k$ which each have a single boundary component and negative Euler characteristic. The decomposition of $Z$ into surfaces yields a decomposition as a graph of spaces, where the underlying graph $\Lambda$ consists of vertices $V\Lambda = \{u_0, u_1, \ldots, u_k\}$ and edges $E\Lambda = \{ e_i,\bar e_i\mid 1 \leq i \leq k \text{ and } e_i=(u_i, u_0) \}$. 
        The vertex spaces are $Z_{u_0} \cong S^1$ and $Z_{u_i} \cong \Sigma_i$ for $1 \leq i \leq k$. Each edge space is $S^1$. 
        
        This decomposition gives the JSJ decomposition of $\G$ in the sense of \cite{bowditch}. Let $T$ denote the JSJ tree, and let $\rho: T \to \Lambda$ denote the map given by quotienting by $\G$. There are two types of vertices in $T$:
        \begin{enumerate}
        	\item infinite-valence vertices, denoted $V_{1}T$, that correspond to the hanging Fuchsian vertex groups;
        	\item  valence-$k$ vertices, denoted $V_2T$, that correspond to the cosets of the cyclic subgroup along which  the surface groups are amalgamated.
        \end{enumerate} 
        So, $\rho(V_2T)=\{u_0\}$ and $\rho(V_1 T)=\{u_1,...,u_k\}$.   Note that for all $w\in VT\sqcup ET$, the stabilizer   $\G_w$ is a  quasi-convex subgroup of $\G$ (\cite[Proposition 1.2]{bowditch}), so $\partial \G_w\subseteq \partial \G$.
    \end{notation}

    \begin{prop} \label{prop:ssa_onlyif}
        If $\chi(\Sigma_i) = \chi(\Sigma_j)$ for some $i \neq j$, then $\G$ is not graphically discrete.
    \end{prop}
    \begin{proof}
        Suppose without loss of generality that $\chi(\Sigma_1) = \chi(\Sigma_2)$. We can give the space $Z$ the structure of a simplicial complex such that $\Sigma_1$ and $\Sigma_2$ are isomorphic and there is an automorphism of $Z$ flipping $\Sigma_1$ and $\Sigma_2$ and fixing the rest of the complex. For each $v \in V_2T$, this automorphism lifts to an elliptic automorphism of the universal cover $h_v: \widetilde Z \to \widetilde Z$ that fixes the vertex space $\widetilde{Z}_v$ and induces the transposition on $\lk(v)$ that fixes all edges except $\widetilde e_1, \widetilde e_2 \in \lk(v)$ where $\rho(\widetilde e_i) = e_i$. The existence of these flipping automorphisms implies that if $x \in \widetilde{Z}_v^{(0)}$ for some $v \in V_2T$ and $G = \Aut\bigl(\widetilde{Z}^{(1)}\bigr)$, then the image of the stabilizer $G_x$ in $\Aut(T)$ is infinite. Hence, by Corollary~\ref{cor:vertex_stab_jsj_alt}, the group $\G$ is not graphically discrete. 
    \end{proof}

\begin{prop}\label{prop:ssa_if}
	If  $\chi(\Sigma_i)\neq \chi(\Sigma_j)$ for all $i\neq j$, then every lattice embedding modulo finite kernel into a locally compact group is trivial. Hence $\Gamma$ is graphically discrete.
\end{prop}
		   In the proof of Proposition \ref{prop:ssa_if}, we make use of the following elementary lemma: 
\begin{lem}\label{lem:index}
If $G$ is a group and $A, B, C$ are subgroups with $B\leq C$, then:
	\begin{align*}
		[AC:AB]=[C:(A\cap C)B]\tag{\dag} \label{eqn:index}
	\end{align*} provided the products in (\ref{eqn:index}) are all subgroups.
\end{lem}
\begin{proof}
	We consider the function $\phi:\frac{C}{(A\cap C)B}\to \frac{AC}{AB}$ given by $c(A\cap C)B\mapsto cAB$. Since $B\leq C$,  $C\cap AB=(A\cap C)B$, which shows $\phi$ is well-defined and injective. Surjectivity of $\phi$ follows from the fact that as $AC$ is a subgroup, $AC=CA$.
\end{proof}

\begin{proof}[Proof of Proposition \ref{prop:ssa_if}]
    Let $\Gamma$ be the fundamental group of a simple surface amalgam so that $\chi(\Sigma_i) \neq \chi(\Sigma_j)$ for some $i \neq j$. Suppose $\rho:\Gamma \rightarrow G$ is a lattice embedding modulo finite kernel into a locally compact group. We will show $\rho$ is trivial in the sense of Definition~\ref{def:VIembeddings}. This will imply that $\Gamma$ is graphically discrete by Proposition~\ref{prop:propC}(\ref{item:propC3}). 
    As $\Gamma$ is torsion-free, the map $\rho$ is injective. As $\Gamma$ is hyperbolic and not virtually free, it is not virtually isomorphic to a non-uniform lattice in a connected simple rank one Lie group.  Proposition \ref{prop:non_uniform_lattices} implies $\rho$ is uniform. By Remark \ref{rem:2ndcountable}, we can assume that $G$ is second countable.  We identify $\Gamma$ with its image under $\rho$. 
 
 Consider the induced quasi-action  of $G$ on $\Gamma$, which gives a continuous homomorphism $\phi:G\to \Homeo(\partial \Gamma)$ with compact kernel by \cite[Theorem 3.5]{furman2001mostow}. If $T$ is the Bowditch JSJ tree of $\Gamma$, there is a continuous  injective map $\Psi:\Homeo(\partial \Gamma)\to \Aut(T)$ by Proposition \ref{prop:bdry_jsj_injective}. 	 	
	We note that the action of $G$ on $T$ extends the natural action of $\Gamma$ on $T$, and that $G$ is compact-by-(totally disconnected).
		As $G$ acts continuously on $T$, for each $w\in VT\sqcup ET$, the stabilizer $G_w$ is open. Therefore,  Lemma \ref{lem:open_lattice} implies each $\Gamma_w$ is a uniform lattice in $G_w$. In particular, if $w$ has infinite valence, $G_w$ acts continuously on $\partial \Gamma_w$ with compact kernel~\cite[Theorem 3.5]{furman2001mostow}. For each edge $e\in ET$, as  $\Gamma_e\cong \bZ$ is graphically discrete, the group $G_e$ contains a maximal compact open subgroup $K_e\vartriangleleft G_e$ by Remark \ref{rem:gd_maxcmpctopen}. We will show $K_e=K_f$ for all $e,f\in ET$, which proves $K_e$ is a compact open normal subgroup of $G$ as required (see Remark \ref{rem:triv<->c-by-d}).
		
		Firstly, suppose $v\in V_1T$ and $e\in \lk(v)$. Recall  $\partial \Gamma_v \subseteq \partial \Gamma$ is a cyclically ordered Cantor set and that $G_v$ preserves this cyclic ordering by \cite[Lemma 3.2]{bowditch}.  Moreover, the pair of points in $\partial \Gamma_e$ are adjacent in this cyclic ordering. As $K_e$ is compact and every edge stabilizer is open, for any finite $R\subseteq \lk(v)$ containing $e$, the orbit $K_eR\subseteq \lk(v)$ consists of finitely many edges. Since $K_e$ preserves the cyclic order   $\cup_{f\in K_eR} \partial \Gamma_f$ and $K_e$ fixes the adjacent pair of points $\partial \Gamma_e$ pointwise, $K_e$ fixes $\cup_{f\in K_eR} \partial \Gamma_f$ pointwise.   Thus, $K_e$ fixes  $\cup_{f\in \lk(v)} \partial \Gamma_f$. Since $\cup_{f\in \lk(v)} \partial \Gamma_f$ is dense in $\partial \Gamma_v$, the group $K_e$ fixes $\partial \G_v$. 
        Letting $K_v$ be the kernel of the action of $G_v$ on $\partial \Gamma_v$, we see that $K_e \leq K_v$ for all $e\in \lk(v)$. Since $K_v \leq G_e$, we have $K_e = K_v$ by maximality of $K_e$. 
        In particular, $K_e=K_f$ for all  $e,f\in \lk(v)$.
		
		Now suppose $v\in V_2T$. We claim every $g\in G_v$ fixes $\lk(v)$. This will imply $G_v=G_e$ for all $e\in \lk(v)$, and hence  $K_e=K_f$ for all  $e,f\in \lk(v)$. Combined with the previous paragraph, it will then follow that $K_e=K_f$ for all $e,f\in ET$  as required. Suppose for contradiction there is some $g\in G_v$ and $e\in \lk(v)$ with $ge\neq e$. Let  $\iota(e)=w\in V_1T$ be the hanging Fuchsian vertex incident to $e$. We will show that $\chi(\Gamma_w)=\chi(\Gamma_{gw})$, contradicting our hypotheses.

		As noted above, $K_w=K_e$ is a compact  open normal subgroup of $G_w$.  Since $\Gamma_w$ and $g^{-1}\Gamma_{gw}g$ are torsion-free uniform lattices  in $G_w$, they inject to finite-index subgroups of the discrete quotient $G_w/K_w$. %
		   As Euler characteristic is multiplicative, in order to show $\chi(\Gamma_w)$ and $\chi(g^{-1}\Gamma_{gw}g)=\chi(\Gamma_{gw})$ are equal, it is sufficient to show $[G_w:\Gamma_wK_w]=[G_{w}:(g^{-1}\Gamma_{gw}g)K_{w}]$. We emphasize that as $g$ need not normalize $\Gamma$, the subgroups $\Gamma_w$ and $g^{-1}\Gamma_{gw}g$ are not typically equal.

We first note that $\Gamma_v=\Gamma_e=\Gamma_{ge}\cong \bZ$.
Define $K_v=\cap_{f\in\lk(v)} K_f$, which is compact and normal in $G_v$.
Since $K_e$ is compact and $\G_e\cong\Z$ is discrete and torsion-free,  the intersection $\G_e\cap K_e$ is trivial.  Lemma \ref{lem:index} now implies  $[\G_e K_e:\G_e K_v]=[K_e:K_v]$.
Therefore,
\begin{align*}
	[G_v:G_e][G_e:\G_e K_e][K_e:K_v]&=[G_v:G_e][G_e:\G_e K_e][\G_e K_e:\G_e K_v]\\&=[G_v:\G_e K_v]=[G_v:\G_v K_v].
\end{align*}
Running the same argument with $ge$ instead of $e$, we get
$$[G_v:G_{ge}][G_{ge}:\G_{ge} K_{ge}][K_{ge}:K_v]=[G_v:\G_v K_v].$$
Conjugating by $g$, we have $[G_v:G_e]=[G_v:G_{ge}]$ and $[K_e:K_v]=[K_{ge}:K_v]$, so we deduce that $[G_e:\G_e K_e]=[G_{ge}:\G_{ge} K_{ge}]$.

 As $\Gamma_w$ acts transitively on edges in $\lk(w)$, we have $G_w=\Gamma_wG_e$. Similarly, $G_{gw}=\Gamma_{gw}G_{ge}$. Therefore, applying Lemma \ref{lem:index} twice, we have
	\begin{align*}
		[G_w:\Gamma_wK_w]&=[\Gamma_wG_e:\Gamma_wK_e]=[G_e:(\Gamma_w\cap G_e)K_e]=[G_e:\Gamma_eK_e]=[G_{ge}:\Gamma_{ge}K_{ge}]\\
		&=[G_{ge}:(\Gamma_{gw}\cap G_{ge}) K_{ge}]=[\Gamma_{gw}G_{ge}:\Gamma_{gw}K_{ge}]=[G_{gw}:\Gamma_{gw}K_{gw}].
	\end{align*}
		It follows that $[G_w:\Gamma_wK_w]=[G_{w}:g^{-1}\Gamma_{gw}gK_{w}]$ as required.
\end{proof}

Combining Propositions \ref{prop:ssa_onlyif} and \ref{prop:ssa_if} gives:
\begin{thm}\label{thm:ssa_char}
	 If $\Gamma$ is the fundamental group of a simple surface amalgam, the following are equivalent:
	\begin{enumerate}
		\item $\Gamma$ is graphically discrete;
		\item $\Gamma$ has no nontrivial lattice embeddings;
		\item the surfaces in the amalgam have pairwise-distinct Euler characteristics.
	\end{enumerate}
\end{thm}

Two simple surfaces amalgams are quasi-isometric if and only if the number of surfaces in the unions are equal~\cite{malone}. Thus, we have the following corollary. 

\begin{cor}
    Graphical discreteness is not a quasi-isometry invariant for one-ended groups. 
\end{cor}

 \section{Graphical discreteness of closed \texorpdfstring{$3$-}{3-}manifold groups}

\subsection{Geometric 3-manifolds}
In this section we show the following:
\begin{thm}\label{thm:propc-3manifolds}
	The fundamental group of a closed  geometric 3-manifold is graphically discrete.
\end{thm}

\begin{rem}
	Note that  $\mathbb{Z}\times F_2$ is  simultaneously a  uniform lattice in $\mathbb{Z}\times \Aut(T)$ and a non-uniform lattice in $\mathbb{R}\times \PSL_2(\bR)$. Thus, it is the fundamental group of a  compact geometric  3-manifold with non-empty toroidal boundary that is \emph{not} graphically discrete. In particular, Theorem~\ref{thm:propc-3manifolds} does not generalize to fundamental groups of  geometric 3-manifolds with boundary.
\end{rem}

\begin{prop}\label{prop:sol_c}
	If $\Gamma$ is a lattice in the Lie group $\textbf{Sol}$, then $\Gamma$ is graphically discrete.
\end{prop}
\begin{proof}
	By \cite[Theorems 3.1 \& 4.28]{raghunathan1972discrete} lattices in $\textbf{Sol}$ are always uniform and torsion-free.
	Suppose $G$ is a totally disconnected locally compact group containing $\Gamma$ as a lattice.  It is sufficient to show $G$ is compact-by-discrete. By \cite[Theorem 1]{dymarz2015envelopes}, there is a continuous map $\rho:G\rightarrow \Isom(\textbf{Sol})$ with compact kernel $K$. It follows from Lemma \ref{lem:totdisc_to_lie} that $K$ is open, hence  $G$ is compact-by-discrete. 
\end{proof}

We will now show a finitely generated group quasi-isometric to $\Hy^2 \times \R$ is graphically discrete. 
We need the following theorem concerning quasi-actions on $\mathbb{H}^2\times \mathbb{R}$:
\begin{prop}[{\cite[\S 3,4,6]{kleinerleeb2001symmetric}}]\label{prop:quasi-action_project}
	If there is a cobounded quasi-action of a group $G$ on $\mathbb{H}^2\times \mathbb{R}$, then there is an isometric action $\rho:G\rightarrow \Isom(\mathbb{H}^2)$ 
	such that \[\sup_{x\in \mathbb{H}^2\times \mathbb{R},g\in G}(\pi(gx),\rho(g)\pi(x))<\infty,\] 
	where $\pi:\mathbb{H}^2\times \mathbb{R}\rightarrow \mathbb{H}^2$ is the projection.
	Moreover, when $G$ is finitely generated and the quasi-action on $\mathbb{H}^2\times \mathbb{R}$ is proper, then $\ker(\rho)$ is two-ended.
\end{prop}

\begin{prop}\label{prop:H2timesr_propc}
	A finitely generated group quasi-isometric to $\Hy^2\times \mathbb{R}$ is graphically discrete.
\end{prop}
\begin{proof}
	Suppose $\Gamma$ is a finitely generated group quasi-isometric to $\mathbb{H}^2\times \mathbb{R}$ and $\Gamma$ acts geometrically on a connected locally finite graph $X$. We prove that $G=\Isom(X)$ is compact-by-discrete. 
	
	The Milnor--Schwarz lemma implies there is a quasi-isometry $f:X\rightarrow \mathbb{H}^2\times \mathbb{R}$ that quasi-conjugates the isometric action of $G$ on $X$ to a cobounded quasi-action of $G$ on $\mathbb{H}^2\times \mathbb{R}$. Proposition \ref{prop:quasi-action_project} implies there is an isometric action $\rho:G\rightarrow \Isom(\mathbb{H}^2)$ such that the coarse Lipschitz map $\pi\circ f$ is $G$-equivariant up to uniformly bounded error. Since $\mathbb{H}^2$ has no nontrivial bounded isometries, the map $\rho$ is continuous by Corollary \ref{cor:cts_quasiaction}.
	
	By Lemma \ref{lem:totdisc_to_lie}, $\ker(\rho)$ is open. Since $\Gamma$ is a lattice in $G$, Lemma \ref{lem:open_lattice} implies   $\Gamma \cap \ker(\rho)$ is a lattice in $\ker(\rho)$. As the quasi-action $G\curvearrowright \mathbb{H}^2\times \mathbb{R}$  restricts to a proper cobounded quasi-action $\Gamma \curvearrowright \mathbb{H}^2\times \mathbb{R}$, Proposition \ref{prop:quasi-action_project} implies $\Gamma \cap \ker(\rho)$ is two-ended. Since $\ker(\rho)$ is totally disconnected and contains a uniform lattice that is a  two-ended subgroup, $\ker(\rho)$ has a unique maximal compact open  normal  subgroup $K$ with quotient infinite cyclic or infinite dihedral \cite[Corollary 19.39]{cornulier2018quasiisometric}. Since $K$ is a  topologically characteristic subgroup of $\ker(\rho)$, it is a  normal compact open subgroup of $G$. Hence $G$ is compact-by-discrete.
\end{proof}

\begin{rem}
	The same argument holds without much modification for all the groups in \cite{kleinerleeb2001symmetric}.
\end{rem}

Since  $\mathbb{H}^2\times \mathbb{R}$ and $\widetilde{\SL_2(\bR)}$ are quasi-isometric \cite{gersten1992bounded},  Proposition \ref{prop:H2timesr_propc} implies:
\begin{cor}\label{cor:latticesol}
	A cocompact lattice in either $\PSL_2(\R)  \times \mathbb{R}$ or $\widetilde{\SL_2(\bR)}$ is graphically discrete.
\end{cor}

 \begin{proof}[Proof of Theorem \ref{thm:propc-3manifolds}]
 	Recall (see \cite[Figure 4.22]{thurston1997threedimensional}) that the fundamental group of a  closed geometric three-manifold is either:
 	\begin{enumerate}
 		\item virtually nilpotent ($S^3$, $S^2\times \mathbb{R}$, $\mathbb{R}^3$ or  $\textbf{Nil}$),
 		\item a cocompact lattice in $\textbf{Sol}$,
 		\item a cocompact lattice in $PSL_2(\R) \times \mathbb{R}$ or $\widetilde{\SL_2(\bR)}$, or
 		\item a cocompact lattice in $\mathbb{H}^3$.
 	\end{enumerate}
 	 The virtually nilpotent and hyperbolic fundamental groups are graphically discrete by Theorem~\ref{thm:propc_examples}. The remaining groups are graphically discrete by Proposition \ref{prop:sol_c} and Corollary \ref{cor:latticesol}.
 \end{proof}

 \subsection{Non-geometric 3-manifolds}
 In this section we use a  result of  Kapovich--Leeb to classify lattice embeddings of non-geometric 3-manifolds \cite{KapovichLeeb97}. 
 
Let $M$ be a closed irreducible, oriented $3$-manifold with a nontrivial geometric decomposition. 
As a consequence of geometrization (see Theorem 1.14 in~\cite{AschenbrennerFriedlWilton15}) the JSJ tori give a decomposition of $M$ into Seifert fibered and hyperbolic components.
By taking tubular neighborhoods of the JSJ tori we obtain a graph of spaces decomposition $f: M \to Y$ where $Y$ is a finite graph.
The vertex spaces $M_v = f^{-1}(v)$ are the Seifert fibered or hyperbolic components, and the edge spaces $M_e = f^{-1}(m_e)$, where $m_e$ is the midpoint of an edge, are homeomorphic to $ \mathbb{T}^2$.
Let $T$ denote the associated JSJ tree, which is the Bass--Serre tree associated to the graph of spaces decomposition.  The universal cover $\tilde{M}$ is a graph of spaces over $T$ with vertex and edge spaces denoted $\tilde M_v$ and $\tilde M_e$.

 Let $\Gamma=\pi_1(M)$. 	We first note that $\Gamma$ has no non-uniform lattice embeddings. Indeed, $\Gamma$ is acylindrically hyperbolic \cite[Corollary 2.9]{minasyanosin2015acylindrical}. Since $\Gamma$ is non-geometric, it is not isomorphic to a non-uniform lattice in a rank one simple Lie group, hence Proposition~\ref{prop:non_uniform_lattices} implies it has no non-uniform lattice embeddings.
 Now suppose $\rho :\Gamma\to G$ is a  uniform lattice embedding. By Remark \ref{rem:2ndcountable}, we assume without loss of generality that $G$ is second countable. We identify $\Gamma$ with its image in $G$. The uniform lattice embedding $\rho$ induces a  quasi-action of $G$ on $\Gamma$ as in Section \ref{sec:qi}. Since $\Gamma$ acts geometrically  on $\tilde{M}$, this induces a quasi-action
 $\{f_g\}_{g\in G}$  of $G$ on $\tilde{M}$. 
 
 \begin{lem} \label{lem:nongeocont}
 	The quasi-action $\{f_g\}_{g\in G}$  of $G$ on $\tilde{M}$ induces a continuous action of $G$ on $T$ such that for all $g\in G$,  the Hausdorff distance between $f_g(\tilde M_v)$ and $\tilde M_{gv}$ is finite for all $v\in VT$. Moreover, $G$ is compact-by-(totally disconnected).
 \end{lem}
 \begin{proof}
 	The main theorem of \cite{KapovichLeeb97} proves that the quasi-action $\{f_g\}_{g\in G}$ of $G$ on $\tilde{M}$ induces an action  $\sigma:G\to \Aut(T)$ such that  $f_g( \tilde M_v)$ and $\tilde M_{\sigma(g)(v)}$ are at uniform finite Hausdorff distance  for all $v\in VT$ and $g\in G$.  To see this action is continuous, let $(g_i)$ be a sequence in $G$ converging to the identity. By Lemma \ref{lem:coarse_conv}, there is a constant $B$ such that for each $x\in \tilde{M}$,  $d(f_{g_i}( x) ,x)\leq B$ for all $i$ sufficiently large. Since no two distinct vertex spaces of $\tilde{M}$ are at finite Hausdorff distance, for each $v\in VT$, the element $\sigma(g_i)$ fixes $v$  for $i$ sufficiently large; hence, the action is continuous.	
 	
 	Since the action of $\Gamma$ on $T$ is acylindrical \cite[Lemma 2.4]{wiltonzalesskii2010profinite},  there exist edges $e$ and $e'$ such that $\Gamma_{e}\cap \Gamma_{e'}=\Gamma\cap (G_e\cap G_{e'})$ is finite. Since the open subgroup $G_{e}\cap G_{e'}$ contains a finite group as a uniform lattice by Lemma \ref{lem:open_lattice}, it is compact. Thus, $G$ contains a compact open subgroup, hence is compact-by-(totally disconnected) by Lemma~\ref{lem:vanDantzig}.
 \end{proof}

We now show fundamental groups of closed irreducible non-geometric  $3$-manifolds have no nontrivial lattice embeddings. The proof is motivated by the observation that if $M$ is such a manifold, then the isometry group of the universal cover of $M$ is discrete; in particular, if an isometry fixes a copy of the universal cover of a JSJ torus, then it fixes the entire space. The main difficulty in the proof is that if $G$ is a locally compact group containing $\pi_1(M)$ as a uniform lattice, then it need not be the case that $G$ acts on $\tilde{M}$.

\begin{thm} \label{thm:nongeo3mangd}
       Let $M$ be a closed irreducible, oriented $3$-manifold with a nontrivial geometric decomposition. Let $G$ be a locally compact group containing $\pi_1(M)$ as a  lattice. Then $G$ contains a compact open normal subgroup $K$ such that $G/K$ is isomorphic to the fundamental group of a compact $3$-orbifold  covered by $M$. In particular, $\pi_1(M)$ has no nontrivial lattice embeddings, hence is graphically discrete.
\end{thm}

\begin{proof}
     We begin by outlining the proof. Let $M$ be as in the statement of the theorem, and let $G$ be a locally compact group containing $\Gamma \coloneqq \pi_1(M)$ as a lattice.  As noted above, $\Gamma$ has no non-uniform lattice embeddings, so $G$ contains $\G$ as a uniform lattice. We will show $G$ has a compact open normal subgroup $K$ and is therefore graphically discrete by Proposition~\ref{prop:propC}(\ref{item:propC1}). We then apply arguments analogous to those of Kapovich--Leeb~\cite[\S 5.2]{KapovichLeeb97}. To exhibit such a $K$, we will use an action of $G$ on the JSJ tree, and prove each vertex and edge stabilizer $G_v$ and $G_e$ of~$G$ has a maximal compact open normal subgroup, $K_v$ and $K_e$, respectively, so that $K_v = K_e$ if $e$ is incident to $v$. Taking $K= K_v$ will then give the desired subgroup.

 The group $G$ is compact-by-(totally disconnected) and acts continuously on the JSJ tree $T$ by Lemma~\ref{lem:nongeocont}. Each vertex stabilizer $G_v$ is open and compact-by-(totally disconnected). Since $\Gamma$ is a uniform lattice in $G$ and $G_v$ is open, $\Gamma_v\coloneqq \Gamma\cap G_v$ is a uniform lattice in $G_v$ by Lemma \ref{lem:open_lattice}.

First suppose $v$ is a vertex corresponding to a hyperbolic component. Then, $\Gamma_v$ is graphically discrete by Theorem~\ref{thm:propc_examples}, so $G_v$ is compact-by-discrete. The discrete quotient $Q_v$ of $G_v$ contains $\Gamma_v$ as a uniform lattice; hence, $Q_v$ is quasi-isometric to $\Gamma_v$. By \cite{Schwartz95},  $Q_v$ acts on $\bH^3$ by isometries with image a non-uniform lattice and with finite kernel. Without loss of generality, replace $Q_v$ by its quotient by this finite kernel. Then, let $K_v$ be the kernel of the homomorphism $G_v\to Q_v$. Since $Q_v$ is discrete without finite normal subgroups, $K_v$ is the maximal compact normal open subgroup of $G_v$.

We will now show if $e$ is an edge incident to $v$, then $K_v$ is the unique maximal compact normal subgroup of $G_e \leq G_v$. The discrete quotient $Q_v$ of $G_v$ acts geometrically on truncated hyperbolic space $\Omega_v \subseteq \Hy^3$. So, $G_v$ acts on $\Omega_v$ with kernel $K_v$. There is a one-to-one correspondence between the stabilizers of the set of horospheres of $\Omega_v$ under the $G_v$-action and the set of edge stabilizers $G_e \leq G$ with $e$ incident to $v$. Thus, if $e$ is incident to $v$, then the edge group $G_e$ acts geometrically on a horosphere. Any isometry of $\Omega_v$ fixing this horosphere fixes $\Omega_v$, hence $K_v$ is precisely the kernel of the action of $G_e$ on this horosphere.
By \cite[Example 2.E.23.5]{cornulierdlH2016metric}, $K_v$ is contained in a unique maximal compact normal subgroup of $G_e$ (the compact radical of $G_e$).
Since $G_e/K_v$ is a  finitely generated group acting geometrically and faithfully on $\mathbb{E}^2$, it  has no nontrivial finite normal subgroups. Therefore, $K_v$ is the unique maximal compact normal subgroup of $G_e$. 

Now suppose $v$ is a  vertex corresponding to a Seifert fibered component, and let $\tilde M_v=\Sigma_v\times \bR$ be a Seifert fibered vertex space, where $\Sigma_v$ is a convex subset of $\bH^2$ whose boundary is a union of disjoint geodesics.  Since $\Gamma_v$ acts geometrically on $\tilde M_v$ and is a uniform lattice in $G_v$, we deduce that $G_v$ admits a cobounded  quasi-action on $\tilde M_v$. 

To find the desired compact open normal subgroup, we employ work of Kapovich--Leeb~\cite[\S 5.2]{KapovichLeeb97}. We summarize the relevant points in their argument, and refer to their paper for details. By reflecting along boundary flats, the cobounded action of $G_v$ on $\Sigma_v\times \bR$ extends to a quasi-action of $G_v$ on $\bH^2\times \bR$. This quasi-action preserves fibers of the form $\{x\}\times \bR$ and so descends to a quasi-action of $G_v$ on $\bH^2$ \cite[Proposition 5.4]{KapovichLeeb97}. Using work of Tukia~\cite{tukia} and the line pattern coming from $\p \Sigma_v$, the authors prove the quasi-action on $\Hy^2$ can be quasi-conjugated to an isometric discrete cocompact action on $\bH^2$.  They conclude there is a map $\phi_v:G_v\to \Isom(\bH^2)$ whose image is a discrete non-uniform lattice $Q_v$ in $\bH^2$ where $Q_v$ is the fundamental group of a 2-orbifold $S_v$. 

Since the image of $\phi:G_v \rightarrow \Isom(\Hy^2)$ is discrete, the kernel of $\phi_v$ is an open subgroup~$R_v$. The group  $R_v\cap \Gamma_v$  is the kernel of the restriction $\phi_v|_{\Gamma_v}$, which is a finitely generated two-ended group.  Since $\bZ$ is graphically discrete and   $R_v\cap \Gamma_v$ is a lattice in $R_v$, there is a compact open normal subgroup $K_v\vartriangleleft R_v$ with quotient $L_v$ isomorphic to $\bZ$ or $D_\infty$. 
Since neither $Q_v$ nor $L_v$ has finite normal subgroups, $K_v$ is a maximal compact normal open subgroup of $G_v$. 
Therefore $G_v/K_v$ fits into the short exact sequence \[1\to L_v\to G_v/K_v\to Q_v=\pi_1(S_v)\to 1\]
and so $G_v/K_v$ is the fundamental group of a Seifert fibered 3-orbifold $O_v$. Hence $G_v$ acts on $\widetilde{O}_v$ with kernel $K_v$.
If $e$ is an edge incident to $v$, then $G_e$ is precisely the  stabilizer of the corresponding  boundary component of $\widetilde{O}_v$. The kernel of the action of $G_e$ on the corresponding boundary flat is precisely $K_v$, since any isometry of $\widetilde{O}_v$ fixing a boundary flat is trivial. Therefore $K_v$ is the unique maximal compact normal subgroup of $G_e$ as in the hyperbolic case.

We have shown that a maximal compact open normal subgroup of each vertex group $G_v$ is equal to the maximal compact open normal subgroup of each incident edge group.
Thus $K_v=K_{v'}$ for all vertices, and $K=K_v$ is a compact open normal subgroup of $G$.
By the exact same reasoning as in \cite[\S 5.3]{KapovichLeeb97} (or alternatively, applying \cite[Theorem 1.2]{KapovichLeeb97} to $G/K$), we see $G/K$ is the fundamental group of a compact non-geometric 3-dimensional orbifold $O$.
The image of $\Gamma$ under the composition  $\Gamma\xrightarrow{\rho}G \to G/K\cong \pi_1(O)$ is a finite-index subgroup, hence $\pi_1(M)$ is isomorphic to a finite-index subgroup of $\pi_1(O)$. Therefore $M$ is homeomorphic to a finite cover of $O$ by Theorem 2.1.2 of \cite{AschenbrennerFriedlWilton15}.
\end{proof}

\section{Blowups of groups acting on trees}\label{sec:blowups}

    This section provides a general construction to ``blowup'' the action of a locally compact group on a tree to a geometric action  on a locally finite graph with a tree of spaces decomposition. The results will be utilized in this paper to show that if two finitely generated groups act on the same proper metric space that admits a tree of spaces decomposition with certain properties, then the groups act geometrically on the same simplicial complex admitting a tree of spaces decomposition in a particularly nice way (see Theorem~\ref{thm:get_treeofspaces}).

\begin{defn}\label{defn:blowup}
	Let $G$ be a locally compact group acting continuously and minimally on a tree $T$. A \emph{blowup} of $G\curvearrowright T$ is a surjective map $p:X\to T$ such that
	\begin{enumerate}
		\item $X$ is a locally finite connected graph, and $G$ acts geometrically on $X$;
		\item $p$ is $G$-equivariant and simplicial with respect to the first barycentric subdivision of $T$;
		\item  for all $v\in VT$,   $X_v\coloneqq p^{-1}(v)$ is a connected subgraph of $X$ which we call a \emph{vertex space};
		\item for all $e\in ET$, if $m_e$ is the midpoint of $e$, then $X_e\coloneqq p^{-1}(m_e)$ is a connected subgraph of $X$ which we call an \emph{edge space}.
	\end{enumerate}
\end{defn}

The following proposition can be thought of an interweaving of the construction of the Cayley--Abels graph of a tdlc group, and the construction of the Bass--Serre tree from a graph of groups decomposition.

\begin{prop}[Blowup existence] \label{prop:blowups_exist} 
		Let $G$ be a locally compact group acting continuously, minimally and cocompactly on a tree $T$. Suppose all vertex and edge stabilizers of $T$ are compactly generated and $G$ contains a compact open subgroup.
		Then a blowup $p:X\to T$ of $G\curvearrowright T$ exists.
	\end{prop}
\begin{proof}
    Let $\Omega_0\subseteq VT\sqcup ET$ be a set of representatives of the  $G$-orbits of vertices and edges of $T$.     
     Let $K$ be a compact open subgroup of $G$. We first show we can replace $K$ by a compact open subgroup  of $G$ contained in $G_w$ for all $w \in \Omega_0$. For every $w\in VT\sqcup ET$, let $K_w\coloneqq K\cap G_w$. We claim every $K_w$ is a finite-index subgroup of $K$. Indeed, since $K$ is compact, $K$ acts elliptically on $T$. Hence, there is a vertex $w_0\in VT$ with $K\leq G_{w_0}$. Since the action of $G$ on $T$ is minimal and cocompact, there exists a $g\in G$ such that $w$ lies on the path $P=[w_0,gw_0]$. Since $K\cap gKg^{-1}$  fixes $P$ pointwise, $K\cap gKg^{-1}\leq K_w\leq K$. As $K$ is a compact open subgroup of $G$, it is commensurated. Since $K\cap gKg^{-1}$  is a finite-index subgroup of $K$, it follows $K_w$ is also a finite-index subgroup of $K$. Let  $K_0\coloneqq \cap_{w\in \Omega_0}K_w$. As $\Omega_0$ is finite, $K_0$  is a compact open subgroup of $G$.
 By replacing $K$ with $K_0$ if necessary, we can thus assume $K\leq G_w$ for every $w\in \Omega_0$. 
 
 We also make the following additional choices on which the construction depends. 
 For each $w\in \Omega_0$, the group $G_w$ is compactly generated, so there exists a finite symmetric set $S_w\subseteq G_w$ such that $S_w\cup K_w$ generates~$G_w$. 
 Let $\Omega_1$ be the closure of $\Omega_0$, i.e.\ $\Omega_1$ is the smallest subgraph of $T$ (not necessarily connected) containing $\Omega_0$.
  Pick a finite symmetric set $F\subseteq G$ such that $\Omega_1\subseteq F\Omega_0$. 
  
  We define a simplicial graph $X$ as follows. The vertex set of $X$ is defined to be 
\[VX\coloneqq G/K \times \Omega_0,\] where $G/K$ is the collection of left cosets of $K$ in $G$. Two vertices $(g_1K,w_1)$ and $(g_2K,w_2)$ are joined by an edge in $X$ if either of the following hold:
\begin{enumerate}
	\item[(I)] $w_1=w_2$ and $g_1^{-1}g_2\in KS_{w_1}K$;
	\item[(II)]   $g_1^{-1}g_2\in KFK$ and $g_1w_1$ is an edge incident to the vertex $g_2w_2$ or vice versa.
\end{enumerate}
We call these \emph{type I} and \emph{type II} edges, respectively.  Note that they are mutually exclusive: an edge  is a type II edge if and only if it is not a type I edge.

We define a map $p:X\rightarrow T$ that  is simplicial with respect to $B(T)$,  the first barycentric subdivision of $T$, as follows:
\begin{itemize}
	\item if $v\in VT\cap \Omega_0$, then $(gK,v)\mapsto gv$;
	\item if $e\in ET\cap \Omega_0$, then $(gK,e)\mapsto m_{ge}$ where $m_{ge}$ is the midpoint of the  edge $ge$.
\end{itemize} 
The map is well-defined because $Kw=w$ for all $w\in \Omega_0$. The map is simplicial since type I edges of $X$ are sent to vertices of $B(T)$, and type II edges of $X$ are sent to edges of $B(T)$.

We claim that $X_y:= p^{-1}(y)$ is a connected subgraph for all vertices $y$ of $B(T)$. Indeed, if $(g_1K,w_1)$ and $(g_2K,w_2)$ are two vertices in $X_y$, then  $w_1=w_2$ and $g_1^{-1}g_2\in G_{w_1}$. 
 Since $K\cup S_{w_1}$ generates $G_{w_1}$,  there exist $s_1,s_2,\dots, s_n\in S_{w_1}$ and $k_1,\dots, k_n\in K$ such that $s_1k_1s_2\dots s_nk_n =g_1^{-1}g_2$, and so 
 \begin{align*} \left((g_1K, w_1),(g_1s_1K, w_1)\right),
 	\left((g_1s_1k_1K, w_1), (g_1s_1k_1s_2K, w_1)\right),
\dots
\end{align*}
is an edge path in $X_y$ from $(g_1K,w_1)$ to $(g_2K,w_2)=(g_2K,w_1)$  consisting only of type I edges. 

To show $X$ is connected, note that since $T$ is connected and each $X_y$ is connected, it is sufficient to show that if $v\in VT$ is incident to $e\in ET$, then there is an edge in $X$ from $X_v$ to $X_e$. Indeed, there exists $g\in G$ such that $g^{-1}e\in \Omega_0$ and so $g^{-1}v\in \Omega_1=\overline{\Omega_0}$. Thus there exists $f\in F$ such that $f^{-1}g^{-1}v\in \Omega_0$. It follows that $(gK,g^{-1}e)\in X_e$ and $(gfK,f^{-1}g^{-1}v)\in X_v$ are joined by a type II edge in $X$. Thus $X$ is connected.  

We now show $X$ is locally finite. Consider the finite set  $J\coloneqq  \bigcup_{w\in \Omega_0}S_w\cup F$. Since $K$ is a commensurated subgroup of $G$, there is a finite subset $J_0\subseteq G$ such that $KJK\subseteq J_0K$.  Therefore, every vertex adjacent to  some $(g_1K,w_1)\in VX$ is of the form $(g_1jK,w_2)$ for some $j\in J_0$ and $w_2\in \Omega_0$. Since there are only finitely many such vertices, $X$ is locally finite.

Define an action of $G$ on $X$ by $g\cdot (g_1K,w_1)=(gg_1K,w_1)$. It is clear that $(g_1K,w_1)$ and $(g_2K,w_2)$ are joined by an edge if and only if $(gg_1K,w_1)$ and $(gg_2K,w_2)$ are, and that $p$ is $G$-equivariant. The graph $X$ is locally finite and $G$ has $\lvert \Omega_0\rvert<\infty$ orbits of vertices in $X$, so $G$ acts cocompactly on $X$. Since every vertex stabilizer is open, the action of $G$ on~$X$ is continuous. Moreover, the action of $G$ on $X$ is proper since $X$ is locally finite and vertex stabilizers are compact. We have thus shown that $p:X\to T$ is a blowup.
\end{proof}

The following lemmas allow us to upgrade the blowup construction to build a model geometry with additional desirable properties needed to apply Theorem \ref{thm:commoncover_BV}.

\begin{lem}\label{lem:fin_pres_loops}
	Let $G$ be a compactly presented group acting geometrically on a locally finite connected graph $X$. For $M\geq 0$, let $X_M$ be the 2-complex obtained by attaching 2-cells to every edge loop of length at most $M$ in $X$. There exists a constant $M_0$ such that $X_M$ is simply-connected for all $M\geq M_0$.
\end{lem}
\begin{proof}
 	Equip the vertex set $X^{(0)}$ with a metric such that $d(x,y)$ is the length of the shortest edge path in $X$ joining $x$ and $y$. The metric space $X^{(0)}$ is quasi-isometric to $G$ so is \emph{coarsely simply connected}; see e.g.\ \cite[\S 6,8]{cornulierdlH2016metric}. Thus, there is a constant $N\geq 1$ such that the 2-skeleton of the Rips complex $Y=P^{(2)}_N(X^{(0)})$ is simply connected \cite[Proposition 6.C.4]{cornulierdlH2016metric}. Since $N\geq 1$, there is a natural injection $\phi:X\to Y$. Setting $M=3N$ yields a continuous map $\psi:Y\to X_M$ sending a 1-simplex $[x,y]\in Y$ to a geodesic path in $X^{(1)}_M$ joining $x$ and $y$, and sending  a 2-simplex $[x,y,z]$ to a 2-cell in $X_M$  attached to edge path $[x,y]\cup[y,z]\cup[z,x]$. This can be done so that $\psi\circ \phi$ is the identity map on $X=X^{(1)}_M$. To show $X_M$ is simply-connected, it is enough to show any loop in $X$  can be filled by a disc in $X_M$. For any loop $f:S^1\to X$, there is a map $F:D^2\to Y$ such that $F|_{S^1}=\phi\circ f$. Thus, $(\psi\circ F)|_{S^1}=f$ as required.
\end{proof}

The following lemma will be essential for obtaining an imprimitivity system on a tree of spaces given imprimitivity systems on vertex spaces. A \emph{continuous $G$-tree} is a tree admitting a continuous $G$-action.
A \emph{finite refinement} of a continuous $G$-tree $T$ is a continuous  $G$-tree $T'$ and an equivariant map $p:T'\to T$ obtained by collapsing finite subtrees of $T'$ to points. 

\begin{lem}[Tree refinement lemma]\label{lem:tree_expansion}
	Let $G$ be a locally compact group acting continuously on a tree $T$ without edge inversions. Assume that for every vertex $v\in VT$, the stabilizer $G_v$ contains a compact open normal subgroup. Then there is a  finite refinement  $p:T'\to T$ such that for each $v\in VT'$,
	\begin{itemize}
		\item $G_v$ contains a finite-index open subgroup that is the stabilizer of a vertex or edge of $T$, and
		\item $G_v$ contains a compact open normal subgroup  fixing $\lk(v)$.
	\end{itemize}
\end{lem}
\begin{proof}
    We first define the tree $T'$. 
    For each $v \in VT$, pick a compact open normal subgroup   $K_v\vartriangleleft G_v$ such that $gK_vg^{-1}=K_{gv}$ for all $g\in G$. 
    Define the vertex set of $T'$ by 
\[ 	VT'\coloneqq VT\sqcup \bigl\{(v,K_ve)\mid v\in VT,e\in \lk(v)\bigr\}.\] 
    We call vertices in $VT$  \emph{type I} vertices and the remaining vertices \emph{type II} vertices. Define the set of edges of $T'$ to be  
	\begin{align*}ET'=&\bigl\{\bigl(v,(v, K_ve)\bigr)\mid v\in VT, e\in \lk(v)\bigr\}\,\, \sqcup\\
		& \bigl\{\bigl((v, K_v\bar e),(v', K_{v'}e)\bigr)\mid e\in ET \text{ and } e=(v, v')\bigr\}.
\end{align*}

Define an action of $G$ on $VT'$ by extending the action on type I vertices and on type~II vertices by $g(v,K_ve)=(gv,K_{gv}ge)=(gv,gK_ve)$.
This action on $VT'$ extends to an action of $G$ on $T'$. Since $G$ acts continuously on $T$ and each $K_v$ is open, the stabilizer of each vertex of $T'$ is open. Hence, the action of $G$ on $T'$ is continuous. For each $v \in VT$, let $Y_v$ be the induced subgraph of $T'$ with vertex set $\{v\}\cup\{(v,K_ve)\mid e\in \lk(v)\}$. Each $Y_v$ is a tree and is equal to $\Star(v)$. It is straightforward to verify that $T'$ is a tree and collapsing each subtree $Y_v$ to the vertex $v$ yields a finite refinement  $p:T'\to T$.

The stabilizer $G_v$ of a type I vertex $v\in VT'$ is precisely the stabilizer of the corresponding  vertex $v\in VT$. Moreover, $G_v$ contains the compact open normal subgroup $K_v$ that fixes $\lk(v)$, since all vertices in $\lk(v)$ are of the form $(v,K_ve)$ for some $e\in \lk(v)$.

Let $u\coloneqq (v,K_ve)$ be a type II vertex, where $e=(w,v)$. We first note that  $G_{u}$ contains $G_e$ as a finite-index open subgroup. It remains to show $G_u$ contains a compact open normal subgroup fixing $\lk(u)$. Since $G_u\leq G_v$ and $K_v$ fixes $u$, we deduce that  $K_v$ is a compact open normal subgroup of $G_u$.
Since the action of $G$ on $T$ is continuous and $K_v$ is compact, there is a finite set $F\subseteq G$ such that $K_ve=Fe$.
  The vertices adjacent to $u$ are precisely $v$ and the $F$ translates of $(w,K_{w}\bar e)$, so  $u$  has finite valence in $T'$. It now follows from continuity of the $G$-action on $T'$ that the kernel $R_u$ of the action of $G_u$ on $\lk(u)$ is a finite-index open normal subgroup of $G_u$. Thus $K_v\cap R_u$ is a compact open normal subgroup of $G_u$ fixing $\lk(u)$.   
\end{proof}

The next theorem is the main result of this subsection. The conditions in the conclusion of Theorem \ref{thm:get_treeofspaces} will be needed to apply Theorem \ref{thm:commoncover_BV} to deduce action rigidity.

 \begin{thm}[Model Geometry Theorem] \label{thm:get_treeofspaces}
 	Let $G$ be a locally compact group acting minimally,  cocompactly and continuously on a tree $T$. Suppose every vertex and edge stabilizer is compactly presented and every vertex stabilizer contains a compact open normal subgroup. Then $T$ has a finite refinement $T'$  such that  $G$ acts geometrically on a tree of spaces $(X,T')$ and for every $w\in VT'\cup ET'$, 
 	\begin{enumerate}
 		\item  $X_w$ is simply-connected;
 		\item  the action of $G_w$ on $X_w$ is discrete;
 		\item all edge and vertex spaces are locally finite cell complexes and the edge maps $\phi_e$ are cellular isomorphisms onto their images;
 		\item for each vertex  $v\in VT'$, the subcomplexes $\{\phi_e(X_e)\}_{e\in \lk(v)}$ are disjoint.
 	\end{enumerate}
 Moreover, if each vertex stabilizer $G_v$ contains a compact open normal subgroup fixing $\lk(v)$, we can assume $T=T'$.
 \end{thm}
\begin{proof}
	First apply Lemma \ref{lem:tree_expansion} to obtain a finite refinement $p:T'\to T$ such that  every vertex or edge stabilizer of $T'$  contains a finite-index open subgroup isomorphic to a vertex or edge stabilizer of $T$. Since all vertex and edge stabilizers of $T$ are compactly presented, so are all  vertex and edge stabilizers of $T'$ \cite[Corollary 8.A.5.]{cornulierdlH2016metric}. Moreover, for every vertex $v\in VT'$, there is  a compact open normal subgroup $K_v\vartriangleleft G_v$ that fixes $\lk(v)$. In the case that every vertex stabilizer of the original tree $T$ contains a compact open normal subgroup fixing $\lk(v)$, we do not need to apply Lemma \ref{lem:tree_expansion} and can take $T=T'$.
	
	Assume without loss of generality that $K_{gv}=gK_{v}g^{-1}$ for all $g\in G$ and $v\in VT'$. For each edge $e=(v,w)\in ET'$, define $K_e=K_vK_w$. Since $K_v$ and $K_w$ are compact open normal subgroups of $G_e$, we deduce $K_e$ is a compact open normal subgroup of $G_e$.  Note also that $K_{ge}=gK_eg^{-1}$ for all $g\in G$. 
	
	Since $G$ contains a compact open subgroup, we can apply Proposition \ref{prop:blowups_exist} to deduce a blowup $p:Y\to T'$ exists.
		For each vertex or edge $w\in VT'\cup ET'$, as $G_w$ is compact-by-discrete and acts geometrically on the graph $Y_w$, there is a  $G_w$-imprimitivity system $\sim_w$ on $Y_w$ whose blocks are $K_w$-orbits of vertices. This can be done because every $K_w$ is a compact open normal subgroup of $G_w$; see the proof of Proposition~\ref{prop:compactbydiscrete}. The union of all these imprimitivity systems gives an imprimitivity system $\sim$ on $Y$ for which each block is finite and contained in a single vertex or edge space.  The choice of the $K_w$ ensure $\sim$ is $G$-invariant, hence $p:Y\to T'$ factors through a blowup $p_Z:Z\to T'$, where $Z=Y/\sim$.
		
	 We now define a graph of spaces $(X,T')$ satisfying the required properties. For each vertex $v\in VT'$, let $W_v$ denote the subgraph of $Z$ induced by $Z_v\cup \bigsqcup_{e\in \lk(v)}Z_e$. Since $K_v\leq K_e$ for all $e\in \lk(v)$, the group $K_v$ acts trivially on $W_v$. Thus, the action of $G_v$ on $W_v$ is discrete. Similarly, the action of $G_e$ on the edge space $W_e:=Z_e$ is discrete.  For each $w\in VT'\cup ET'$, the group $G_w$ is compactly presented and acts geometrically on $W_w$, so by Lemma \ref{lem:fin_pres_loops}, there exists $M \geq 0$ such that the complex $(W_w)_M$ is simply-connected. As there are only  finitely many $G$-orbits of $W_w$ for $w\in VT'\cup ET'$, the number $M$ can be chosen uniformly. Set $X_w=(W_w)_M$, which is a locally finite cell complex. There are natural inclusions $\phi_e:X_e\hookrightarrow X_v$ if $e\in ET'$ is incident to $v\in VT'$. It is clear that for each vertex $v\in VT'$, the subcomplexes $\phi_e(X_e)$ for $e\in\lk(v)$ are disjoint.   The spaces $\{X_x\mid x\in VT'\cup ET'\}$ and maps $\{\phi_e\}_{e\in ET'}$ assemble to form a tree of spaces $(X,T')$. Moreover, the $G$-action on $Z$ induces a $G$-action on $(X,T')$.
\end{proof}

\section{Common Covers}\label{sec:commoncover}

\begin{assumption}\label{ass:ToS}
In this section we work with a topological group $G$ that acts continuously on a tree of spaces $(X,T)$.
We will assume that the edge and vertex spaces are locally finite cell complexes, that each edge map $\phi_e$ is a cellular isomorphism onto its image, and that for each vertex space $X_v$ the subcomplexes $\phi_e(X_e)$ for $e\in\lk(v)$ are disjoint.
\end{assumption}

    \begin{notation}\label{defn:ToS}
        For $v \in VT$, let $G_v \leq G$ be the subgroup stabilizing the vertex space $X_v$. Let $q_v:G_v \rightarrow \Aut(X_v)$ be the induced homomorphism, and let 
        \[ Q_v := q_v(G_v) \leq \Aut(X_v). \]
        Similarly, for $e \in ET$, let $G_e \leq G$ be the subgroup stabilizing the edge space $X_e$. Let $q_e:G_e \rightarrow \Aut(X_e)$ be the induced homomorphism, and let 
        \[ Q_e:= q_e(G_e) \leq \Aut(X_e). \]
        Then, $Q_e = Q_{\overline{e}}$. 
        We will refer to $Q_v$ and $Q_e$ as the \emph{localized vertex and edge groups}.
        Note that any $g\in G_v$ acts on $\lk(v)$ and, since the subcomplexes $\phi_e(X_e)$ are disjoint for $e\in\lk(v)$, we have that $q_v(g)$ determines this action. In other words, the action of $G_v$ on $\lk(v)$ descends to an action of $Q_v$ on $\lk(v)$.
    \end{notation}  
    
    \begin{rem}
     We will consider subgroups $\G,\G'\leq G$ that act faithfully and geometrically on $(X,T)$. The vertex and edge stabilizers of $\G$ and $\G'$ are given by $\G_v := \G \cap G_v$, $\G_v':= \G' \cap G_v$, $\G_e :=\G \cap G_e$, and $\G_e' := \G' \cap G_e$.
     The motivation for considering the localized vertex groups is that typically the intersection of the vertex groups $\G_v \cap \G_v'$ inside $G_v$ will have trivial (or at least a very disappointing) intersection.
     However, under the assumption that $Q_v$ is discrete  as a subgroup of $\Aut(X_v)$, the images $q_v(\G_v)$ and $q_v(\G_v')$ are guaranteed to be commensurable, so their intersection has finite index in the respective images. 
    \end{rem}

    \begin{defn}  \label{defn:ToS_loc_edges} 
    For each $e \in \lk(v)$, let 
        \[Q_v^e := \Stab_{Q_v}(\phi_e(X_e)) \leq Q_v. \] 
     The subcomplex $\phi_e(X_e) \subseteq X_v$ is isomorphic to the edge space $X_e$, so restricting an automorphism of $X_v$  to $\phi_e(X_e)$ yields a surjective homomorphism \[F_v^e:Q_v^e \twoheadrightarrow Q_e.\] 
    \end{defn}
    
    \begin{rem}
     The homomorphism $F_v^e$ need not be injective. For example, let $X_e \cong \R$, let $X_v \cong \H^2$, and let $\phi_e(X_e)$ be a geodesic line in $\H^2$ (or, more precisely, cell complex versions of these spaces). Then, the reflection in $\Hy^2$ about this line is a nontrivial element of $Q_v^e$ that is trivial in $Q_e$. Nonetheless, when $Q_v$ acts discretely on $X_v$, it immediately follows that the kernel of $F_v^e$ is finite. 
     
     If the kernels of all $F_v^e$ were trivial, then any $g\in G$ fixing an edge space would also fix the adjacent vertex spaces, and by fanning outward in $T$ we would see that $g$ fixes the whole tree of spaces. If in addition, the action of each  $Q_v$ on $X_v$ was discrete, then we would conclude that $G$ acts discretely on $(X,T)$, and commensurability of $\G$ and $\G'$ in $\Aut(X,T)$ would be immediate.
    \end{rem}

    \begin{notation}
        Recall from Definition \ref{defn:actonToS} that for each $g \in G$ and $v \in VT$, the restriction of $g$ to the vertex space~$X_v$ yields an isomorphism $g_v: X_v \rightarrow X_{gv}$. If $g \in G_v$, then $q_v(g) = g_v$. Similarly, for $e \in ET$, there is an isomorphism $g_e:X_e \rightarrow X_{ge}$. 
        To each $g \in G$ and $v\in VT$ we obtain an isomorphism $Q_v \rightarrow Q_{gv}$ induced by conjugating by $g_v$.         
    \end{notation}

    \begin{thm}[Common Cover Theorem] \label{thm:commoncover_BV} 
    Let $G$ be a topological group acting continuously on a tree of spaces $(X,T)$ as in Assumption \ref{ass:ToS}.
	Let $\G,\G'\leq G$ be finitely generated subgroups acting geometrically on $(X,T)$ and without edge inversions on $T$. Suppose that the following conditions are satisfied. 
	\begin{enumerate}	
	 \item \label{item:freeAction} The edge and vertex stabilizers of $\G$ and $\G'$ each act freely on the corresponding edge and vertex spaces; 
     \item The edge and vertex spaces are simply connected;
	 \item\label{item:Qvprop} The action of  $Q_v$  on $X_v$ is discrete for all $v \in VT$;
     \item\label{item:Qcommand} For any (finite) set of $G$-orbit representatives $\{e_i\}_{i=1}^k\subseteq\lk(v)$, the group $Q_v$ commands the set of subgroups $\{Q_v^{e_i}\}_{i=1}^k$.
	\end{enumerate}
    Then, the quotient graphs of spaces $X/\G$ and $X/\G'$ have a common finite cover. In particular, $\G$ and $\G'$ are weakly commensurable in $\Aut(X,T)$.
    \end{thm}

    \begin{rem}
    The conclusion that $\G$ and $\G'$ are weakly commensurable in $\Aut(X,T)$ still follows even if $\G,\G'$ act with edge inversions on $T$. Indeed, if we consider $T$ as a bipartite graph with a partition of its vertex set, then $\G,\G'$ have subgroups of index at most 2 that stabilize both halves of the partition, and hence do not contain edge inversions.
    \end{rem}
    
    To prove Theorem~\ref{thm:commoncover_BV}, we construct $(\hat{X}, S)$, a locally-finite tree of compact spaces so that $\hat{X}$ is a regular cover of both $X/\G$ and $X/\G'$. Then, we apply the following theorem to obtain a common finite-sheeted cover of $X/\G$ and $X/\G'$. This theorem is obtained from Leighton's Theorem for Graphs of Objects \cite[Theorem 4.7]{Shepherd22} by taking $\cC$ to be the category of finite cell complexes.
    
    \begin{thm} \cite[Theorem 4.7]{Shepherd22} \label{thm:GoO} 
    Let $f_i:(Z,T')\to (Y_i,\gL_i)$ be coverings of graphs of spaces for $i=1,2$ such that:
    \begin{enumerate}
        \item $T'$ is a tree,
        \item all edge and vertex spaces are finite cell complexes,
        \item $f_i$ restricts to isomorphisms between edge and vertex spaces.
    \end{enumerate}
    Then $(Y_1,\gL_1)$ and $(Y_2,\gL_2)$ have a common finite-sheeted cover.
    \end{thm}

    \begin{rem} \label{rem:OnConditions} (\emph{On Conditions (\ref{item:freeAction})--(\ref{item:Qcommand}) in Theorem~\ref{thm:commoncover_BV}.})            
        If the edge and vertex groups act freely on the corresponding edge and vertex spaces of $(X,T)$ (condition (\ref{item:freeAction})), then $q_v$ and $q_e$ restrict to injective homomorphisms $\G_v, \G_v' \rightarrow Q_v$ and $\G_e, \G_e' \rightarrow Q_e$. In particular, $q_v(\G_v) \cong \G_v$ and so on. 
        We also note that condition (\ref{item:freeAction}) implies that $F_v^e$ is injective on $q_v(\G_v) \cap Q_v^e$ and $q_v(\G_v') \cap Q_v^e$.
        
        We assume each vertex space $X_v$ is simply connected, which together with condition (\ref{item:freeAction}), implies $\pi_1(X_v/\G_v) \cong \G_v$, and likewise for the edge spaces.      
        We assume the group $Q_v$ acts discretely on the vertex space $X_v$, so the vertex groups $q_v(\G_v)$ and  $q_v(\G_v')$ are commensurable inside $Q_v$. Condition~(\ref{item:Qcommand}) is used to ensure that the common finite-index subgroups of the vertex groups $q_v(\G_v)$ and $q_v(\G_v')$ guaranteed by condition (\ref{item:Qvprop}) can be chosen to agree along incident edge groups. 
    \end{rem}
    
    The common ``locally finite'' regular cover $\hat{X}$ of $X/\G$ and $X/\G'$ can be constructed by specifying common finite-index subgroups of the vertex groups
    $\hat{Q}_v \triangleleft q_v(\G_v) \cap q_v(\G_v') \leq Q_v$ for $v \in VT$ whose induced actions on the edge spaces match up in the sense of the next proposition. %

    \begin{prop} \label{prop:lf_tree_assump}
    	With the setup of Theorem \ref{thm:commoncover_BV}, suppose that for each $v \in VT$ there exists a finite-index subgroup $\hat{Q}_v \triangleleft Q_v$ so that the following hold:
        \begin{enumerate}
         \item\label{item:vscc} (Vertex space common covers) $\hat{Q}_v \leq q_v(\G_v) \cap q_v(\G_v')$. 
         \item\label{item:eq} (Equivariance) $g_v \hat Q_v g_v^{-1} = \hat{Q}_{gv}$ for all $g \in G$.
         \item\label{item:glueQh} (Gluing condition) For each edge $e = (u,v) \in ET$ with $\hat{Q}_v^e := Q_v^e \cap \hat{Q}_v $ and $\hat{Q}_u^{\bar{e}} := Q_u^{\bar{e}} \cap \hat{Q}_u$, we have
         \[ \hat{Q}_e :=  F_v^e(\hat{Q}_v^e) = F_u^{\bar{e}}(\hat{Q}_u^{\bar{e}}) \leq Q_e,\]
         and the restrictions $F_v^e|_{\hat{Q}_v^e}:\hat{Q}_v^e \rightarrow \hat{Q}_e$ and $F_u^{\bar{e}}|_{\hat{Q}_u^{\bar{e}}}:\hat{Q}_u^{\bar{e}} \rightarrow \hat{Q}_e $ are isomorphisms.
         \end{enumerate}
       Then, there exists a locally-finite tree of compact spaces $(\hat{X}, S)$ so that $\hat{X}$ is a regular cover of both $X/\G$ and $X/\G'$. 
    \end{prop}

\begin{defn}
    If $f_1:A_1\to B$ and $f_2:A_2\to B$ are continuous maps of topological spaces, then the associated \emph{fiber product} is the topological space $C:=\{(a_1,a_2)\in A_1\times A_2\mid f_1(a_1)=f_2(a_2)\}$, considered as a subspace of $A_1\times A_2$. Note that the projection maps $\pi_1:C\to A_1$ and $\pi_2:C\to A_2$ satisfy $f_1\pi_1=f_2\pi_2$. 

    If $C'$ is another topological space, then a commutative diagram of continuous maps
    \begin{equation}\label{fibreproduct}
		\begin{tikzcd}[
			ar symbol/.style = {draw=none,"#1" description,sloped},
			isomorphic/.style = {ar symbol={\cong}},
			equals/.style = {ar symbol={=}},
			subset/.style = {ar symbol={\subset}}
			]
			C'\ar{r}{\pi'_2}\ar{d}{\pi'_1}&A_2\ar{d}{f_2}\\
			A_1\ar{r}{f_1}&B,
		\end{tikzcd}
	\end{equation}
 is called a \emph{fiber product diagram} if there is a homeomorphism $h:C\to C'$ such that $\pi_1=\pi'_1 h$ and $\pi_2=\pi'_2 h$. We note that if $f_2:A_2\to B$ is a covering map of cell complexes, then $\pi'_1:C'\to A_1$ restricts to a covering map on each component of $C'$.
\end{defn}	

\begin{rem}\label{rem:fibreproduct}
If the map $\pi'_1$ in diagram (\ref{fibreproduct}) restricts to a covering on each component of $C'$, then any bijection $h:C\to C'$ with $\pi_1=\pi'_1 h$ and $\pi_2=\pi'_2 h$ will automatically be a homeomorphism. So in this case, to check that (\ref{fibreproduct}) is a fiber product diagram, it suffices to check that for each $(a_1,a_2)\in A_1\times A_2$ with $f_1(a_1)=f_2(a_2)$ there is a unique $c\in C'$ with $\pi'_1(c)=a_1$ and $\pi'_2(c)=a_2$.
\end{rem}

Fiber product diagrams occur naturally in the context of coverings of graphs of spaces. Specifically, suppose $f:(\hat{X},\hat{\Lambda})\to (X,\Lambda)$ is a morphism of graphs of spaces in which the maps between vertex spaces and the maps between edge spaces are covering maps. Then $f$ is a covering if and only if, for each vertex $v\in\Lambda$ and lift $\hat{v}\in\hat{\Lambda}$, the following is a fiber product diagram:
    \begin{equation}\label{fibreproductGoS}
		\begin{tikzcd}[
			ar symbol/.style = {draw=none,"#1" description,sloped},
			isomorphic/.style = {ar symbol={\cong}},
			equals/.style = {ar symbol={=}},
			subset/.style = {ar symbol={\subset}}
			]
			\bigsqcup_{\hat{e}\in\lk(\hat{v})}\hat{X}_{\hat{e}}\ar{r}\ar{d}&\hat{X}_{\hat{v}}\ar{d}\\
			\bigsqcup_{e\in\lk(v)}X_e\ar{r}&X_v
		\end{tikzcd}
	\end{equation}

    \begin{proof}[Proof of Proposition \ref{prop:lf_tree_assump}]
        We first define the locally-finite tree of compact spaces $(\hat{X}, S)$. 
        Let $S \subseteq T$ be a subtree so that for each $v \in VS$, the subtree $S$ contains exactly one edge from each $\hat{Q}_v$-orbit in $\lk(v)$. Since the group $\G$ acts cocompactly on $X$, the tree $S$ is locally finite. 
        For $v \in VS$ and $e \in ES$, define vertex and edge spaces by
        \[\hat{X}_v:= X_v / \hQ_v \,\, \textrm{ and } \,\,
          \hat{X}_e:= X_e /\hQ_e,  \]
          where $\hQ_e$ is given in condition (\ref{item:glueQh}).  The space $\hX_v$ is compact because $\hQ_v \leq q_v(\G_v) \cap q_v(\G_v')$ is a finite-index subgroup of $Q_v$, and $q_v(\G_v)\leq Q_v$ acts cocompactly on $X_v$. Similarly, $\hX_e$ is compact. If $e \in \lk(v)$, then as the subspaces $\{\phi_{e'}(X_{e'})\}_{e'\in \lk(v)}$ are disjoint,  %
           condition (\ref{item:glueQh}) implies
        \[\phi_e(X_e)/\hat{Q}_v^e \cong X_e /\hat{Q}_e = \hat{X}_e.\]  
        Thus, there is an embedding $\hat{X}_e \hookrightarrow \hat{X}_v$. The above data defines a tree of spaces $(\hX, S)$ as desired. 
        
        We now prove that $(\hat{X},S)$ is a regular cover of $(X,T)/\G$; the proof for $(X,T)/\G'$ is identical. There are finite-sheeted regular covers of the vertex and edge spaces 
         \[\hX_v = X_v/\hQ_v \rightarrow X_v/q_v(\G_v) \,\, \textrm{ and } \,\, \hX_e = X_e/\hQ_e \rightarrow X_e/q_e(\G_e) \]
        since $\hQ_v$ is a finite-index normal subgroup of $q_v(\G_v)$ by assumption, and likewise for $\hQ_e \triangleleft q_e(\G_e)$.        
        The covering maps for vertex and edge spaces define a morphism of graphs of spaces $\theta: (\hat{X},S)\to(X,T)/\G$, as shown in the diagram below on the right. Indeed,  the commutativity of the diagram on the right results from the commutativity of the diagram on the left, which follows from definitions. 
	\begin{equation*}
	\begin{tikzcd}[
	ar symbol/.style = {draw=none,"#1" description,sloped},
	isomorphic/.style = {ar symbol={\cong}},
	equals/.style = {ar symbol={=}},
	subset/.style = {ar symbol={\subset}}
	]
	\hQ_e \ar{d}{(F_v^e)^{-1}}\ar{r}& q_e(\G_e)\ar{d}\\
	\hQ_v \ar{r}& q_v(\G_v)
	\end{tikzcd}
	\quad \quad \quad \quad 	
	\begin{tikzcd}[
	ar symbol/.style = {draw=none,"#1" description,sloped},
	isomorphic/.style = {ar symbol={\cong}},
	equals/.style = {ar symbol={=}},
	subset/.style = {ar symbol={\subset}}
	]
	\hat{X}_e = X_e/\hQ_e \ar{d}\ar{r}&X_e/q_e(\G_e)\ar{d}\\
	\hat{X}_v = X_v/\hQ_v \ar{r}&X_v/q_v(\G_v)
	\end{tikzcd}
	\end{equation*}

We now show that $\theta$ is a covering of graphs of spaces by verifying that the following is a fiber product diagram for any $v\in VS$ and $e\in\lk(v)$:
\begin{equation}\label{fibreproductphi}
		\begin{tikzcd}[
			ar symbol/.style = {draw=none,"#1" description,sloped},
			isomorphic/.style = {ar symbol={\cong}},
			equals/.style = {ar symbol={=}},
			subset/.style = {ar symbol={\subset}}
			]
			\bigsqcup_{\hat{e}\in S\cap \G_v\cdot e}\hat{X}_{\hat{e}}\ar{r}\ar{d}&\hat{X}_v\ar{d}\\
			X_e/q_e(\G_e)\ar{r}&X_v/q_v(\G_v)
		\end{tikzcd}
	\end{equation}
 Note that, for $\hat{e}\in S\cap \G_v\cdot e$, if $g\in \G_v$ with $g\hat{e}=e$, then the map $\hat{X}_{\hat{e}}\to X_e/q_e(\G_e)$ is given by $\hat{Q}_{\hat{e}}\cdot x\mapsto q_e(\G_e)\cdot g_{\hat{e}}(x)$.

 By Remark \ref{rem:fibreproduct}, it suffices to consider points $q_e(\G_e)\cdot y\in X_e/q_e(\G_e)$ and $\hat{Q}_v\cdot x\in \hat{X}_v$ that map to the same point in $X_v/q_v(\G_v)$, and prove that they have a unique common lift in $\sqcup_{\hat{e}\in S\cap \G_v\cdot e}\hat{X}_{\hat{e}}$.
 Firstly, since they map to the same point in $X_v/q_v(\G_v)$, we know that there is $g\in\G_v$ with $g_v\phi_e(y)=x$.
 Now by construction of $S$ and condition (\ref{item:vscc}) of Proposition \ref{prop:lf_tree_assump}, there exists $h\in\G_v$ with $q_v(h)\in\hat{Q}_v$ and $\hat{e}:=hge\in S$. Putting $\hat{y}:=(hg)_e(y)$, we have that $\hat{Q}_{\hat{e}}\cdot\hat{y}\in\hat{X}_{\hat{e}}$ maps to $q_e(\G_e)\cdot y\in X_e/q_e(\G_e)$. 
 Now observe
 $$\phi_{\hat{e}}(\hat{y})=(hg)_v\phi_e(y)=h_v(x),$$
 and $\hat{Q}_v\cdot h_v(x)=\hat{Q}_v\cdot x$ since $q_v(h)\in\hat{Q}_v$, hence $\hat{Q}_{\hat{e}}\cdot\hat{y}$ also maps to $\hat{Q}_v\cdot x\in \hat{X}_v$.
 Finally, we must argue that such a point $\hat{Q}_{\hat{e}}\cdot\hat{y}$ is unique.
 Indeed, the subcomplexes $\phi_{e'}(X_{e'})\subset X_v$ for $e'\in\lk(v)$ are disjoint, so $\hat{Q}_v\cdot x$ intersects only one $\hat{Q}_v$-orbit of these subcomplexes, hence the edge space $X_{\hat{e}}$ containing $\hat{y}$ is uniquely determined.
 Moreover, the map $\hat{X}_{\hat{e}}\to\hat{X}_v$ is an embedding, so the point $\hat{Q}_{\hat{e}}\cdot\hat{y}\in\hat{X}_{\hat{e}}$ is also uniquely determined.
 Thus $\theta$ is a covering as claimed.

	Having constructed a covering map $\theta : (\hat{X},S)\to(X,T)/\G$, it remains to show the cover is regular. 
	This requires showing that the deck transformations are transitive on the fibers of $\theta$. 
	First note that for $g\in \G$ and $v\in VT$, condition (\ref{item:eq}) implies that $g_v\hat{Q}_vg_v^{-1}=\hat{Q}_{gv}$, and so if $v, gv\in VS$ then the map $g_v:X_v\to X_{gv}$ descends to an isomorphism $\hat{g}_v:\hat{X}_v\to\hat{X}_{gv}$. Similarly, if $e,ge\in ES$ then we get an isomorphism $\hat{g}_e:\hat{X}_e\to\hat{X}_{ge}$.

	Now returning to the regularity of $\theta$, suppose that $\hat x_1, \hat x_2 \in \theta^{-1}(x)$ for some vertex $x$ in a vertex space of $(X,T)/\G$.
	Then $\hat x_i \in \hat X_{v_i}$, where $v_i \in VS$. The point $\hat x_i$ corresponds to an orbit $\hat Q_{v_i} \cdot x_i$ for some $x_i \in X_{v_i}$, and $\theta(\hat x_i)$ corresponds to the orbit $\G \cdot x_i$.
	So $\theta(\hat x_1) = \theta(\hat x_2)$ implies that $\G\cdot{x_1} = \G\cdot{x_2}$. Hence, there exists $g \in \G$ such that $gx_1 = x_2$.
	Let $v := v_1$.
	We construct a deck transformation $\hat{f}:(\hat{X},S)\to(\hat{X},S)$ that extends the map $\hat{g}_v:\hat X_v=\hat X_{v_1}\to \hat X_{v_2}$, such that the restriction of $\hat{f}$ to each vertex space $\hat{X}_u$ is a map $\hat{f}_u=\hat{h}_u:\hat{X}_u\to\hat{X}_{\hat{f}u}$ for some $h\in \G$, and similarly for the restrictions to edge spaces. 
	It is clear from these local maps that the covering $(\hat{X},S)\to(X,T)/\G$ will be invariant under $\hat{f}$. 
	We will define $\hat{f}$ inductively by fanning out from $v$, so we first set $\hat{f}v=gv=v_2$ and $\hat{f}_v=\hat{g}_v$. 
	
	Given $e\in\lk(v)\cap ES$, it might be that $ge\notin ES$. However, by the construction of $S$, there exists $h\in \G_{gv}$ such that $q_{gv}(h)\in\hat{Q}_{gv}$ and $hge\in ES$. We then define $\hat{f} e=hge$ and $\hat{f}_e=\widehat{(hg)}_e$. We now consider the following commutative diagram. 
	\begin{equation*}
		\begin{tikzcd}[
			ar symbol/.style = {draw=none,"#1" description,sloped},
			isomorphic/.style = {ar symbol={\cong}},
			equals/.style = {ar symbol={=}},
			subset/.style = {ar symbol={\subset}}
			]
			X_e\ar{d}\ar{r}{g_e}& X_{ge}\ar{d}\ar{r}{h_{ge}} &X_{hge}\ar{d}\\
			X_v\ar{r}{g_v}& X_{gv} \ar{r}{h_{gv}} &X_{gv}
		\end{tikzcd}
	\end{equation*}
	Since $q_{gv}(h)\in\hat{Q}_{gv}$, we see $\hat h_{gv}:\hat X_{gv}\to \hat X_{gv}$ is the identity, so the above commutative diagram  descends to the following commutative diagram.

		\begin{equation*}
		\begin{tikzcd}[
			ar symbol/.style = {draw=none,"#1" description,sloped},
			isomorphic/.style = {ar symbol={\cong}},
			equals/.style = {ar symbol={=}},
			subset/.style = {ar symbol={\subset}}
			]
			\hat{X}_e\ar{d}\ar{r}{\hat{f}_e=\widehat{(hg)}_e}&\hat{X}_{\hat{f}e}\ar{d}\\
			\hat{X}_v\ar{r}{\hat{f}_v=\hat{g}_v}&\hat{X}_{gv}.
		\end{tikzcd}
	\end{equation*}
	
	If $u=\iota(e)\in VS$ is the other end of $e$, then we continue the definition of $\hat{f}$ by setting $\hat{f}u=hgu$ and $\hat{f}_u=\widehat{(hg)}_u$. 
	 We then repeat the procedure for other edges in $\lk(v)\cap ES$. Note that the map $g:\lk(v)\to\lk(gv)$ will send edges in different $\hat{Q}_v$-orbits to edges in different $\hat{Q}_{gv}$-orbits (again using condition (\ref{item:eq})), while composing with the various $h$ does not change the orbits; this means that $\hat{f}:\lk(v)\cap ES\to\lk(gv)\cap ES$ is an injection, and hence a bijection.
	
	 At the next stage we fan out from $u$ by doing a similar thing for edges in $\lk(u)\cap ES$ (other than $\bar{e}$), and so on. 
	 We have already noted that $\hat{f}$ induces bijections on links in $S$, so it must define an automorphism of $S$, and hence an automorphism of $(\hat{X},S)$. Finally, because we were free to choose the initial points $\hat x_1, \hat x_2 \in \theta^{-1}(x)$, the collection of all such $\hat{f}$ must form the complete group of deck transformations of the covering $(\hat{X},S)\to(X,T)/\G$, which shows that this covering is regular.         
    \end{proof}

    To prove Theorem~\ref{thm:commoncover_BV}, we will first find subgroups $\hat{Q}_v \trianglelefteq Q_v$ that satisfy the common cover, equivariance, and gluing condition hypotheses of Proposition~\ref{prop:lf_tree_assump}, then we will apply Leighton's Theorem for graphs of objects to the cover $(\hat{X}, S)$ of $(X,T)/\G$ and $(X,T)/\G'$ provided by Proposition~\ref{prop:lf_tree_assump}. 
    The first step in finding the subgroups $\hat{Q}_v$ will be taking finite-index subgroups $\check Q_v \leqslant Q_v$ that will satisfy the common cover and equivariance conditions.

    \begin{defn}
      The \emph{localized cores} of the localized vertex groups are the following intersections:
    \begin{equation*}
\check{Q}_v:=\bigcap_{g\in G, \; u = gv}  g_v^{-1}  q_u(\G_{u}) g_v \cap  g_v^{-1} q_u(\G_{u}') g_v \leq  Q_v,
\end{equation*}
        where $g_v^{-1}q_u(\Gamma_u)g_v := \{g_v^{-1} \circ q_u(\gamma) \circ g_v \in \Isom(X_v) \,|\, \gamma \in \Gamma_u\}$.
    \end{defn}

\begin{lem}
    If the localized vertex groups  act discretely on their associated vertex spaces, then the localized core
	$\check{Q}_v$ is a finite-index normal subgroup of $Q_v$, and $g_v\check{Q}_v g_v^{-1}=\check{Q}_{gv}$ for all $g\in G$.
\end{lem}
\begin{proof}
	It is immediate from the definition that $g_v\check{Q}_v g_v^{-1}=\check{Q}_{gv}$, and this implies the normality of $\check{Q}_v$ in $Q_v$. 
	
	The groups $\G$ and $\G'$ act geometrically on $(X,T)$, so there is a constant $N$ such that for all $v\in VT$, $\G_v$ and $\G'_v$ each act on $X_v$ with fundamental domain a subcomplex consisting of at most $N$ simplices. The same is then true for $q_v(g^{-1}\G_{gv}g)$ and $q_v(g^{-1}\G'_{gv}g)$ acting on $X_v$, and by discreteness of the $Q_v$-action we see that these are subgroups of $Q_v$ of uniformly bounded index. As $\check{Q}_v$ is an intersection of such subgroups, and $Q_v$ is finitely generated (as it acts geometrically on $X_v$), it follows that $\check{Q}_v$ has finite index in $Q_v$.
\end{proof}
    
    In the proof of Theorem~\ref{thm:commoncover_BV} below, we define $\hat Q_v$ to be a $G$-equivariant choice of finite-index subgroup of $\check Q_v$, which will guarantee that the equivariance and common cover conditions of Proposition~\ref{prop:lf_tree_assump} are still satisfied.
    We ensure that this choice of finite-index subgroup satisfies the gluing condition by applying the commanding hypothesis.

\begin{proof}[Proof of Theorem~\ref{thm:commoncover_BV}]
    We will first construct groups $\{\hat{Q}_v \leq Q_v \, | \, v \in VT\}$ satisfying the conditions of Proposition~\ref{prop:lf_tree_assump} by using the commanding assumption and the localized cores. Then, we apply Leighton's Theorem for graphs of objects to the cover $(\hat{X}, S)$ of $X/\G$ and $X/\G'$ guaranteed by Proposition~\ref{prop:lf_tree_assump}. 
    
    The outline of the construction of the desired groups $\{\hat{Q}_v \leq Q_v \, | \, v \in VT\}$ is as follows. We begin by using the commanding hypothesis on the groups $\{Q_v \, | \, v \in VT\}$ to obtain finite-index subgroups $\dot{Q}_v^e \leq Q_v^e$ for all $v \in VT$ and $e \in \lk(v)$. Then, we specify finite-index subgroups $\hat{Q}_v^e \leq \dot{Q_v^e}$ that are normal in $Q_v^e$, satisfy $F_v^e(\hat{Q}_v^e) = F_u^{\bar{e}}(\hat{Q}_u^{\bar{e}})$ for an edge $e = (u,v)$, and so that $F_v^e|_{\hat{Q}_v^e}$ is an isomorphism onto its image. Finally, we apply the commanding hypothesis to obtain $\hat{Q}_v$ so that $\hat{Q}_v \cap Q_v^e = \hat{Q}_v^e$. We ensure that our constructions are $G$-equivariant throughout. 
    
    We first specify, for each $v \in VT$, a choice of finite-index subgroups $\dot{Q}_v^e \leq Q_v^e$ for all $e \in \lk(v)$.  
    Let $v \in VT$. Let $\{e_i\}_{i=1}^k$ be a finite set of $G_v$-orbit representatives of edges in $\lk(v)$ (or equivalently $Q_v$-orbit representatives). Let 
    \[\check{Q}_v^e:=\check{Q}_v\cap Q_v^e\triangleleft Q_v^e,\] 
    where $\check{Q}_v$ is the localized core. 
    The restriction of $F_v^e$ to $\check Q_v^e$ is injective since $\check Q_v^e \leq q_v(\G_v^e)$, and $F_v^e$ is injective on $q_v(\G_e)$ by Remark~\ref{rem:OnConditions}.
    By our hypotheses, $Q_v$ commands $\{ Q_v^{e_i}\}_{i=1}^k$ so there exist finite-index subgroups $\dot{Q}_v^{e_i} \trianglelefteq Q_v^{e_i}$ as in Definition~\ref{defn:commands}.
    We may assume that $\dot{Q}_v^{e_i}\triangleleft \check Q_v^{e_i}$. Define $\dot{Q}_v^e\triangleleft \check Q_v^e$ for all $e\in\lk(v)$ by conjugating. 
    Furthermore, define families $\{\dot{Q}_u^e\}_{e \in \lk(u)}$ for all $u\in G\cdot v$ by conjugating, so  $g_u\dot{Q}_{u}^{e}g_u^{-1}=\dot{Q}_{gu}^{ge}$ for all $g\in G$, $u\in G\cdot v$ and $e\in\lk(u)$. 
    We define $\{\dot{Q}_u^e\}_{e \in ET}$ by repeating the above procedure for each $G$-orbit of vertices. 
    One upshot is that we can bound the indices $[Q_e:F_v^e(\dot{Q}_v^e)]$ by some constant $M$ independent of $e$ and $v$ since there are finitely many vertex and edge orbits.
   
    To now define the subgroups $\hat{Q}_v^e \leq \dot{Q}_v^e \leq Q_v^e$, let $\hat{Q}_e\trianglelefteq Q_e$ be the intersection of all subgroups of $Q_e$ of index at most $M$. 
    Then, $\hat{Q}_e$ has finite index in $Q_e$, and $\hat{Q}_e=\hat{Q}_{\bar{e}}$. 
    Let $\hat Q_v^e = (F_v^e)^{-1}(\hat Q_e) \cap \dot{Q}_v^e$.
    The map $F_v^e$ is injective on $\dot{Q}_v^e$ since $\dot{Q}_v^e\leq \check{Q}_v^e$, and observe that by construction $\hat Q_e \leq F_v^e(\dot Q_v^e)$, so $F_v^e$ will isomorphically identify $\hat Q_v^e$ with $\hat Q_e$.
    Since $g_e Q_e g_e^{-1}=Q_{ge}$ for $g\in G$, the characteristic definition of $\hat{Q}_e$ and the $G$-equivariance of $\check Q_v$ and $\check Q_v^e$ implies that $g_e\hat{Q}_e g_e^{-1}=\hat{Q}_{ge}$ and that $g_v\hat{Q}_{v}^{e}g_v^{-1}=\hat{Q}_{gv}^{ge}$.
   
    Finally, apply the fact that $Q_v$ commands the subgroups $\{Q_v^{e_i}\}_{i=1}^k$ and the fact that $\hat{Q}_v^{e_i} \leq \dot{Q}_v^{e_i}$ to obtain a finite-index normal subgroup $\hat{Q}_v\triangleleft Q_v$ such that $\hat{Q}_v\cap Q_v^{e_i}=\hat{Q}_v^{e_i}$. By normality of $\hat{Q}_v$ and $G$-equivariance of the $\hat{Q}_v^e$, we in fact have that $\hat{Q}_v\cap Q_v^e=\hat{Q}_v^e$ for all $e\in\lk(v)$. By intersecting with $\check{Q}_v$ we can assume that $\hat{Q}_v\leq \check{Q}_v$.  By the $G$-equivariance of the subgroups $\check{Q}_v$, $\dot{Q}_v^e$ and $\hat{Q}_v^e$, we can choose  the subgroups $\hat{Q}_v$ as above such that  $g_v\hat{Q}_v g_v^{-1}=\hat{Q}_{gv}$ for all  $v\in VT$ and $g\in G$. Hence, condition~(2) is satisfied.  By construction, condition~(1) of Proposition~\ref{prop:lf_tree_assump} is satisfied.  
    Condition~(3) holds because $F_v^e$ restricts to an isomorphism from $\hat Q_v^e$ to $\hat Q_e$.

    There exists a pair of regular coverings $(\hat{X},S)\to(X,T)/\G$ and $(\hat{X},S)\to(X,T)/\G'$ by Proposition~\ref{prop:lf_tree_assump}. 
    The covering $(\hat{X},S)\to(X,T)/\G$ corresponds to a subgroup $\hat{\G}\triangleleft \G$. 
    Since the deck transformation group $\G/\hat{\G}$ acts properly, cocompactly and discretely on the tree $S$ (note this is the deck transformation action constructed in Proposition~\ref{prop:lf_tree_assump}, not the restriction of the original action of $\G$ on $T$), there is a finite-index subgroup $\bar \G\leq \G$ containing $\hat \G$ such that $\bar \G / \hat \G$ acts freely on $S$.
    The covering then factors as $(\hat{X},S)\to(\hat{X},S)/\bar{\G}\to (X,T)/\G$, where the first covering restricts to isomorphisms between vertex and edge spaces.  Similarly, we have a composition of coverings $(\hat{X},S)\to(\hat{X},S)/\bar{\G}'\to (X,T)/\G'$. Finally, we obtain a common finite cover of $(\hat{X},S)/\bar{\G}$ and $(\hat{X},S)/\bar{\G}'$ by applying Theorem \ref{thm:GoO}. This completes the proof of Theorem~\ref{thm:commoncover_BV}, so $\G$ and $\G'$ are abstractly commensurable.
    \end{proof}

        \section{Action rigidity}

\subsection{Ends and accessibility for locally compact groups}

	We recall the definition of ends of a locally compact compactly generated group \cite[\S 19.4.2]{cornulier2018quasiisometric}.
	The \emph{space of ends} $E(X)$ of a geodesic metric space $X$ is the inverse limit of $\{\pi_0(X\setminus B)\}_{B}$, where $B$ ranges over bounded subsets of $X$. The \emph{number of ends} of $e(X)$ of $X$ is $e(X)\coloneqq |E(X)|\in \bN\cup\{\infty\}$. A quasi-isometry $f:X\to Y$ induces homeomorphism $E(X)\cong E(Y)$, so that $e(X)=e(Y)$. If $G$ is a compactly generated locally compact group, then $e(G)$ is defined to be the number of ends of its Cayley graph with respect to a compact generating set, which is well-defined up to quasi-isometry. Although this graph is not proper if $G$ is not discrete, this does not matter as we defined ends using bounded rather than compact subsets.
	
	Suppose $\Gamma$ is a finitely generated group and that $\rho:\Gamma\to G$ is a uniform lattice embedding into a locally compact group $G$. Then Lemma~\ref{lem:finitevscompact} implies $G$ is compactly generated. Moreover, \cite[Propositions 4.B.4. and 5.C.3]{cornulierdlH2016metric} imply that $\rho$ is a quasi-isometry, where $\Gamma$ and $G$ are equipped with the word metric with respect to finite and compact generating sets respectively. Consequently, we have the following:
\begin{lem}\label{lem:ends_lattice}
	If a finitely generated group $\Gamma$ is  a uniform lattice in a locally compact group $G$, then $e(\Gamma)=e(G)$.
\end{lem}

We make use of the following result of Houghton:

\begin{thm}[{\cite{houghton1974Ends},\cite[\S 19.4.2]{cornulier2018quasiisometric}}]\label{thm:houghton_ends}
	Let $G$ be a compactly generated locally compact group. If $e(G)=\infty$, then $G$ is compact-by-(totally disconnected).
\end{thm}

We will use the following reformulation of an accessibility result due to Cornulier, building on work of Stallings, Abels, Dunwooody and Kr\"on--M\"oller:

\begin{thm}[\cite{dunwoody1985accessibility,kronmoller08roughcayley,cornulier2018quasiisometric}]\label{thm:dunwoody_accessibility}
	Let $G$ be an infinite-ended  compactly presented locally compact  group. Then $G$ has a continuous cocompact action on a tree with compact edge stabilizers and vertex stabilizers compactly generated and having at most one end.
\end{thm}
\begin{proof}[Remark on the Proof]
		It follows from \cite[Corollary 19.39]{cornulier2018quasiisometric} that $G$ does not admit a continuous proper transitive isometric action on the real line.
	It now follows from Definition 19.40 and  Theorem 19.45 of \cite{cornulier2018quasiisometric} that $G$ admits a continuous cocompact action on a tree $T$ with compact edge stabilizers and
	vertex stabilizers having at most one end. Although not explicitly stated in \cite{cornulier2018quasiisometric}, it follows from \cite[Theorem 3.27]{kronmoller08roughcayley} and \cite[Theorem 19.41]{cornulier2018quasiisometric}  that vertex stabilizers are compactly generated.
\end{proof}

\subsection{Action rigidity for infinite-ended groups}

Recall from Definition \ref{defn:commandsfinite} that a group $\Gamma$ \emph{commands its finite subgroups} if for every finite subgroup $\gL \leq \Gamma$, there exists a finite-index subgroup $\Gamma' \leq \Gamma$ so that $\Gamma' \cap \gL = \{1\}$. For example, residually finite groups and virtually torsion-free groups command their finite subgroups.

\begin{thm} \label{thm:actionrigidA}
Suppose $\Gamma$ and $\Gamma'$ are finitely presented infinite-ended groups that command their finite  subgroups. Suppose every one-ended vertex group in a Stallings--Dunwoody decomposition of $\Gamma$ is graphically discrete. If $\Gamma$ and $\Gamma'$ act geometrically on the same proper  metric space, then $\Gamma$ and $\Gamma'$ are abstractly commensurable. In particular, $\Gamma$ is action rigid within the family of groups that command their finite  subgroups.
\end{thm}
\begin{proof}
    Suppose that $\Gamma$ and $\Gamma'$ act geometrically on a proper  metric space~$Y$. 
    Then $G:=\Isom(Y)$ is an infinite-ended compactly presented locally compact  group that contains $\G$ and $\G'$ as uniform lattices modulo finite kernels. 
    By Theorem~\ref{thm:dunwoody_accessibility}, $G$ admits a continuous cocompact action on a tree $T$ with compact edge stabilizers and vertex stabilizers compactly generated and with at most one end. After passing to a subtree if necessary, we may assume the action is minimal. 
    
    Since $G_v$ is open for all $v \in VT$, it contains $\Gamma_v$ as a uniform lattice modulo finite kernel by Lemma~\ref{lem:open_lattice}. By Theorem~\ref{thm:houghton_ends}, $G$ is compact-by-(totally disconnected). Therefore, as  $\Gamma_v$ is graphically discrete, the group $G_v$ contains a compact open normal subgroup. Thus, by Theorem~\ref{thm:get_treeofspaces}, the tree $T$ has a finite refinement $T'$ such that $G$ acts geometrically on a tree of spaces $(X,T')$ that satisfies Assumption \ref{ass:ToS}. Moreover, for all $a \in VT \sqcup ET$, the space $X_a$ is simply connected and the action of $G_a$ on $X_a$ is discrete. The groups $\Gamma$ and $\Gamma'$ are uniform lattices modulo finite kernels in $\Aut(X,T')$.
    
    Since $\Gamma$ and $\Gamma'$ command all of their finite subgroups, we may replace these groups by finite-index subgroups $\hat{\Gamma}\leq \Gamma$ and $\hat{\Gamma'}\leq \Gamma'$ so that the edge and vertex stabilizers of $\hat{\Gamma}$ and $\hat{\Gamma'}$ act freely on the corresponding edge and vertex spaces. In particular, $\hat{\Gamma}, \hat{\Gamma'}$ act freely on $\Aut(X,T')$. The group $Q_v = q_v(G_v)$ as in Notation~\ref{defn:ToS} is discrete, and hence contains $\Gamma_v$ as a finite-index subgroup. Thus, $Q_v$ commands all of its finite subgroups since $\Gamma_v$ does. Therefore, $\hat{\Gamma}$ and $\hat{\Gamma'}$ are weakly commensurable in $\Aut(X,T')$ by Theorem~\ref{thm:commoncover_BV}. 
    In particular,  $\hat{\Gamma}$ and $\hat{\Gamma'}$ are abstractly commensurable, hence so are  $\Gamma$ and $\Gamma'$.
\end{proof}

    By adding a stronger hypothesis to the group $\Gamma$, we obtain that $\Gamma$ is action rigid. 

\begin{defn} \label{def:persistcommand}
  A finitely generated group $\Gamma$ \emph{persistently commands its finite subgroups} if any group $\Gamma'$ virtually isomorphic to $\Gamma$ commands its finite subgroups. 
\end{defn}

\begin{example}
   Virtually polycyclic groups (e.g.\ finitely generated nilpotent groups) persistently  command their finite subgroups by \cite[Theorem 13.77]{DrutuKapovich18}.
\end{example}

    \begin{example} \label{ex:persistent3man}
        The fundamental group of a compact $3$-manifold persistently commands its finite subgroups. Indeed, if a finitely generated group $\Gamma$ is virtually isomorphic to the fundamental group of a compact 3-manifold $M$, then $\Gamma$ contains a finite-index subgroup $\Gamma'$ isomorphic to the fundamental group of a compact 3-manifold $N$ by~\cite[Theorem 1.2]{haissinskylecuire_QIrigid}. 
        The group $\pi_1(N)$ is residually finite by~\cite{hempel_RF3man} together with geometrization; see~\cite{AschenbrennerFriedlWilton15}. Hence, $\Gamma$ is residually finite, and therefore commands its finite subgroups. 
    \end{example}

\begin{example}\label{exmp:virt_cub}
    Any cubulated hyperbolic group persistently commands its finite subgroups.
    This can be deduced from the fact that all cubulated hyperbolic groups are residually finite, due to Agol \cite[Corollary 1.3]{Agol13}.
    Indeed, suppose $\G_1$ is a hyperbolic group that acts geometrically on a CAT(0) cube complex $X$, and let $\G_2$ be virtually isomorphic to $\G_1$.
    We claim that $\G_2$ has a finite-index subgroup that also acts on $X$ geometrically, hence $\G_2$ is residually finite and commands its finite subgroups.
    As $\G_1$ and $\G_2$ are virtually isomorphic, for $i=1,2$ there exist finite-index subgroups $\Gamma'_i\leq \Gamma_i$ and finite normal subgroups $F_i\vartriangleleft \Gamma'_i$ such that $\Gamma'_1/F_1\cong \Gamma'_2/F_2$.
    Since $\Gamma'_1$ is residually finite, there is a finite-index subgroup $\Gamma''_1\leq \Gamma'_1$ intersecting $F_1$ trivially. Thus there is a finite-index subgroup $\Gamma''_2\leq \Gamma'_2$ containing $F_2$ such that $\Gamma''_2/F_2\cong \Gamma''_1F_1/F_1\cong \Gamma''_1$. Since $\Gamma''_1$ acts geometrically on $X$, so does $\Gamma''_2$ as required.
\end{example}

    \begin{lem} \label{lem:commandfiniteinfended}
        If $\Gamma$ is a finitely generated infinite-ended group such that every one-ended vertex group in a Stallings--Dunwoody decomposition of $\Gamma$ commands its finite subgroups, then $\Gamma$ commands its finite subgroups. 
    \end{lem}
    \begin{proof}
    This can be shown using elementary arguments about finite covers of graphs of groups.
    The result also follows from \cite[Proposition 6.11(1)]{Shepherd23} and the fact that a group commands its finite subgroups if and only if its finite subgroups are discrete in the profinite topology.
    \end{proof}

    \begin{thm} \label{thm:actionrigidity}
        If $\Gamma$ is a finitely presented infinite-ended group so that all one-ended vertex groups in a Stallings--Dunwoody decomposition of $\Gamma$ are graphically discrete and persistently command their finite subgroups, then $\Gamma$ is action rigid.
    \end{thm}
        \begin{proof}
            Suppose that $\Gamma$ and a group $\Gamma'$ act geometrically on the same proper  metric space $Y$. Then, the group $\Gamma'$ is finitely presented and infinite-ended. We claim that $\Gamma'$ commands its finite subgroups. 
            
            The group $G = \Isom(Y)$ acts continuously and compactly on a tree $T$ with compact edge stabilizers and vertex stabilizers that are compactly generated and have at most one end by Theorem~\ref{thm:dunwoody_accessibility}. As above, since $G_v$ is open for all $v \in VT$, it contains $\Gamma_v$ and $\Gamma_v'$ as uniform lattices modulo finite kernels by Lemma~\ref{lem:open_lattice}. Since $\Gamma_v$ is graphically discrete and $G$ is compact-by-(totally disconnected) by Theorem~\ref{thm:houghton_ends}, the group $G_v$ is compact-by-discrete.
            Hence, $\Gamma_v$ and $\Gamma_v'$ are virtually isomorphic. Thus, $\Gamma_v'$ commands its finite subgroups, so $\Gamma'$ commands its finite subgroups by Lemma~\ref{lem:commandfiniteinfended}. Therefore, $\Gamma$ and $\Gamma'$ are abstractly commensurable by Theorem~\ref{thm:actionrigidA}.
        \end{proof}
 
    \begin{rem}
    In fact, the proof of Theorem \ref{thm:actionrigidity} says that if $\Gamma$ is as above and both $\Gamma$ and $\Gamma'$ act geometrically on the same proper  space, then they are abstractly commensurable.
    \end{rem}

    \begin{thm} \label{thm:3manactionrigid}
        The fundamental group of a (possibly reducible) closed $3$-manifold  is action rigid if and only if the manifold does not admit $\Hy^3$ or {\bf Sol} geometry. 
    \end{thm}
    \begin{proof}
        Let $M$ be a closed $3$-manifold; after passing to a two-sheeted cover if necessary, we may assume $M$ is oriented. Suppose first that $M$ is irreducible. If $M$ is non-geometric, then $\pi_1(M)$ is action rigid by Theorem~\ref{thm:nongeo3mangd} and Lemma~\ref{lem:triv_env_act_rig}.
        If $M$ admits $\Hy^3$ or {\bf Sol} geometry, then $M$ is not action rigid since there are infinitely many commensurability classes of manifolds with each of these geometries; see~\cite[Theorem A]{neumann_commensurability}. If $M$ admits one of the geometries $S^3$, $S^2 \times \R$, $\E^3$ or {\bf Nil}, then $\pi_1(M)$ is quasi-isometrically rigid, and hence action rigid; see \cite[Theorem 1.7]{frigerioQIrigid}. 
        Finally, if $M$ admits either $\Hy^2 \times \R$ or $\widetilde{\SL_2(\R)}$ geometry and $\Gamma$ is quasi-isometric to $M$, then $\Gamma$ is virtually isomorphic to the fundamental group of a closed geometric $3$-manifold that admits the geometry of either $\Hy^2 \times \R$ or $\widetilde{\SL_2(\R)}$~\cite{rieffelQI}. Action rigidity for these groups then follows from \cite[Corollary 1.2]{dastessera} and~\cite[Theorem A]{neumann_commensurability}.
        
        Now suppose that $M$ is reducible, so $\pi_1(M)$ splits over a finite subgroup and, therefore, has more than one end. Since $2$-ended groups are action rigid, suppose the fundamental group of $M$ is infinite-ended and each one-ended vertex group in a Stallings--Dunwoody decomposition of $\pi_1(M)$ is the fundamental group of an irreducible closed $3$-manifold. Then, each one-ended vertex group is graphically discrete by Theorems~\ref{thm:propc-3manifolds} and Theorem~\ref{thm:nongeo3mangd} and persistently commands its finite subgroups by Example~\ref{ex:persistent3man}.  
        Thus, $\pi_1(M)$ is action rigid by Theorem~\ref{thm:actionrigidity}.
    \end{proof}

    If $M$ is a $3$-manifold with toroidal boundary, then $\pi_1(M)$ may have a one-ended vertex group in a Stallings--Dunwoody decomposition that is isomorphic to $F_n \times \Z$, which is not graphically discrete. The proof above does not extend to this setting, which includes the case $\pi_1(M) \cong (F_2 \times \Z) * (F_2 \times \Z)$, and we ask the following. 
 
 \begin{question}
    Is the fundamental group of a $3$-manifold with toroidal boundary action rigid? 
 \end{question}

    \begin{example} \label{example_notRF}
        This example, which follows the argument of~\cite[Proposition 7.1]{starkwoodhouse2024action}, illustrates that the commanding finite subgroup hypothesis  is necessary in the action rigidity results given in this section. By work of Raghunathan~\cite[Main Theorem]{raghunathan_torsion} there exists a torsion-free uniform lattice $\Lambda \leq \Spin(2,n)$ and a finite extension $1 \rightarrow F \rightarrow \hat{\Lambda} \rightarrow \Lambda \rightarrow 1$ so that $\hat{\Lambda}$ is not virtually torsion-free. Thus, there exists an element $f \in F$ of order $r$ that is contained in every finite-index subgroup of $\hat{\Lambda}.$ Let 
        \[ \Delta = \Delta(r,5,5) = \la a,b,c\,|\, a^2=b^2=c^2=(ab)^r = (bc)^5 = (ac)^5 = 1 \ra\]
        be a hyperbolic triangle group. Then, the amalgamated free product $H = \hat{\Lambda} *_{\la f \ra = \la ab \ra} \Delta$ is not action rigid, but $\hat{\Lambda}$ and $\Delta$ are both graphically discrete by Theorem~\ref{thm:propc_examples}. Indeed, $\hat \Lambda$ acts geometrically on the symmetric space $X$ associated to $\SO(2,n)$ because  $\Spin(2,n)$ is a double cover of the identity component of $\SO(2,n)$. The proof of~\cite[Proposition 7.1]{starkwoodhouse2024action} can thus be applied by replacing vertex spaces isometric to $\Hy_{\mathbb{F}}^n$ in the model geometry with copies of the symmetric space $X$.
    \end{example}

\section{Quasi-isometric rigidity}\label{sec:model_geom}
In this section we outline a two-step method  for proving quasi-isometric rigidity theorems; see Proposition~\ref{prop:qi_framework}. We then apply this approach in Theorem \ref{thm:common_model_geom},  giving quasi-isometric rigidity for certain graphs of hyperbolic $n$-manifold groups.

\subsection{From action rigidity to quasi-isometric rigidity} \label{sec:actiontoQIrigid}

\begin{defn}\label{defn:uniform_QI}
	 A metric space $X$ has \emph{the uniform quasi-isometry property} if there exist constants $K\geq 1$ and $A\geq 0$ such that for every quasi-isometry $f:X\to X$, there is a $(K,A)$-quasi-isometry $f':X\to X$ such that $\sup_{x\in X}d(f(x),f(x'))<\infty$.
\end{defn}

We recall that a finitely generated group $\Gamma$ is \emph{quasi-isometrically rigid} if every finitely generated group quasi-isometric to $\Gamma$ is virtually isomorphic to $\Gamma$.

\begin{prop}\label{prop:qi_framework}
	Let $\Gamma$ be a finitely generated tame group. Suppose that:
	\begin{enumerate}
		\item $\Gamma$ has the uniform quasi-isometry property;
		\item $\Gamma$ is action rigid.
	\end{enumerate}
Then $\Gamma$ is quasi-isometrically rigid.
\end{prop}
\begin{proof}
	Let $G=\QI(\Gamma)$. Since $\Gamma$ is tame and has the uniform quasi-isometry property, there is a  cobounded quasi-action $\{f_g\}_{g\in G}$ of $G$ on $\Gamma$ such that $[f_g]=g$; see e.g.\ \cite[Lemma 3.4]{margolis2024model}. 
	It follows from work of Margolis that there is a proper quasi-geodesic metric space $X$ and a quasi-isometry $f:\Gamma\to X$ that quasi-conjugates the quasi-action of $G$ on~$\Gamma$ to an isometric action of $G$ on $X$ \cite[Proposition 4.5 and Theorem 4.20]{margolis2022discretisable}. Suppose $\Gamma'$ is a finitely generated group quasi-isometric to $\Gamma$. Since $\Gamma'$ admits a proper cobounded quasi-action on $\Gamma$, the map $f$ quasi-conjugates this quasi-action to a geometric action of $\G'$  on $X$.  Thus, every finitely generated group quasi-isometric to~$\Gamma$ acts geometrically on $X$. Since $\Gamma$ is action rigid, this implies $\Gamma$ is quasi-isometrically rigid. 
\end{proof}

\begin{rem}
	If $\Gamma$ is tame and has infinite outer automorphism group, then the uniform quasi-isometry property is never satisfied \cite[Claim 1.26]{whyte2010coarse}. In particular, infinite-ended groups do not have the uniform quasi-isometry property.
\end{rem}

\subsection{Action rigidity for graphs of groups}\label{sec:actionrigidityGoG}

The aim of this subsection is to prove the action rigidity result given in Theorem~\ref{thm:hyp_action_rigid_criterion} by applying Theorems~\ref{thm:get_treeofspaces} and~\ref{thm:commoncover_BV}. We begin with the relevant definitions. 

\begin{defn}
	If $\cG$ is a finite graph of finitely presented groups, an \emph{associated Bass--Serre tree of spaces} $(X,T)$ is the universal covering of a graph of spaces $(Z,\Lambda)$  associated to $\cG$ (in the sense of Definition \ref{def:gog_gosp}) in which all vertex and edge spaces of $(Z,\Lambda)$ are compact CW complexes.
\end{defn}

\begin{notation}
	Throughout the remainder of  this section, given a tree of spaces $(X,T)$, we identify an edge space $X_e$ with the subspace of the form $X_e\times \{0\}$ in the total space; see Definition \ref{defn:GoS}. In this way, $X_e=X_{\bar e}$ is a subspace of $X$ for each $e\in ET$.
\end{notation}

Work of Papasoglu and Mosher--Sageev--Whyte \cite{papasoglu2005quasi,moshersageevwhyte2011quasiactions} gives many examples of trees of spaces that are preserved by quasi-isometries in the following sense:
\begin{defn}\label{defn:tree_QI}
	A tree of spaces $(X,T)$ is \emph{preserved by quasi-isometries} if for every   $(K,A)$-quasi-isometry $f:X\to X$, there is a constant $B=B(X,K,A)$ such that the following hold:
	\begin{enumerate}
		\item if $w\in VT\cup ET$, there is some $w'\in VT\cup ET$ such that $d_\Haus(f(X_w),X_{w'})\leq B$;
		\item if $w'\in VT\cup ET$, there is some $w\in VT\cup ET$ such that $d_\Haus(f(X_w),X_{w'})\leq B$.
	\end{enumerate}
\end{defn}

\begin{notation}
	Suppose $\cG=\cG(\Lambda,\{\G_v\},\{\G_e\},\{\Theta_e\})$ is a graph of groups. If $v\in \Lambda$, the \emph{collection  of images of incident edge maps}  $\{\Theta_e(\G_e)\}_{e\in \lk(v)}$ is  assumed to be indexed by $\lk(v)$. In particular, distinct edges $e_1,e_2\in\lk(v)$ such that $\Theta_{e_1}(\G_{e_1})= \Theta_{e_2}(\G_{e_2})$ correspond to distinct elements of $\{\Theta_e(\G_e)\}_{e\in \lk(v)}$.
\end{notation}

We recall a graph of groups $\cG$ is \emph{minimal} if the action of $\pi_1(\cG)$ on the associated Bass--Serre tree has no proper invariant subtree.

\begin{thm}\label{thm:hyp_action_rigid_criterion}
	Suppose $\Gamma$ is the fundamental group of a minimal finite graph of groups $\cG$ satisfying the following properties:
	\begin{enumerate}
		\item Every vertex group of $\cG$  is graphically discrete, hyperbolic  and cubulated.
		\item For each vertex group of $\cG$, the collection  of images of incident edge maps is an almost malnormal family of infinite quasi-convex subgroups.
		\item The associated Bass--Serre  tree of spaces $(X,T)$ is preserved by quasi-isometries.
	\end{enumerate}
	Then $\Gamma$ is action rigid.
\end{thm}

\begin{rem}\label{rem:inf_index_image}
	Let $\cG=\cG(\Lambda,\{\G_v\},\{\G_e\},\{\Theta_e\})$ be a minimal graph of groups such that every    vertex group $\G_v$ is infinite and the collection  of images of incident edge maps is an almost malnormal family of $\G_v$. Then every edge monomorphism $\G_e\to \G_v$ has infinite-index image in $\G_v$. 
\end{rem}

We will use the following results in the proof of Theorem~\ref{thm:hyp_action_rigid_criterion}. Recall %
an action of a group $G$ on a tree $T$ is \emph{$k$-acylindrical} if any segment of length $k$ in $T$ has finite stabilizer. 

\begin{lem}\label{lem:acylindrical}
	Suppose $\Gamma$ is the fundamental group of a minimal finite graph of groups $\cG$  such that for each vertex group, the collection  of images of incident edge maps is an almost malnormal family.  Then the action of $\Gamma$ on the associated Bass--Serre tree is $2$-acylindrical. 
	
	Moreover, if vertex groups of $\cG$ are hyperbolic and every edge map has quasi-convex image, then $\Gamma$ is hyperbolic and all vertex and edge stabilizers are quasi-convex in $\Gamma$.
\end{lem}
\begin{proof}
	The fact that the action of $\Gamma$ is 2-acylindrical follows  from the condition that images of edge maps form an almost malnormal family of $G_v$. 	
	Hyperbolicity can be deduced from the Bestvina--Feighn combination theorem \cite{bestvinafeighn92combination} by adapting a proof of Kapovich, who considers a slightly more restrictive notion of acylindricity \cite{kap01combination}. Alternatively, the result follows verbatim from a much more general theorem of Martin--Osajda \cite{martinxreosajda2021combination}.
\end{proof}

We make extensive use of the next proposition to prove rigidity results in this section. Informally, it gives a criterion on a tree of space $(X,T)$ which ensures the combinatorial structure of $T$ can be completely recovered from the coarse geometry of its vertex and edge spaces.

\begin{prop}\label{prop:hd_graph}
	Let $\cG$ be a minimal finite graph of finitely generated groups. Suppose that for each vertex group,  the collection  of images of incident edge maps is an almost malnormal family of infinite subgroups. 
	Assume $T$ is the Bass--Serre tree associated to $\cG$, the corresponding tree of spaces is $(X,T)$. Then the following hold for all $w,w'\in VT \sqcup ET$:
	\begin{enumerate}
		\item If $d_\Haus(X_w,X_{w'})<\infty$, then $w=w'$.\label{item:hd_graph1}
		\item If $w\neq w'$, then   $X_{w'}\subseteq N_R(X_{w})$ for some $R$ if and only if $w$ is a vertex and $w'$ is an edge incident to $w$.\label{item:hd_graph2}
		\item For any distinct edges $e_1,e_2\in ET$, $N_R(X_{e_1})\cap N_R(X_{e_2})$ is bounded for all $R$.\label{item:hd_graph3}
	\end{enumerate}
\end{prop}

\begin{proof}
Let $\G=\pi_1(\cG)$ and let $\Gamma_w$ denote the stabilizer of $w$ for each $w\in VT\sqcup ET$. By the Milnor--Schwarz lemma, the action of $\G$ on $X$ induces a quasi-isometry $f:\G\to X$ such that $d_\Haus (f(\G_w),X_w)<\infty$ for each $w\in VT\sqcup ET$. It follows from \cite[Lemma 2.2 and Corollary 2.4]{moshersageevwhyte2011quasiactions} that:
\begin{enumerate}[label=(\alph*)]
	\item $X_w$ and $X_{w'}$ are at finite Hausdorff distance if and only if $\G_w$ and $\G_{w'}$ are commensurable;\label{item:finHD}
	\item $X_{w}\subseteq N_R(X_{w'})$ for some $R$ if and only if $\G_w$ is commensurable to a subgroup of~$\G_{w'}$.\label{item:coarse_contain}
	\item $N_R(X_{w})\cap N_R(X_{w})$ is bounded for all $R$ if and only if $\G_w\cap \G_{w'}$ is finite.\label{item:coarse_intersection}
\end{enumerate} 
Remark \ref{rem:inf_index_image}, Lemma \ref{lem:acylindrical} and the hypotheses on $\Gamma$ ensure that (\ref{item:hd_graph1})--(\ref{item:hd_graph3}) follow immediately from \ref{item:finHD}--\ref{item:coarse_intersection} respectively.
 \end{proof}
We now use Proposition \ref{prop:hd_graph} to show that under suitable hypotheses, quasi-isometries of $X$ induce automorphisms of $T$ in a functorial manner:
\begin{cor}\label{cor:tree_auts}
	Let $(X,T)$ be a tree of spaces  preserved by quasi-isometries and satisfying properties (\ref{item:hd_graph1}) and (\ref{item:hd_graph2}) in the conclusion of Proposition \ref{prop:hd_graph}. Then for every $(K,A)$-quasi-isometry $f:X\to X$, there is a constant $B=B(X,K,A)$ and a unique  tree automorphism $f_*:T\to T$ such that $d_\Haus(f(X_w),X_{f_*(w)})\leq B$ for all $w\in VT \cup ET$. 
	Moreover, the following are satisfied  for all quasi-isometries  $f,g:X\to X$:
	\begin{enumerate}
		\item If $f$ and $g$ are close, then $f_*=g_*$.
		\item $(g\circ f)_*=g_*\circ f_*$;
		\item $(\id_X)_*=\id_T$.
	\end{enumerate}
\end{cor}
\begin{proof}
	Let   $f:X\to X$ be a $(K,A)$-quasi-isometry. Since $(X,T)$ is preserved by quasi-isometries,  there is a constant $B=B(X,K,A)$ and a function $f_*:VT\sqcup ET\to VT\sqcup ET$ such that $d_\Haus(f(X_w),X_{f_*(w)})\leq B$ for all $w\in VT\sqcup ET$. By (\ref{item:hd_graph1}) of Proposition \ref{prop:hd_graph},  no two distinct vertex or edge spaces are at finite Hausdorff distance, hence $f_*$ is uniquely determined. Moreover, (\ref{item:hd_graph2}) of Proposition \ref{prop:hd_graph} ensures $f_*$ sends vertices to vertices, edges to edges, and preserves incidence of edges and vertices. Therefore, $f_*$ must be a tree morphism. Applying the same argument to a coarse inverse $\bar f$ of $f$, we deduce $f_*$ is invertible and hence an automorphism. The remaining properties follow easily from the definition of $f_*$ and the fact that distinct vertex or edge spaces are at infinite Hausdorff distance. 
\end{proof}

\begin{proof}[Proof of Theorem \ref{thm:hyp_action_rigid_criterion}]
    Suppose $\Gamma$ and $\Gamma'$ are uniform lattices modulo finite kernels in the same locally compact group $G$. As in Remark \ref{rem:2ndcountable}, we can assume without loss of generality that $G$ is second countable. Let $(X,T)$ be the associated Bass--Serre tree of spaces for $\Gamma$. We first exhibit a quasi-action of $G$ on the space $X$.  Agol's Theorem~\cite{Agol13} implies vertex groups of $\cG$ are virtually special. Thus, $\Gamma$ has, in the terminology of Wise, an almost malnormal quasiconvex hierarchy terminating in virtually special groups, so $\Gamma$ is virtually special~\cite[Theorem 11.2]{wise2021structure}. Hence, $\Gamma$ is virtually torsion-free. We can now replace $\Gamma$ by a finite-index torsion-free subgroup so that $\Gamma$ is a uniform lattice in $G$.  As in Section \ref{sec:qi}, the uniform lattice embedding $\Gamma\to G$ induces a quasi-action of $G$ of $\Gamma$. Since $\Gamma$ acts geometrically on $X$, the quasi-action of $G$ on $\Gamma$ can be quasi-conjugated to a quasi-action $\{f_g\}_{g\in G}$ of $G$ on $X$. 

    We now show that $G$ acts minimally, cocompactly, and continuously on the tree $T$. By Proposition \ref{prop:hd_graph} and Corollary \ref{cor:tree_auts}, there is a constant $B$ such that for every $g\in G$ there is a tree automorphism $g_*\in \Aut(T)$ such that $d_\Haus(f_g(X_w),X_{g_*(w)})\leq B$. By the functoriality properties of $g\mapsto g_*$ given in Corollary \ref{cor:tree_auts}, this defines an action of  $G$ on $T$ extending the action of $\Gamma$ on $T$. In particular, this action is cocompact and minimal.
    By Lemma \ref{lem:coarse_conv}, there exists a constant $C$ such that for every $x\in X$ and sequence $(g_i)\to 1$ in $G$, we have $d(f_{g_i}(x),x)\leq C$ for all $i$ sufficiently large.  Since distinct vertex or edge spaces are at infinite Hausdorff distance by Proposition \ref{prop:hd_graph}, it follows that if a sequence $(g_i)\to \id$ in~$G$, then the sequence $(g_i)_*$ converges to the identity pointwise in $T$. Hence, the action of $G$ on $T$ is continuous. 

    We claim each vertex group $G_v$ contains a compact open normal subgroup fixing  $\lk(v)$.  Let $v\in VT$. 
    Since $\Gamma$ is a uniform lattice in $G$ and $G_v$ is open, Lemma \ref{lem:open_lattice} implies $\Gamma_v$ is a uniform lattice in $G_v$. Since $\Gamma_v$ is graphically discrete, $G_v$ has a compact open normal subgroup $K_v\vartriangleleft G_v$. Now for each $g\in G_v$, we have $d_\Haus(f_g(X_v),X_{v})\leq B$, hence there is quasi-action $\{f^v_g\}_{g\in G_v}$ of $G_v$ on $X_v$ such that $\sup_{x\in X_v} d(f_g(x),f^v_k(x))\leq B$. This quasi-action is cobounded since $\Gamma_v\leq G_v$ acts cocompactly on $X_v$. By Lemma \ref{lem:cpct_normal_qaction},  there is a constant $D$ such that $\sup_{x\in X_v}d(x,f_k^v(x))\leq D$ for all $k\in K_v$.
    Since no two distinct edge spaces are at finite Hausdorff distance and $d_\Haus\bigl(f_k^v(\phi_e(X_e)),\phi_{k_*(e)}(X_{k_*(e)})\bigr)<\infty$ for all $k\in K_v$ and $e\in \lk(v)$, it follows that $K_v$ fixes $\lk(v)$ pointwise.

 Finally, for each $w\in VT \cup ET$, since $\Gamma_w$ is finitely presented and $G_w$ is open, it follows from Lemmas  \ref{lem:open_lattice} and \ref{lem:finitevscompact} that $G_w$ is compactly presented. Thus, Theorem \ref{thm:get_treeofspaces} yields a geometric action of $G$ on a tree of spaces $(X',T)$ satisfying the conclusions of Theorem \ref{thm:get_treeofspaces}. 
 
 To deduce that $\Gamma$ and $\Gamma'$ are abstractly commensurable, we show Theorem \ref{thm:commoncover_BV} can  be applied. Each $G_v$ is compact-by-discrete, so $\Gamma'_v$ is virtually isomorphic to $\Gamma_v$. Thus, $\Gamma'_v$ is virtually special by the argument in Example~\ref{exmp:virt_cub}, together with Agol's Theorem. Therefore, $\Gamma'$ is virtually torsion-free as in the argument at the beginning of this proof. Since $\Gamma$ and $\Gamma'$ are virtually torsion-free, we replace them with finite-index subgroups that act freely on $(X',T)$ and are uniform lattices (rather than uniform lattices modulo finite kernels) in~$G$.

Define $Q_v$ and $Q_v^e$ as in Notation \ref{defn:ToS} and Definition \ref{defn:ToS_loc_edges}. The group $Q_v$ is a  hyperbolic group containing $\Gamma_v$ as a finite-index subgroup. In particular, $Q_v$ is virtually special. For each $e\in \lk(v)$, the group $Q_v^{e}$ is a quasi-convex subgroup of $Q_v$  containing  $\Gamma_e$ as a finite-index subgroup. This implies, for any set of $G$-orbit representatives $\{e_i\}_{i=1}^k\subseteq\lk(v)$, the groups $\{Q_v^{e_i}\}_{i=1}^k$ form an almost malnormal family of quasi-convex subgroups in $Q_v$. Therefore, $Q_v$ commands $\{Q_v^{e_i}\}_{i=1}^k$  by Example~\ref{rem:special_command}. 	All the hypotheses of Theorem \ref{thm:commoncover_BV} are met, so $\Gamma$ and $\Gamma'$ are abstractly commensurable.
\end{proof}

    While our main application of Theorem~\ref{thm:hyp_action_rigid_criterion} appears in the next subsection, we include two corollaries. A {\it primitive} element of a group is an element that is not a proper power of another group element.

    \begin{cor}\label{cor:amalgam_3mfld}
        If $M$ and $N$ are closed hyperbolic $3$-manifolds, then $\pi_1(M) *_{\Z} \pi_1(N)$, where a generator of $\Z$ maps to primitive elements in $\pi_1(M)$ and $\pi_1(N)$, is action rigid. 
    \end{cor}

    Manning--Mj--Sageev~\cite{mj2024cubulating} proved many hyperbolic surface-by-free groups are cubulated, which gives examples to which the following corollary can be applied. 
    
    \begin{cor}\label{cor:surface-by-free}
        If $\Gamma$ and $\Gamma'$ are cubulated hyperbolic surface-by-free groups, then $\Gamma *_{\Z} \Gamma'$, where a generator of $\Z$ maps to primitive elements in $\Gamma$ and $\Gamma'$, is action rigid. 
    \end{cor}
        \begin{proof}[Proof of Corollaries \ref{cor:amalgam_3mfld} and \ref{cor:surface-by-free}]
    	We first note that for any torsion-free hyperbolic group $\Gamma$, any primitive infinite cyclic subgroup $H$ is quasi-convex and malnormal. Quasi-convexity of two-ended subgroups of hyperbolic groups was shown by Gromov \cite{gromov1987hyperbolic}.  To see $H$ is malnormal, we first note that $H$ is equal to the maximal two-ended subgroup of $\Gamma$ containing $H$. Thus $H$  coincides with the stabilizer of its limit set in $\partial \Gamma$, hence  is equal to its commensurator in $\Gamma$ (see for instance \cite[\S 7]{minasyan2005some}). This implies that $H$  is malnormal.
    	
    	We now prove the corollaries. In both cases, the ambient group $\Lambda$ splits as a graph of groups $\Gamma *_{\Z} \Gamma'$, where $\Gamma$ and $\Gamma'$ are torsion-free hyperbolic groups. In both cases, the vertex groups are cubulated; this follows from Remark \ref{rem:closed_mfld_cubulated} when vertex groups are 3-manifolds. It also follows from Theorem \ref{thm:propc_examples} that the vertex groups of $\Lambda$ are graphically discrete. It follows the previous paragraph that the incident edge group of each vertex group of $\Lambda$ is a  malnormal  infinite quasi-convex subgroup of the vertex group.  Lemma \ref{lem:acylindrical} thus implies $\Lambda$ is hyperbolic and the associated Bass--Serre tree $T$  is 2-acylindrical. 
    	
    	If a non-elementary torsion-free hyperbolic group splits over a finite or two-ended subgroup, then its outer automorphism group is infinite (see e.g. \cite{levitt2005automorphisms}). Since fundamental groups of closed hyperbolic 3-manifolds and hyperbolic surface-by-free groups have finite outer automorphism group \cite[Theorem 1.3]{farbmosher2002surfacebyfree}, it follows that vertex groups of $\Lambda$ cannot split over a finite or two-ended subgroup.
    	  	Therefore, $T$ has two-ended edge stabilizers and vertex stabilizers that do not split over any finite or two-ended subgroup.  As $T$ is 2-acylindrical, it is equal to the JSJ tree of cylinders of $\Lambda$ over two-ended subgroups, hence  coincides with Bowditch's JSJ  tree of a one-ended hyperbolic group~\cite{bowditch,guirardel2017jsj}. This tree is preserved by homeomorphisms of $\partial \Lambda$, so the associated Bass--Serre tree is preserved by quasi-isometries.  Theorem~\ref{thm:hyp_action_rigid_criterion} implies that $\Lambda$  is action rigid as required.
    \end{proof}

    Quasi-isometric rigidity in the above examples is open. 

   We remark that Hruska--Wise proved that a separable quasi-convex subgroup of a hyperbolic group is virtually almost malnormal, giving other instances where Theorem~\ref{thm:hyp_action_rigid_criterion} can be applied; see Corollary~\ref{cor:QI3manifold}.

\begin{prop}[{\cite[Theorem 9.3]{hruskawise2009packing}}]\label{prop:virtmal}
	Let $\Gamma$ be a hyperbolic group, and let $H\leq \Gamma$ be a quasi-convex separable subgroup. Then there exists a finite-index subgroup $\Gamma'\leq \Gamma$ containing $H$ such that $H$ is almost malnormal in $\Gamma'$.
\end{prop}

\subsection{Quasi-isometric rigidity for graphs of hyperbolic manifold groups}

    The aim of this section is to prove the following quasi-isometric rigidity theorem. If $\Gamma$ is a group of finite cohomological dimension $\cd(\Gamma)$, the \emph{cohomological codimension} of a subgroup $\Lambda\leq \Gamma$ is $\cd(\Gamma)-\cd(\Lambda)$.

\begin{thm}\label{thm:common_model_geom}
	Let $\Gamma$ be the fundamental group of a finite graph of groups $\cG$ satisfying the following properties:
	\begin{enumerate}
	\item Each vertex group of $\cG$ is cubulated and is the fundamental group of a closed hyperbolic $n$-manifold for some $n\geq 3$.
	\item For each vertex group of $\cG$, the collection  of images of incident edge maps is an almost malnormal family of  non-elementary quasi-convex subgroups with  cohomological codimension at least two.
	\end{enumerate}
Then any finitely generated group quasi-isometric to $\Gamma$ is abstractly commensurable to~$\Gamma$.
\end{thm}

\begin{rem}\label{rem:closed_mfld_cubulated}
	Work of Kahn--Markovic and Bergeron--Wise implies that fundamental groups of closed hyperbolic $3$-manifolds are cubulated~\cite{kahnmarkovic2012immersing, bergeronwise2012boundary}. A theorem of Bergeron--Haglund--Wise shows many additional hyperbolic $n$-manifolds are cubulated \cite{bergeronhaglundwise2011hyperplane}. 
\end{rem}

Before proving Theorem \ref{thm:common_model_geom}, we state a concrete special case.

\begin{cor}\label{cor:QI3manifold}
	Let $\Gamma_1$ and $\Gamma_2$ be fundamental groups of closed hyperbolic 3-manifolds, and let $F_1\leq \Gamma_1$ and  $F_2\leq \Gamma_2$ be isomorphic quasi-convex non-abelian free subgroups. There exist finite-index subgroups $\Gamma'_1\leq \Gamma_1$ and $\Gamma'_2\leq \Gamma_2$ containing $F_1$ and $F_2$ respectively,  such that any finitely generated group quasi-isometric to  $\Gamma=\Gamma'_1*_{F_1=F_2}\Gamma'_2$ is abstractly commensurable  to $\Gamma$.
\end{cor}
\begin{proof}
By Proposition \ref{prop:virtmal}, there are finite-index subgroups $\Gamma'_1\leq \Gamma_1$ and $\Gamma'_2\leq \Gamma_2$ containing $F_1$ and $F_2$ as malnormal subgroups. Let $\Gamma=\Gamma'_1*_{F_1=F_2}\Gamma'_2$. By Remark~\ref{rem:closed_mfld_cubulated}, $\Gamma_1$ and $\Gamma_2$ are cubulated. The rest of the hypotheses of Theorem~\ref{thm:common_model_geom} are satisfied, so the theorem can be applied.
\end{proof}

Fix $\Gamma$ and $\cG$ as in Theorem \ref{thm:common_model_geom} for the remainder of this section. Let $(X,T)$ be the Bass--Serre tree of spaces associated to $\cG$. We first prove $\Gamma$ is action rigid using Theorem~\ref{thm:hyp_action_rigid_criterion}, and then we will prove $\Gamma$ has the uniform quasi-isometry property. Quasi-isometric rigidity then follows from Proposition \ref{prop:qi_framework}, which allows us to  upgrade  action rigidity to quasi-isometric rigidity.

\subsubsection{Action Rigidity}

 To prove that $\Gamma$ is action rigid, we use Theorem~\ref{thm:hyp_action_rigid_criterion}. To do so,  it remains to prove that the tree of spaces $(X,T)$ is preserved by quasi-isometries. We note that the hypotheses of Theorem~\ref{thm:common_model_geom} ensure the conclusions of Proposition~\ref{prop:hd_graph} hold. 
The next proposition is a consequence of work of Mosher--Sageev--Whyte \cite{moshersageevwhyte2011quasiactions}. It will be used to prove both  action rigidity  and  the uniform quasi-isometry property.
\begin{prop}[{\cite[Theorem 1.6]{moshersageevwhyte2011quasiactions}}]\label{prop:induced_tree_aut}
	The tree of spaces $(X,T)$ is preserved by quasi-isometries. 
\end{prop}
\begin{proof}
	 We apply Theorem 1.6 \cite{moshersageevwhyte2011quasiactions}, which  says that if a graph of groups $\cG$ satisfies the following five conditions, then its associated Bass--Serre tree of spaces is preserved by quasi-isometries. 
	 \begin{enumerate}[label=(MSW\arabic*), leftmargin=*]
	 	\item $\cG$ is finite type, irreducible, and of finite depth.\label{item:msw1}
	 	\item  No depth zero raft of the Bass--Serre tree of $\cG$ is a line.\label{item:msw2}
	 	\item Every depth zero vertex group of $\cG$ is coarse PD.\label{item:msw3}
	 	\item For every one vertex, depth zero raft of the Bass--Serre tree, the crossing
	 	graph condition holds.\label{item:msw4}
	 	\item  Every vertex and edge group is coarse finite type.\label{item:msw5}
	 \end{enumerate}
	 Since some of these conditions are quite technical and will not be needed outside the current proof, we refer the reader to \cite{moshersageevwhyte2011quasiactions} for definitions of the above terms. We now show that a graph of groups $\cG$ satisfying the hypothesis of Theorem \ref{thm:common_model_geom} satisfies the preceding five conditions, which says $(X,T)$ is preserved by quasi-isometries.
	 
	Firstly, $\cG$ is a finite graph of groups with hyperbolic vertex and edge groups, so $\Gamma$ is finitely presented and $\cG$ has finite type. By Remark \ref{rem:inf_index_image}, every edge group is an infinite-index subgroup of its adjacent vertex groups. Therefore, $\cG$ is irreducible and all vertex groups have depth zero, hence  $\cG$ satisfies \ref{item:msw1}. All depth zero rafts consist of a single vertex, so \ref{item:msw2} is satisfied. Each vertex group of $\cG$ is the fundamental group of a closed aspherical manifold,  hence is coarse PD; thus \ref{item:msw3} is satisfied. 
	
	Since every vertex group of $\cG$ is the fundamental group of a closed aspherical manifold, it is  of type $FP$. Furthermore, as edge groups are hyperbolic, hence of type $FP_\infty$, and  have finite cohomological dimension (as they are subgroups of their adjacent vertex groups), they also have type $FP$. In particular, all vertex and edge groups have cohomological dimension equal to their coarse dimension as defined in \cite{moshersageevwhyte2011quasiactions}. As a result, \cite[Lemma 3.8]{moshersageevwhyte2011quasiactions} ensures that no edge group $X_e$ coarsely separates an adjacent vertex group, so each vertex group has empty crossing graph and \ref{item:msw4} vacuously holds. As noted above, since  every vertex and edge group of $\cG$ has type $FP$, it is of coarse finite type, hence \ref{item:msw5} holds.
\end{proof}

\begin{prop}\label{prop:action_rigid_graphofgroups}
	If  $\Gamma$ is as in Theorem \ref{thm:common_model_geom}, then $\Gamma$ is action rigid.
\end{prop}
\begin{proof}
    It follows from the definition of $\Gamma$, Theorem \ref{thm:propc_examples}, and Proposition \ref{prop:induced_tree_aut} that all the hypotheses of Theorem~\ref{thm:hyp_action_rigid_criterion} are satisfied, hence $\Gamma$ is action rigid.
\end{proof}

\begin{rem}
	Thus far, we have made no use of the fact that edge groups are non-elementary. In particular, if $\Gamma$ satisfies all the hypotheses of Theorem \ref{thm:common_model_geom} except the non-elementary edge group condition, then $\Gamma$ is action rigid.
\end{rem}

\subsubsection{The uniform quasi-isometry property}
In light of Proposition \ref{prop:qi_framework},  Example \ref{exmp:tame_morse} and Proposition \ref{prop:action_rigid_graphofgroups}, in order to show $\Gamma$ is quasi-isometrically rigid it remains to show $\Gamma$ satisfies the uniform quasi-isometry property. 
We first recall relevant results of Biswas~\cite{biswas12flows}; see \cite{schwartz1997symmetric,biswasmj2012pattern,mj_patternrigidity} for related results.

\begin{defn}
Let $\Gamma$ be a finitely generated group acting geometrically on a proper geodesic hyperbolic space $X$. A \emph{$\Gamma$-symmetric pattern} $\cJ$ in $X$ is a $\Gamma$-invariant  collection of quasi-convex subsets of $X$ such that: \begin{enumerate}
	\item  for every $J\in \cJ$, the stabilizer $\stab_\Gamma(J)$ acts cocompactly on $J$ and  is an infinite, infinite-index subgroup of $\Gamma$;
	\item $\cJ$ contains only finitely many $\Gamma$-orbits.
\end{enumerate}
A \emph{symmetric pattern} in $X$ is a $\Gamma$-symmetric pattern for some finitely generated group~$\Gamma$ acting geometrically on  $X$. We denote $X$ together with a symmetric pattern $\cJ$ by $(X,\cJ)$. 
A \emph{pattern-preserving quasi-isometry} $f:(X,\cJ)\to (X',\cJ')$  is a quasi-isometry $f:X\rightarrow X'$ such that there exists a constant $A \geq 0$ so that:
\begin{enumerate}
	\item for all $J_1\in \cJ$, there exists a $J_2\in \cJ'$ such that $d_{\Haus}(f(J_1),J_2)\leq A$;
	\item  for all $J_2\in \cJ'$, there exists a $J_1\in \cJ$ such that $d_{\Haus}(f(J_1),J_2)\leq A$.
\end{enumerate}
Let $\QI(X,\cJ)$ be the subgroup of $\QI(X)$ consisting of equivalence classes of pattern-preserving quasi-isometries of $\cJ$.
\end{defn}

\begin{rem}\label{rem:pattern_cosets}
Suppose a group $\Gamma$ acts geometrically on a proper hyperbolic space $X$ and   $H_1,\dots, H_n\leq \Gamma$ is a collection of infinite, infinite-index quasi-convex subgroups. If  $x\in X$,  then $\cJ\coloneqq \{gH_i x\mid 1\leq i\leq n, g\in \Gamma\}$ is a $\Gamma$-symmetric pattern. Moreover,  up to modifying each element of $\cJ$ by uniformly bounded  finite Hausdorff distance, every $\Gamma$-symmetric pattern arises in this way. 
\end{rem}

Biswas proved that for $n\geq 3$, pattern-preserving quasi-isometries of $\Hy^n$  are bounded distance from isometries. More precisely, the theorem below follows from the discussion on \cite[p590]{biswas12flows}. Note that Biswas states it in the more general framework of uniformly proper maps.

\begin{thm}[{\cite{biswas12flows}}] \label{thm:rel_rigidity}
	Suppose $\cJ$ and $\cJ'$ are symmetric patterns in $\bH^n$ for some $n\geq 3$. If $f:(\bH^n,\cJ)\to (\bH^n,\cJ')$ is a  pattern-preserving quasi-isometry, then there is a hyperbolic isometry $f':\bH^n\rightarrow\bH^n$ such that $\sup_{x\in \bH^n}d(f(x),f'(x))<\infty$.
\end{thm}

The following is a consequence of Theorem~\ref{thm:rel_rigidity}, which is used implicitly in the proof of \cite[Corollary 1.3]{biswas12flows}; see also \cite[Corollary 1.5]{mj_patternrigidity}.

\begin{cor}\label{cor:patternrigid}
    Suppose a group $\Gamma$ acts geometrically on $\Hy^n$ for some $n \geq 3$ and $\cJ$ is a $\Gamma$-symmetric pattern in $\bH^n$. Then, $\QI(\bH^n,\cJ)$ can be identified with a discrete subgroup of $\Isom(\bH^n)$, and under this identification $\Gamma$ is a finite-index subgroup of $\QI(\bH^n,\cJ)$.
\end{cor}
 For completeness, we provide a proof of Corollary \ref{cor:patternrigid} from Theorem \ref{thm:rel_rigidity}.
\begin{proof}
	Theorem \ref{thm:rel_rigidity} ensures that $\QI(\bH^n,\cJ)\leq\Isom(\bH^n)$. It is easy to see that $\QI(\bH^n,\cJ)$ is closed, hence a Lie group by Cartan's Theorem. Since $\QI(\bH^n,\cJ)$ coarsely permutes the discrete pattern $\cJ$, it is totally disconnected; this follows from \cite[Proposition 3.5]{mj_patternrigidity}, noting that the compact-open topology on $\Isom(\bH^n)$ coincides with the topology on $\Isom(\bH^n)$ as a subgroup of $\Homeo(\partial \bH^n)$ equipped with the topology of uniform convergence. Lemma \ref{lem:totdisc_to_lie} applied to the inclusion map implies $\QI(\bH^n,\cJ)$ is discrete. As $\Gamma$ is a uniform lattice contained in the discrete group $\QI(\bH^n,\cJ)$, it is a finite-index subgroup of $\QI(\bH^n,\cJ)$.
\end{proof}

    Symmetric patterns arise naturally in the setting of Theorem~\ref{thm:common_model_geom}. Indeed, let $\Gamma$ be as in Theorem~\ref{thm:common_model_geom} and let $(X,T)$ be the associated Bass--Serre tree of spaces. If $v$ is a  vertex of $T$, then the vertex group $\Gamma_v$ is a uniform lattice in $\bH^{n_v}$ for some $n_v\geq 3$. Since $\Gamma_v$ acts geometrically on both $X_v$ and $\bH^{n_v}$, the Milnor--Schwarz Lemma implies there exists a coarsely $\Gamma_v$-equivariant quasi-isometry $h_v:X_v\rightarrow \bH^{n_v}$.  Let $\overline{h_v}$ be a coarse inverse to $h_v$. Since there are finitely many $\Gamma$-orbits of vertices of $T$, we can without loss of generality assume $h_v$ and $\overline{h_v}$ have uniform quasi-isometry constants as $v$ varies. For each $v\in VT$, define a $\Gamma_v$-symmetric pattern $\cJ_v\subseteq X_v$ by \[\cJ_v\coloneqq \{\phi_e(X_e)\mid e\in \lk(v)\}.\]  Since the map $h_v$ is coarsely $\Gamma_v$-equivariant, Remark \ref{rem:pattern_cosets} yields a $\Gamma_v$-symmetric pattern $\hat \cJ_v\subseteq \bH^{n_v}$ such that $h_v$ is a pattern-preserving quasi-isometry $(X_v,\cJ_v)\to (\bH^{n_v},\hat \cJ_v)$. 

    The following lemma will be used to show quasi-isometries are pattern preserving.

\begin{lem}\label{lem:unif_peripherals}
For $K\geq 1$ and $A\geq 0$, there is a constant $B=B(K,A,X)$ such that the following holds. If $v,w\in VT$ and $f:X_v\to X_w$ is a $(K,A)$-quasi-isometry such that  $d_\Haus (f(J),J')<\infty$ for some $J\in \cJ_v$ and $J'\in \cJ_w$, then $d_\Haus (f(J),J')\leq B$.
\end{lem}
\begin{proof}
	For each $v\in VT$ and  $J\in \cJ_v$, since $\stab_{\Gamma_v}(J)$ is quasi-convex and acts cocompactly on  $J$, it follows from a result of Swenson that $J$ has finite Hausdorff distance from $\WCH(\Lambda J)$, where $\Lambda J\subseteq \partial X_v$ is the limit set of $J$ and $\WCH(\Lambda J)$ is the weak convex hull of $J$ \cite{swenson}. Since there are only finitely many $\Gamma$ orbits of such $J$, there exists a constant $A_1 \geq 0$ such that $d_\Haus(\WCH(\Lambda J),J)\leq A_1$ for all $J\in \cJ_v$.
	
	Now suppose $f:X_v\to X_w$ is a $(K,A)$-quasi-isometry such that $d_\Haus (f(J),J')<\infty$ for some $J\in \cJ_v$ and $J'\in \cJ_w$. Then $f$ induces a homeomorphism $\partial f:\partial X_v\to \partial X_w$ such that $\partial f(\Lambda J)=\Lambda J'$. The Extended Morse Lemma for $\delta$-hyperbolic spaces implies there is a constant $A_2=A_2(K,A,X)$ such that $d_\Haus(f(\WCH(\Lambda J)),\WCH(\Lambda J'))\leq A_2$, hence $d_\Haus(f(J),J')\leq KA_1+A+A_1+A_2\eqqcolon B$ as required.
\end{proof}

    We can now use Theorem \ref{thm:rel_rigidity} to prove the second  hypotheses of Proposition \ref{prop:qi_framework}.

\begin{prop}\label{prop:unif_qis}
	If $\Gamma$ is as in Theorem \ref{thm:common_model_geom}, then $\Gamma$ has the uniform quasi-isometry property.
\end{prop}
\begin{proof}
	Since $\Gamma$ is quasi-isometric to the space $X$, where $(X,T)$ is an associated Bass--Serre tree of spaces, it is sufficient to show that $X$ has the uniform quasi-isometry property.  Namely, we show there exist constants $K$ and $A$ such that every quasi-isometry $f:X\to X$ is close to a $(K,A)$-quasi-isometry. Note that the optimal quasi-isometry constants of $f$ will typically be much larger than $K$ and $A$. 
	
	Our strategy will be to use the induced map $f_*:T\to T$ to construct a $(K,A)$-quasi-isometry $\hat f:X\to X$ close to $f$. We first use Theorem \ref{thm:rel_rigidity} to show that for each vertex $v\in VT$, there are pattern-preserving  quasi-isometries $\hat f_v:X_v\to X_{f_*(v)}$ whose quasi-isometry constants are independent of $f$. Then, we glue these quasi-isometries of vertex spaces together, using tameness of edge spaces, to yield a quasi-isometry $\hat f$.

	We fix a quasi-isometry $f:X\to X$. By Proposition~\ref{prop:induced_tree_aut}, the tree of spaces $(X,T)$ is preserved by quasi-isometries.  Let $f_*: T \rightarrow T$ be a corresponding tree automorphism as in Corollary \ref{cor:tree_auts}. Then, for every $w\in VT \cup ET$, there is a quasi-isometry $f_w:X_w\to X_{f_*(w)}$ such that $\sup_{x\in X_w} d(f(x),f_w(x))<\infty$ by~\cite[Lemma 2.1]{farbmosher2000abelianbycyclic}. 
	Proposition~\ref{prop:induced_tree_aut} implies $f$ coarsely preserves edge spaces of $X$, so for every vertex $v\in VT$, $f_v$ sends  elements of $\cJ_v$ to $\cJ_{f_*(v)}$  up to finite Hausdorff distance. Applying a similar argument to a coarse inverse $\bar f$ of $f$ yields a coarse inverse $\bar f_v$ of $f_v$ that  sends  elements of $\cJ_{f_*(v)}$ to $\cJ_{v}$  up to finite Hausdorff distance.  Thus, by Lemma~\ref{lem:unif_peripherals}, the map $f_v$ is a pattern-preserving quasi-isometry  $(X_v,\cJ_v)\to (X_{f_*(v)},\cJ_{f_*(v)})$.
	
	Therefore, for each $v \in VT$, the composition $h_{f_*(v)}f_v\overline{h_v}$ is a pattern-preserving quasi-isometry $(\bH^{n},\hat \cJ_v)\to (\bH^{n},\hat \cJ_{f_*(v)})$, where $n=n_v=n_{f_*(v)}$ and $h_v:X_v \rightarrow \Hy^{n_v}$ is defined above.  Theorem~\ref{thm:rel_rigidity} implies $h_{f_*(v)}f_v\overline{h_v}$ is close to an isometry $\phi_v:\bH^n\to \bH^n$. Thus, there exist constants $K_1 \geq 1$ and $A_1 \geq 0$, which are independent of~$v$, such that the map $\hat f_v\coloneqq \overline{h_{f_*(v)}}\phi_vh_v:X_v\to X_{f_*(v)}$ is a  $(K_1,A_1)$-quasi-isometry and $\sup_{x\in X_v} d(f_v(x),\hat f_v(x))<\infty$. Crucially, $K_1$ and $A_1$ depend only on $X$ and not  on the quasi-isometry constants of the original quasi-isometry~$f$.
	
	It follows from Lemma \ref{lem:unif_peripherals} that there is a constant $B$, depending only on $X$, such that for each $v\in VT$ and  $e\in \lk(v)$, \[d_\Haus(\hat f_v (\phi_e(X_e)),\phi_{f_*(e)}(X_{f_*(e)}))\leq B,\]
	where $\phi_e:X_e \rightarrow X_v$ is the edge map. Another application of \cite[Lemma 2.1]{farbmosher2000abelianbycyclic} implies that for every $v\in VT$ and  $e\in \lk(v)$, there exists a $(K_2,A_2)$-quasi-isometry $\hat f_{v,e}:X_e\to X_{f_*(e)}$ such that \[\sup_{x\in X_e} d\bigl( \phi_{f_*(e)}(\hat f_{v,e}(x)),\hat f_v(\phi_e(x))\bigr)\leq B',\]
	where $K_2$, $A_2$ and $B'$ depend only on $X$.
	Since $\sup_{x\in X_v}d(\hat f_v(x),f(x))<\infty$, it follows that $\sup_{x\in X_e}d(\hat f_{v,e}(x),f(x))<\infty$.
	
	Now, for each $e\in ET$, we have two $(K_2,A_2)$-quasi-isometries $\hat f_{\iota(e),e},\hat f_{\tau(e),e}:X_e\to X_{f_*(e)}$ such that $\sup_{x\in X_e}d(\hat f_{\iota(e),e}(x),\hat f_{\tau(e),e}(x))<\infty$. Since $X_e$ is a non-elementary hyperbolic space, it is tame; see Example \ref{exmp:tame_morse}. Hence there is a constant $A_3=A_3(K_2,A_2,X)$ such that $\sup_{x\in X_e}d(\hat f_{\iota(e),e}(x),\hat f_{\tau(e),e}(x))\leq A_3$. Again, it is crucial to note all these constants depend only on $X$ and not  $f$.   It now follows from  \cite[Proposition 2.14]{cashenmartin2017quasi} that the quasi-isometries $\{\hat f_v\}_{v\in VT}$ can be glued together with controlled error to yield a $(K_4,A_4)$-quasi-isometry $\hat f:X\to X$, where $K_4$ and $A_4$ depend only on $X$, such that $\sup_{x\in X} d(\hat f(x),f(x))<\infty$. 
\end{proof}

\begin{proof}[Proof of Theorem \ref{thm:common_model_geom}]
	Since $\G$ is hyperbolic, it is tame by Example \ref{exmp:tame_morse}. The result is thus  a combination of Propositions \ref{prop:qi_framework},  \ref{prop:action_rigid_graphofgroups} and \ref{prop:unif_qis}. 
\end{proof}

Provided an analogue of Theorem \ref{thm:rel_rigidity} holds, the proof of Theorem \ref{thm:common_model_geom} goes through virtually unchanged for uniform lattices in more general rank one symmetric spaces.  For quaternionic hyperbolic spaces and the Cayley hyperbolic plane, the required statement follows from  Pansu's much stronger quasi-isometric rigidity theorem \cite{pansu1989metriques}, where the conclusion of Theorem \ref{thm:rel_rigidity} holds without the pattern. For complex hyperbolic spaces,  Biswas--Mj prove an analogue of Theorem \ref{thm:rel_rigidity} provided the pattern $\cJ$ arises as the translates of odd-dimensional Poincar\'e duality groups \cite{biswasmj2012pattern}.  
Applying these  results and  following the  proof of  Theorem \ref{thm:common_model_geom} gives the following:
\begin{thm}\label{thm:common_model_geom_rak_one}
	Let $\Gamma$ be the fundamental group of a finite graph of groups $\cG$ satisfying the following properties:
	\begin{enumerate}
		\item Each vertex group  of $\cG$ is cubulated and is a uniform lattice in a rank one symmetric space  other than $\bH^2$.
		\item For each vertex group  of $\cG$, the collection  of images of incident edge maps is an almost malnormal family of  non-elementary quasi-convex subgroups with virtual cohomological codimension at least two.
		Moreover, for each  vertex group that is a lattice in complex hyperbolic space, all its incident edge groups are odd-dimensional Poincar\'e duality groups.
	\end{enumerate}
	Then any finitely generated group quasi-isometric to $\Gamma$ is abstractly commensurable to~$\Gamma$.
\end{thm}

\bibliographystyle{alpha}
\bibliography{mybib}

\end{document}